\documentclass[11pt,a4paper,reqno]{amsart} 

\usepackage{anysize}
\marginsize{3.5cm}{3.5cm}{2.2cm}{2.2cm}

\usepackage{hyperref} 
\hypersetup{
    colorlinks=true,       
    linkcolor=red,          
    citecolor=blue,        
    filecolor=magenta,      
    urlcolor=cyan           
}

\usepackage[english]{babel}
\usepackage{amsmath}
\usepackage{amsfonts,amssymb}
\usepackage{enumerate}
\usepackage{mathrsfs}
\usepackage{amsthm}
\usepackage{tikz}

\theoremstyle{plain}
\newtheorem{theo}{Theorem}[section]
\newtheorem*{theo*}{Theorem}
\newtheorem{prop}[theo]{Proposition}
\newtheorem{lemm}[theo]{Lemma}
\newtheorem{coro}[theo]{Corollary}
\newtheorem{assu}[theo]{Assumption}
\newtheorem{defi}[theo]{Definition}
\newtheorem{claim}[theo]{Claim}
\theoremstyle{definition}
\newtheorem{rema}[theo]{Remark}
\newtheorem{nota}[theo]{Notation}

\DeclareMathOperator{\RE}{Re}

\DeclareMathOperator{\IM}{Im}
\DeclareMathOperator{\Op}{Op}

\DeclareMathOperator{\supp}{supp}

\DeclareMathOperator{\sign}{sign}

\DeclareSymbolFont{pletters}{OT1}{cmr}{m}{sl}
\DeclareMathSymbol{s}{\mathalpha}{pletters}{`s}
\def\ba{\begin{align}}
\def\bad{\begin{aligned}}
\def\be{\begin{equation}}
\def\ea{\end{align}}
\def\ead{\end{aligned}}
\def\ee{\end{equation}}
\def\e{\eqref}

\def\blA{\bigl\lVert}
\def\bo{C^0}
\def\brA{\bigr\rVert}

\def\widecontrol#1{\wide \Theta_{#1}}

\def\controltroisT{\Theta_{M,T}}

\def\defn{\mathrel{:=}}
\def\eC#1{W^{{#1},\infty}}
\def\eps{\varepsilon}
\def\Fl#1#2{\mathcal{L}({#1},{#2})}
\def\la{\left\vert}
\def\lA{\left\Vert}
\def\le{\leq}
\def\les{\lesssim}
\def\lk{\ell}
\def\Lk{L}
\def\Lr{\mathcal{L}}
\def\meanH{H_0}
\def\mez{\frac{1}{2}}
\def\mR{\mathcal{R}}
\def\Nr{\mathcal{N}}
\def\opk{\mathcal{K}}
\def\ra{\right\vert}

\def\rA{\right\Vert}
\def\Run{R_1}
\def\Rdeux{R_2}
\def\Rtrois{R_3}
\def\Rquatre{R_4}
\def\Rcinq{R_5}
\def\tdm{\frac{3}{2}}
\def\tq{\frac{3}{4}}

\def\V{W}
\def\xC{\mathbb{C}}
\def\xN{\mathbb{N}}
\def\xR{\mathbb{R}}

\def\xT{\mathbb{T}}
\def\xZ{\mathbb{Z}}
\def\wide{\widetilde}

\def\Dx{\la D_x\ra}
\def\px{\partial_x}
\def\RBony{\mathcal{R}}
\def\Pext{P_{\rm ext}}

\def\indicator#1{\mathbf{1}_{{#1}}}
\def\uo{\indicator{\omega}}

\renewcommand{\a}{\alpha}
\renewcommand{\b}{\beta}
\newcommand{\g}{\gamma}
\newcommand{\lm}{\lambda}
\newcommand{\p}{\pi}
\newcommand{\ph}{\varphi}
\newcommand{\s}{\sigma}
\renewcommand{\t}{\tau}

\newcommand{\pa}{\partial}
\newcommand{\T}{\mathbb T}
\newcommand{\R}{\mathbb R}

\newcommand{\mN}{\mathcal{N}}
\newcommand{\mT}{\mathcal{T}}
\newcommand{\mL}{\mathcal{L}}
\newcommand{\mQ}{\mathcal{Q}}
\def\sqrkappa{{}}
\def\kappamuet{{}}
\def\kappanonmuet{\kappa}

\def\Norm#1#2{{C^0([0,{#1}];{#2})}}
\def\NormP#1#2{{C^0([0,{#1}];{#2})}}

\numberwithin{equation}{section}

\title{Control of water waves}
\author{T. Alazard, P. Baldi, D. Han-Kwan}

\def\notina[#1]#2{\begingroup\def\thefootnote{\fnsymbol{footnote}}\footnote[#1]{#2}\endgroup}

\begin{document}

\begin{abstract}
We prove local exact controllability in arbitrary short time of the two-dimensional incompressible Euler equation with free surface, 
in the case with surface tension. This proves that one can generate 
arbitrary small amplitude periodic gravity-capillary water waves by blowing on a localized portion 
of the free surface of a liquid.
\end{abstract}

\maketitle


\vspace{-5mm}

\tableofcontents

\vspace{-7mm}

\section{Introduction}
Water waves are disturbances of the free surface of a liquid.%
\notina[0]{T.A. was partly supported by the grant ``ANA\'E'' ANR-13-BS01-0010-03.
This research was carried out in the frame of Programme STAR, financially supported by UniNA and Compagnia di San Paolo; it was partially supported by the European Research Council under FP7, 
and PRIN 2012 ``Variational and perturbative aspects of nonlinear differential problems''.} 
They are, in general, produced by the immersion of a solid body,
the oscillation of a solid portion of the boundary
or by impulsive pressures applied on the free surface.
The question we address in this paper is the following:
which waves can be generated from the rest position
by a localized pressure distribution applied on the free surface.
This question is strictly related to the generation of waves
in a pneumatic wave maker (see \cite[\S21]{WehausenLaitone}, \cite{CFGK}).
Our main result asserts that, in arbitrarily small time,
one can generate any small amplitude, two-dimensional,
gravity-capillary water waves.
This is a result from control theory.
More precisely, this article is devoted to the study
of the local exact controllability of the incompressible
Euler equation with free surface.


There are many known control results for linear or nonlinear equations (see the book of Coron~\cite{Coron}), 
including equations describing water waves in some asymptotic regimes, 
like Benjamin-Ono (\cite{LinaresRosier,LLR}), 
KdV (\cite{Rosier,LRZ}) or nonlinear Schr\"odinger equation (\cite{DehmanLebeau}). 
In this paper, instead, we consider the full model, that is the incompressible Euler equation with free surface. 
Two key properties of this equation are that it is quasi-linear (instead of semi-linear as Benjamin-Ono, KdV or NLS) and secondly 
it is not a partial differential equation but instead a pseudo-differential equation, 
involving the Dirichlet-Neumann operator which is 
nonlocal and also depends nonlinearly on the unknown.  
As we explain later in this introduction, 
this requires to introduce new tools to prove the controllability. 

To our knowledge, this is the first control result for a quasi-linear 
wave equation relying on propagation of energy. 
In particular, using dispersive properties of gravity-capillary water waves
(namely the infinite speed of propagation), 
we prove that, for any control domain, 
one can control the equation in arbitrarily small time 
intervals. 

\subsection{Main result}
We consider the dynamics of an incompressible fluid moving under the force 
of gravitation and surface tension. At time~$t$, the fluid domain  
$\Omega(t)$ has a rigid bottom and a free surface described by the equation~$y=\eta(t,x)$, so that
$$
\Omega(t)=\left\{\, (x,y)\in \xR^2\,;\, -b<y<\eta(t,x)\,\right\},
$$
for some positive constant $b$ (our result also holds in infinite depth, for $b=\infty$). 
The Eulerian velocity field $v$ is assumed to be irrotational. 
It follows that $v=\nabla_{x,y} \phi$ 
for some time-dependent potential $\phi$ satisfying 
\begin{equation}\label{intro:1}
\Delta_{x,y}\phi=0,\quad 
\partial_{t} \phi +\mez \la \nabla_{x,y}\phi\ra^2 + P +g y = 0,\quad \partial_y\phi\arrowvert_{y=-b}=0,
\end{equation}
where~$g>0$ is the gravity acceleration,
$P$ is the pressure, $\nabla_{x,y}=(\partial_x,\partial_y)$ and $\Delta_{x,y}=\px^2+\partial_y^2$. 
The water waves equations are given by two boundary conditions on the free surface:  
firstly
$$
\partial_{t} \eta = \sqrt{1+(\px \eta)^2}\, \partial_n \phi \arrowvert_{y=\eta}
$$
where~$\partial_n$ is the outward normal derivative, so 
$\sqrt{1+(\px \eta)^2}\, \partial_n \phi =\partial_y\phi-(\px\eta)\px\phi$. 
Secondly, the balance of forces across the free surface reads
$$
P\arrowvert_{y=\eta}=\kappanonmuet H(\eta)+\Pext(t,x)
$$
where $\kappanonmuet$ is a positive constant, $\Pext$ is an external source term 
and $H(\eta)$ is the 
curvature:
\begin{equation*}
H(\eta) := \partial_x \bigg( \frac{\partial_x\eta}{\sqrt{1+(\partial_x\eta)^2}} \bigg)
= \frac{\px^2\eta}{(1 + (\px\eta)^2)^{3/2}} \cdot
\end{equation*}

Following Zakharov~\cite{Zakharov1968} 
and Craig and Sulem~\cite{CrSu}, it is equivalent to work 
with the trace of~$\phi$ at the free boundary
$$
\psi(t,x)=\phi(t,x,\eta(t,x)),
$$
and introduce the Dirichlet-Neumann operator~$G(\eta)$ 
that relates~$\psi$ to the normal derivative 
$\partial_n\phi$ of the potential by 
$$
(G(\eta) \psi)  (t,x)=\sqrt{1+(\partial_x\eta)^2}\,
\partial _n \phi\arrowvert_{y=\eta(t,x)}.
$$
Hereafter the surface tension coefficient $\kappanonmuet$ is taken to be~$1$. Then~$(\eta,\psi)$ solves (see~\cite{CrSu}) the system
\begin{equation}\label{WW1}
\left\{
\begin{aligned}
&\partial_t \eta = G(\eta)\psi,\\
&\partial_t\psi +g\eta+  \frac{1}{2} (\px\psi)^2  
-\frac{1}{2} \frac{\bigl(G(\eta)\psi + (\px\eta)(\px\psi)\bigr)^2}{1+ (\px\eta)^2}
= \kappamuet H(\eta)+\Pext.
\end{aligned}
\right.
\end{equation}
This system is augmented with initial data
\be\label{WW2}
\eta\arrowvert_{t=0}=\eta_{in},\quad \psi\arrowvert_{t=0}=\psi_{in}.
\ee
We consider the case when $\eta$ and $\psi$ are $2\pi$-periodic in the space variable $x$ and we set $\xT\defn \xR/(2\pi \xZ)$. Recall that the mean value of $\eta$ is conserved in time and can be taken to be $0$ without loss of generality. 
We thus introduce the Sobolev spaces $\meanH^\sigma(\xT)$ of functions with mean value $0$. 
Our main result asserts that, given 
any control domain $\omega$ and any arbitrary control time $T>0$, the equation \e{WW1} is controllable in time $T$ for small 
enough data.

\begin{theo}\label{T1}
Let $T>0$ and consider a non-empty open subset $\omega\subset \xT$. 
There exist $\sigma$ large enough and a positive constant 
$M_0$ small enough such that, for any two pairs of functions $(\eta_{in},\psi_{in})$,  
$(\eta_{final},\psi_{final})$ in $\meanH^{\sigma+\mez}(\xT)\times H^\sigma(\xT)$ satisfying 
$$
\lA \eta_{in}\rA_{H^{\sigma+\mez}}+\lA \psi_{in}\rA_{H^\sigma}<M_0,\quad 
\lA \eta_{final}\rA_{H^{\sigma+\mez}}+\lA \psi_{final}\rA_{H^\sigma}<M_0,
$$
there exists $P_{ext}$ in $C^0([0,T];H^{\sigma}(\xT))$, supported in 
$[0,T]\times \omega$, that is
$$
\supp P_{ext}(t,\cdot)\subset \omega, \quad \forall t\in [0,T], 
$$ 
such that the Cauchy problem \e{WW1}-\e{WW2} has a unique solution
$$
(\eta,\psi)\in C^0([0,T];\meanH^{\sigma+\mez}(\xT)\times H^{\sigma}(\xT)), 
$$
and the solution $(\eta,\psi)$ satisfies 
$(\eta\arrowvert_{t=T},\psi\arrowvert_{t=T})=(\eta_{final},\psi_{final})$.
\end{theo}

\begin{rema}
$i)$ This result holds for any $T>0$ and not only for $T$ large enough. Compared to the Cauchy problem, 
for the control problem it is more difficult to work on short time intervals than on large time intervals. 

$ii)$ This result holds also in the infinite depth case (it suffices to replace $\tanh(b\la \xi\ra)$ by $1$ in the proof). In finite depth, the non-cavitation assumption $\eta(t,x)>-b$ holds automatically for small enough solutions. 
\end{rema}

\subsection{Strategy of the proof}

We conclude this introduction by explaining the strategy of the proof and the 
difficulties one has to cope with. 

\subsubsection*{Remarks about the linearized equation}
We use in an essential way the fact that the water waves equation is a dispersive equation. 
This is used to obtain a control result which holds on arbitrarily small time intervals.
To explain this as well as to introduce the control problem, we begin with the analysis of the linearized 
equation around 
the null solution. 
Recall that $G(0)$ is the Fourier multiplier $\la D_x\ra \tanh(b\la D_x\ra)$. 
Removing quadratic and higher order terms in the equation, System~\e{WW1} becomes
$$
\left\{
\begin{aligned}
&\partial_t \eta = G(0)\psi,\\
&\partial_t\psi +g\eta -\kappamuet \px^2 \eta=\Pext.
\end{aligned}
\right.
$$
Introduce the Fourier multiplier (of order $3/2$)
$$
\Lk\defn  \big( (g-\kappamuet\px^2) G(0)\big)^{\frac{1}{2}}.
$$
The operator $G(0)^{-1}$ is well-defined on periodic functions with mean value zero. Then $u=\psi-i \Lk G(0)^{-1}\eta$ satisfies the dispersive equation
$$
\partial_t u+i\Lk u=P_{ext}.
$$
To our knowledge, the only control result for this linear equation is due to Reid who 
proved in \cite{RReid} a control result with a distributed control. He proved that one can steer any initial data to zero in finite time 
using a control of the form $P_{ext}(t,x)=g(x)U(t)$ ($g$ is given and $U$ 
is unknown). 
His proof is based on 
the characterization of Riesz basis and a variant of 
Ingham's inequality (see the inequality \e{i10} stated at the end of this introduction). 
In this paper we are interested in localized control, satisfying $P_{ext}(t,x)=\uo P_{ext}(t,x)$ where $\omega\subset \xT$ is a given open subset. 
However, using the same Ingham's inequality \e{i10} and 
the HUM method, one obtains a variant of Reid's control result where the control is localized.

\begin{prop}\label{PI}
For any initial data $u_{in}$ in $L^2(\xT)$, any $T>0$ 
and any non empty open domain $\omega\subset \xT$, 
there exists a source term $f\in \bo([0,T];L^2(\xT))$ such that the unique 
solution $u$ to
\begin{equation}\label{i12}
\partial_t u +iL u=\uo f\quad ;\quad u\arrowvert_{t=0}=u_{in},
\end{equation}
satisfies $u\arrowvert_{t=T}=0$ (here $\uo$ is the indicator function of $\omega$).
\end{prop}
\begin{rema}\label{R:1}
$i)$ (\emph{Null-controllability for reversible systems}).
By considering the backward equation, the same results holds when one exchanges 
$u\arrowvert_{t=0}$ and $u\arrowvert_{t=T}$. Now by using both results, one deduces 
that for any functions $u_{in},u_{final}$ in $L^2(\xT)$, there exists $f$ such that the unique solution to \e{i12} 
satisfies $u\arrowvert_{t=T}=u_{final}$. This is why one can assume without loss of generality that the final state $u_{final}$ is $0$.

$ii)$ (\emph{Real-valued control}).
This result is not satisfactory for the water waves problem since 
the control $f$ given by Proposition~\ref{PI}Ê
could be complex-valued. To obtain a real-valued 
control requires an extra argument.
\end{rema}

\subsubsection*{Step 1: Reduction to a dispersive equation}

The proof of Theorem~\ref{T1} relies on different tools and different previous results. 
Firstly, Theorem~\ref{T1} is related to the study of the Cauchy problem. 
The literature on the subject goes back to 
the pioneering works of Nalimov~\cite{Nalimov}, 
Yosihara~\cite{Yosihara} and Craig~\cite{Craig1985}. 
There are many results 
and we quote only some of them starting with the well-posedness of the Cauchy problem 
without smallness assumption, which was first proved by 
Wu \cite{WuInvent} 
and Beyer--G\"unther \cite{BG} for the case with surface tension. For some recent 
results about gravity-capillary waves, we refer to Iguchi~\cite{Iguchi},  Germain--Masmoudi--Shatah~\cite{GMS3}, Mesognon~\cite{Mesognon}, 
Ifrim--Tataru~\cite{Ifrim-Tataru}, 
Ionescu--Pusateri~\cite{IP-20141,IP-20142} 
and Ming--Rousset--Tvzetkov \cite{MRT}.

Our study is based on the analysis of the 
Eulerian formulation of the water waves equations by means of microlocal 
analysis. In this direction it is influenced by Lannes~\cite{LannesJAMS} 
as well as~\cite{AM,ABZ1}. 
More precisely, we use a paradifferential approach 
in order to paralinearize the water waves equations and then to symmetrize the obtained equations. 
We refer the reader to the appendix for the definition of paradifferential operators $T_a$.

It is proved in \cite{ABZ1} 
that one can reduce the water waves equations to a single dispersive wave 
equation that is similar to the linearized equation. Namely, it is proved in this reference that there are 
symbols $p=p(t,x,\xi)$ and $q=q(t,x,\xi)$ with $p$ of order $0$ in $\xi$ and $q$ of order $1/2$, such that $u=T_p \psi+i T_q \eta$ satisfies an equation of the form
$$
P(u)u=P_{ext}\quad\text{with}\quad P(u)\defn 
\partial_t +T_{V(u)}\partial_x  +i 
\Lk^\mez \big(T_{c(u)} \Lk^\mez \cdot\big),
$$
where $\Lk^\mez=  \big( (g-\kappamuet\px^2) G(0)\big)^{\frac{1}{4}}$, $T_{V(u)}$ and $T_{c(u)}$ are paraproducts. Here $V,c$ depend on the unknown $u$ 
with $V(0)=0$ and $c(0)=1$, and hence 
$P(0)=\partial_t+iL$ is the linearized operator around the null solution. 
We have oversimplified the result (neglecting remainder terms and simplifying the dependence of $V,c$ on $u$) and 
we refer to Proposition~\ref{T24}Ê for the full statement.

We complement the analysis of \cite{ABZ1} in two directions. Firstly, 
using elementary arguments (Neumann series and the implicit function theorem), 
we prove that one can invert the mapping $(\eta,\psi)\mapsto u$. Secondly, we prove that, up to 
modifying the sub-principal symbols of $p$ and $q$, one can further require that
\be\label{i18}
\int_\xT \IM u(t,x)\, dt=0.
\ee

\subsubsection*{Step 2: Quasi-linear scheme}
Since the water waves system \e{WW1} is quasi-linear, one cannot 
deduce the controllability of the nonlinear equation from the one of $P(0)$. Instead of using a fixed point argument, 
we use a quasi-linear scheme 
and seek $P_{ext}$ as 
the limit of {\em real-valued} functions $P_n$ 
determined by means of approximate control problems. To guarantee that 
$P_{ext}$ will be real-valued we seek $P_n$ as the real part of some function. 
To insure that $\supp P_n\subset \omega$ we seek $P_n$ under the form
$$
P_n=\chi_\omega \RE f_n.
$$
Hereafter, we fix $\omega$, 
a non-empty open subset of $\xT$, and 
a $C^\infty$ cut-off function $\chi_\omega$, supported on $\omega$, 
such that $\chi_\omega(x)=1$ for all $x$ in some open interval $\omega_1 \subset \omega$.

The approximate control problems are defined by induction as follows:  
we choose $f_{n+1}$ by requiring that the unique solution $u_{n+1}$ of the Cauchy problem 
$$
P(u_n)u_{n+1} = \chi_\omega \RE f_{n+1},\quad u_{n+1}\arrowvert_{t=0}=u_{in}
$$
satisfies $u(T)=u_{final}$. Our goal is to prove that 
\begin{itemize}
\item this scheme is well-defined 
(that is one has to prove a controllability result for 
$P(u_n)$);
\item the sequences $(f_n)$ and $(u_n)$ are bounded in $C^0([0,T];H^\sigma(\xT))$;
\item the series $\sum (f_{n+1}-f_n)$ and $\sum (u_{n+1}-u_n)$ converge in $C^0([0,T];H^{\sigma-\tdm}(\xT))$. 
\end{itemize}
It follows that $(f_n)$ and $(u_n)$ are Cauchy sequences 
in $C^0([0,T];H^{\sigma-\tdm}(\xT))$ (and in fact, by interpolation, in $C^0([0,T];H^{\sigma'}(\xT))$ for any $\sigma'<\sigma$). 

To use the quasi-linear scheme, we need to study a sequence of linear approximate control problems. 
The key point is to study the control problem for the linear operator 
$P(\underline{u})$ for some given function 
$\underline{u}$. Our goal is to prove the following result.

\begin{prop}\label{PP10-b-intro}
Let $T>0$. There exists $s_0$ 
such that, if $\lA \underline{u}\rA_{C^0([0,T];H^{s_0})}$ is small enough, depending on $T$, then the following properties hold. 

$i)$ (Controllability) 
For all $\sigma\ge s_0$ and all
$$
u_{in},u_{final}\in \tilde H^{\sigma}(\xT)\defn \left\{ 
w\in H^\sigma(\xT)\,;\, \IM \int_\xT w(x)\, dx=0\right\},
$$
there exists $f$ satisfying $\lA f\rA_{\bo([0,T];H^{\sigma})}
\le K(T)(\lA u_{in}\rA_{H^{\sigma}}+\lA u_{final}\rA_{H^{\sigma}})$ 
such that the unique solution $u$ to
$$
P(\underline{u})u= \chi_\omega \RE f\quad ;~u_{\arrowvert t=0}=u_{in},
$$
satisfies $u(T)=u_{final}$. 

$ii)$ (Stability) Consider another state $\underline{u}'$ with 
$\lA \underline{u}'\rA_{C^0([0,T];H^{s_0})}$ small enough and denote by $f'$ the control associated to 
$\underline{u}'$. Then 
$$
\lA f-f'\rA_{\bo([0,T];H^{\sigma-\tdm})}\le K'(T)(\lA u_{in}\rA_{H^{\sigma}}+\lA u_{final}\rA_{H^{\sigma}})\lA \underline{u}-\underline{u}'\rA_{C^0([0,T];H^{s_0})}.
$$
\end{prop}
\begin{rema}
$i)$ We oversimplified the assumptions and refer the reader to Section~\ref{section9} 
for the full statement. 

$ii)$ Notice that the smallness assumption on $\underline{u}$ involves only 
some $H^{s_0}$-norm, while the result holds for all 
initial data in $H^\sigma$ with $\sigma\ge s_0$.  This is possible 
because we consider a paradifferential equation. 
This plays a key role in the analysis to overcome losses of derivatives with respect to the coefficients.
\end{rema}

\subsubsection*{Step 3: Reduction to a regularized problem} 
We next reduce the analysis by proving that it is sufficient 
\begin{itemize}
\item to consider a {\em classical} equation instead of a {\em paradifferential} equation;
\item to prove a $L^2$-result instead of a Sobolev-result. 
\end{itemize}
This is obtained by commuting $P(\underline{u})$ with some well-chosen elliptic operator $\Lambda_{h,s}$ of order $s$ with
$$
s=\sigma-\tdm
$$
and depending on a small parameter $h$ (the reason to introduce $h$ is explained below). 
In particular $\Lambda_{h,s}$ is chosen so that the operator 
$$
\wide{P}(\underline{u})\defn \Lambda_{h,s} P(\underline{u}) \Lambda_{h,s}^{-1}
$$
satisfies 
\be\label{i22}
\wide{P}(\underline{u})
= P(\underline{u}) + R(\underline{u})
\ee
where $R(\underline{u})$ is a remainder term of order $0$. 
For instance, if $s=3m$ with $m\in\xN$, set 
$$
\Lambda_{h,s}=I+h^{s}\mathcal{L}^{\frac{2s}{3}}\quad\text{where }\Lr\defn L^\mez\big(T_c L^\mez \cdot\big).
$$
With this choice one has $[\Lambda_{h,s},\mathcal{L}]=0$ so \e{i22} holds with 
$R(\underline{u})= [\Lambda_{h,s},T_{V(\underline{u})}]\Lambda_{h,s}^{-1}$. 
It follows from symbolic calculus that $\lA R(\underline{u})\rA_{\Lr(L^2)}\les \lA V\rA_{\eC{1}}$ uniformly in $h$. 

Moreover, since $V(\underline{u})$ and $c(\underline{u})$ are continuous in time with values in 
$H^{s_0}(\xT)$ with $s_0$ large, one can replace paraproducts by usual products, up to remainder terms in $C^0([0,T];\Lr(L^2))$. We have
$$
\wide{P}(\underline{u})= \partial_t +V(\underline{u}) \partial_x  +i 
\Lk^\mez \big(c(\underline{u})  \Lk^\mez \cdot\big)+
R_2(\underline{u})
$$
where
$$
R_2(\underline{u})\defn R(\underline{u}) + 
\big( T_{V(\underline{u})} - V(\underline{u}) \big) \partial_x  
+ i \Lk^\mez \big(\big(T_{c(\underline{u}) }-c(\underline{u}) \big) \Lk^\mez \cdot\big).
$$
The remainder $R_2(\underline{u})$ belongs to $C^0([0,T];\Lr(L^2))$ uniformly in $h$. On the other hand, 
\be\label{i21}
\blA [\Lambda_{h,s},\chi_\omega]\Lambda_{h,s}^{-1}\brA_{\mathcal{L}(L^2)}=O(h),
\ee
which is the reason to introduce the parameter $h$. 
The key point is that one can reduce the proof of Proposition~\ref{PP10-b-intro} to the proof of the following result.

\begin{prop}\label{PP11-intro}
Let $T>0$. There exists $s_0$ 
such that, if $\lA \underline{u}\rA_{C^0([0,T];H^{s_0})}$ is small enough, then the following properties hold. 

$i)$ (Controllability) For all $v_{in}\in L^2(\xT)$  
there exists $f$ with $\lA f\rA_{\bo([0,T];L^2)}\le K(T)\lA v_{in}\rA_{L^2}$ 
such that the unique solution $v$ to $\wide{P}(\underline{u})v=\chi_\omega \RE f,\quad v_{\arrowvert t=0}=v_{in}$ 
is such that $v(T)$ is an imaginary constant:
$$
\exists b\in \xR \, / ~ \forall x\in \xT, \quad v(T,x)=ib.
$$

$ii)$ (Regularity) Moreover $\lA f\rA_{\bo([0,T];H^\tdm)}\le K(T)\lA v_{in}\rA_{H^\tdm}$.

$iii)$ (Stability) Consider another state $\underline{u}'$ with 
$\lA \underline{u}'\rA_{C^0([0,T];H^{s_0})}$ small enough and denote by $f'$ the control associated to 
$\underline{u}'$. Then 
$$
\lA f-f'\rA_{\bo([0,T];L^{2})}\le K'(T) \lA v_{in}\rA_{H^{\tdm}}\lA \underline{u}-\underline{u}'\rA_{C^0([0,T];H^{s_0})}.
$$
\end{prop}

Let us explain how to deduce Proposition~\ref{PP10-b-intro} from the latter proposition. 
Consider $u_{in},u_{final}$ in $\tilde H^\sigma(\xT)$ and seek $f\in C^0([0,T];H^\sigma(\xT))$ such that 
$$
P(\underline{u})u=\chi_\omega \RE f,\quad u(0)=u_{in} \quad \Longrightarrow \quad u(T)=u_{final}.
$$
As explained in Remark~\ref{R:1} it is sufficient to consider the case where 
$u_{final}=0$. Now, to deduce this result from Proposition~\ref{PP11-intro}, the main difficulty is 
that the conjugation with $\Lambda_{h,s}$ 
introduces a nonlocal term: indeed, $\Lambda_{h,s}^{-1}(\chi_{\omega}f)$ is not compactly supported in general. This is a possible source of difficulty since we seek a localized control term. We overcome this problem by considering the control problem for $\wide{P}(\underline{u})$ 
associated to some well-chosen initial data $v_{in}$. Proposition~\ref{PP11-intro} asserts that for all $v_{in}\in H^\tdm(\xT)$ 
there is $\widetilde{f}\in C^0([0,T];H^\tdm(\xT))$ such that 
$$
\wide{P}(\underline{u})v_1=\chi_\omega \RE  \widetilde{f},\quad v_1 |_{t=0}=v_{in} 
~\Longrightarrow v_1(T,x)=ib,~ b\in\xR.
$$
Define $\opk v_{in}=v_2(0)$ where 
$v_2$ is the solution to 
$$
\wide{P}(\underline{u})v_2=\big[ \Lambda_{h,s},\chi_\omega\big] \Lambda_{h,s}^{-1}
\RE \widetilde{f},\quad v_2\arrowvert_{t=T}=0.
$$
Using \e{i21}Ê
one can prove that the $\Lr(H^{\tdm})$-norm of $\opk$ is $O(h)$ and hence 
$I+\opk$ is invertible for $h$ small. So, $v_{in}$ can be so chosen that
$v_{in}+\opk v_{in}=\Lambda_{h,s}u_{in}$. 
Then, 
setting $f\defn \Lambda_{h,s}^{-1}\widetilde{f}$ and 
$u\defn \Lambda_{h,s}^{-1}(v_1+v_2)$, one checks that
$$
P(\underline{u})u=\chi_\omega \RE f,\quad u(0)=u_{in},\quad  u(T,x)=ib,~b\in\xR.
$$
It remains to prove that $u(T)$ is not only an imaginary constant, but it is $0$. 
This follows from the property \e{i18}. Indeed, $P$ can be 
so defined that if $P(\underline{u})u$ is a real-valued function, then $\frac{d}{dt}\int_\xT \IM u(t,x)\, dx=0$. Since $\int_\xT \IM u(0,x)\, dx=0$ by assumption, one deduces 
that $\int_\xT \IM u(T,x)\, dx=0$ and hence $u(T)=0$.

\subsubsection*{Step 4: Reduction to a constant coefficient equation} 

The controllability of $\wide{P}(\underline{u})$ 
will be deduced from the classical HUM method. 
A key step in the HUM method consists in proving that 
some bilinear mapping is coercive. 
To determine the appropriate 
bilinear mapping, we follow an idea introduced in \cite{AB} 
and conjugate $\wide{P}(\underline{u})$ to 
a constant coefficient operator modulo a remainder term of order $0$. 

To do so, we use a change of variables and a pseudo-differential change of unknowns to find 
an operator $M(\underline{u})$ such that 
$$
M(\underline{u})\wide{P}(\underline{u})M(\underline{u})^{-1}
=\partial_t+iL+\mathcal{R}(\underline{u}),
$$
where $\lA \mathcal{R}(\underline{u})\rA_{\Lr(L^2)}\les \lA \underline{u}\rA_{H^{s_0}}$ 
(and hence $\mathcal{R}(\underline{u})$ is a {\em small}Ê
perturbation of order $0$). 

To find $M(\underline{u})$, we begin by considering three changes of variables of the form
\[ 
(1 + \pa_x \kappa(t,x))^{\frac12} \, h(t,x + \kappa(t,x)),\quad 
h(a(t),x),\quad h(t,x-b(t)),
\] 
to replace $\wide{P}(\underline{u})$ with 
\be\label{defiQu}
Q(\underline{u}) = \pa_t + W \pa_x + i L + R_3,
\ee
where $W = W(t,x)$ satisfies $\int_{\xT} W(t,x)\, dx=0$, 
$\| W \|_{C^0([0,T]; H^{s_0-d})} \les \lA \underline{u}\rA_{\bo([0,T];H^{s_0})}$
where $d>0$ is a universal constant, 
and $\Rtrois$ is of order zero. This is not trivial since 
the equation is nonlocal and also because 
this exhibits a cancellation of a term of order $1/2$. Indeed, 
in general the conjugation of $\Lk^\mez \big(c(\underline{u})  \Lk^\mez \cdot\big)$ and a change of variables generates also a term of order $3/2-1$. 
This term disappears here since we consider transformations 
which preserve the $L^2(dx)$ scalar product. 

We next  seek an operator $A$ such that 
$i\big[ A,\Dx^\tdm\big]+W\px A$ is a zero order operator. This leads to consider a pseudo-differential operator $A=\Op(a)$ for some symbol 
$a=a(x,\xi)$ in the H\"ormander class $S^0_{\rho,\rho}$ with $\rho=\mez$, namely 
$a=\exp(i|\xi|^\mez \beta(t,x))$ for some function $\beta$ depending on $W$ 
(see Proposition~\ref{P:39} for a complete statement that also includes a zero order amplitude). 
Here we follow \cite{AB}. 
To keep the paper self-contained (and since some modifications are needed), we recall the strategy of the proof 
in Section~\ref{S:22}.

Concerning the latter transformation, let us compare the equation $P(u)u=0$ 
with the Benjamin-Ono equation:
\be\label{BO}
\partial_t w +w\partial_x w+\mathcal{H}\px^2 w=0,
\ee
where $\mathcal{H}$ is the Hilbert transform. 
The control problem for this equation has been studied through 
elaborate techniques (see for instance the recent paper \cite{LLR}) 
that are specific to this equation and cannot be applied to the 
water waves equations\footnote{This can be seen at the level of the Cauchy problem: 
for the Euler equation with free surface, the well-posedness of the Cauchy problem in the energy space is entirely open.}. 
On the opposite, let us discuss one difference which appears when applying to \e{BO} 
the strategy previously described. Given a function $W=W(t,x)$ with zero mean in $x$, 
let us seek an operator $B$ such that the leading order term in 
$\big[ B,\mathcal{H}\px^2\big]+W\px B$ vanishes. 
This requires (see \cite{Baldi}) to introduce a classical  pseudo-differential operator $B=\Op(b)$ with $b\in S^0_{1,0}$. 
Then the key difference between the two cases could be explained as follows: 
For $r$ large enough, 
\begin{itemize}
\item the mapping $W\mapsto B$ is Lipschitz from $H^r$ into $\mathcal{L}(L^2)$;
\item the mapping $W\mapsto A$ is only continuous from $H^r$ into $\mathcal{L}(L^2)$ (indeed, if $\lA W\rA_{H^r}=O(\delta)$ then we merely have 
$\lA A-I\rA_{\mathcal{L}(L^2;H^{-\mez})}=O(\delta)$). 
\end{itemize}
This is another reason for which 
one cannot use a fixed point argument based on a contraction estimate to deduce the existence 
of the control.

\subsubsection*{Step 5: Observability} 

Then, we establish an observability inequality. 
That is, 
we prove in Proposition~\ref{P45} that there exists $\eps > 0$ such that 
for any initial data $v_0$ whose mean value 
$\langle v_0\rangle = \frac{1}{2\pi} \int_{\xT}v_0(x)\, dx$ satisfies 
\be\label{i24}
\la \RE \langle v_0 \rangle \ra \ge \frac12 \la \langle v_0\rangle \ra
- \eps \lA v_0\rA_{L^2},
\ee
the solution $v$ of 
$$
\partial_t v+iL v=0,\quad v(0)=v_0 
$$
satisfies 
\be\label{i25}
\int_0^T\int_{\omega}  \la \RE (Av)(t,x)\ra^2 \, dxdt\ge K \int_{\xT} \la v_0(x)\ra^2 \, dx.
\ee
To prove this inequality with the real-part in the left-hand side allows to prove the existence of a real-valued control function; 
a similar property is proved for systems of wave equations by Burq and Lebeau in \cite{BurqLebeau}. 

The observability inequality is deduced using a variant of Ingham's inequality (see Section~\ref{S:4}). Recall that 
Ingham's inequality is an inequality for the $L^2$-norm of a sum of oscillatory functions which generalizes Parseval's inequality (it applies 
to pseudo-periodic functions and not only to periodic functions; 
see for instance \cite{KomornikLoreti}). 
For example, one such result asserts that 
for any $T>0$ there exist two positive constants $C_1=C_1(T)$ and $C_2=C_2(T)$ such that
\be\label{i10}
C_1\sum_{n\in\xZ}|w_n|^2\le \int_{0}^T 
\la \sum_{n\in\xZ} w_n e^{i n|n|^\mez t}\ra^2\, dt\le 
C_2\sum_{n\in\xZ}|w_n|^2
\ee
for all sequences $(w_n)$ in $\ell^2(\xC)$.
The fact that this result holds for any $T>0$ 
(and not only for $T$ large enough) is a consequence of a general result due to Kahane on lacunary series (see \cite{Kahane}).

Note that, since the original problem is quasi-linear, 
we are forced to prove an Ingham type inequality for 
sums of oscillatory functions whose phases differ from the phase of the linearized equation. 
For our purposes, we need to consider phases that do not depend linearly on $t$, of the form
$$
\sign(n) \big[\ell(n)^\tdm t+\beta(t,x)|n|^\mez\big],
\qquad \ell(n)\defn \left( (g+\la n\ra^2)\la n\ra \tanh(b\la n\ra)\right)^\mez,
$$
where $x$ plays the role of a parameter. 
Though it is a sub-principal term, to take into account the perturbation 
$\beta(t,x)|n|^\mez$ requires some care 
since $e^{i\beta(t,x)|n|^{1/2}}-1$ is not small. 
In particular we need to prove upper bounds for expressions in which we allow some amplitude depending on time (and whose derivatives 
in time of order $k$ can grow as $|n|^{k/2}$). 

\subsubsection*{Step 6: HUM method}

Inverting $A$, we deduce from \e{i25} an observability result for 
the adjoint operator $Q(\underline{u})^*$ ($Q(\underline{u})$ is as given by 
\e{defiQu}). Then the controllability will be deduced from the classical HUM method (we refer to 
Section~\ref{P:S7} for a version that makes it possible to consider a real-valued control). 
The idea is that 
the observability property implies that some bilinear form is coercive and hence the existence of the control follows from the Riesz's theorem and a duality argument. A possible difficulty is 
that the control $P_{ext}$ is acting only on the equation for $\psi$. 
To explain this, consider the case where $(\eta_{final},\psi_{final})=(0,0)$. 
Since the HUM method is based on orthogonality arguments, the fact that the control is not acting on 
both equations means for our problem that the final state is orthogonal to a co-dimension $1$ space. 
The fact that this final state can be chosen to be $0$ 
will be obtained by choosing this co-dimension $1$ space in an appropriate way, introducing 
an auxiliary function $M=M(x)$ which is chosen later on. 

Consider any real function $M=M(x)$ with $M-1$ small enough, and introduce
$$
L^2_{M}\defn \left\{\varphi\in L^2(\xT;\xC)\,;\, \IM \int_\xT M(x)\varphi(x)\, dx =0\right\}.
$$
Notice that $L^2_M$ is an $\xR$-Hilbert space. 
Also, for any 
$v_0 \in L^2_M$, the condition \e{i24} holds. 
Then, using a variant of the HUM method in this space, one deduces that for all $v_{in}\in L^2$ 
(not necessarily in $L^2_M$) 
there is $f\in C^0([0,T];L^2)$ such that, if
$$
Q(\underline{u})w=\pa_t w+ W \pa_x w+ i L w+ R_3 w=\chi_\omega \RE f, \quad w(0)=w_{in},
$$
then
$$
w(T,x)=ibM(x)
$$
for some constant $b\in\xR$. Now 
$$
Q(\underline{u})=\Phi(\underline{u})^{-1} \wide{P}(\underline{u})\Phi(\underline{u}),
$$
where $\Phi(\underline{u})$ is the composition of 
the transformations in \e{i21}. 
Since $\Phi(\underline{u})$ and $\Phi(\underline{u})^{-1}$ are 
local operators, 
one easily deduce a controllability result for $\wide{P}(\underline{u})$ 
from the one proved for $Q(\underline{u})$. 
Now, choosing 
$M=\Phi(\underline{u}(T,\cdot))(1)$ where $1$ is the constant function $1$, we deduce from $w(T,x)=ibM(x)$ that 
$u(T,x)$ is an imaginary constant, as asserted in 
statement $i)$ of Proposition~\ref{PP10-b-intro}. 
Concerning $M$, notice 
that $M\neq 1$ because of the factor $(1 + \pa_x \kappa(t,x))^{\frac12}$ multiplying $h(t,x + \kappa(t,x))$ in \e{i21}.

\subsubsection*{Step 7: Convergence of the scheme}

Let us discuss the proof of 
the convergence of the sequence of approximate controls $(f_n)$ to 
the desired control $P_{ext}$. This part 
requires to prove new stability estimates 
in order to prove that $(f_n)$ and $(u_n)$ are Cauchy sequences. 
This is where we need statement $ii)$ in Proposition~\ref{PP10-b-intro}, to estimate the difference of two controls associated with different coefficients. 
To prove this stability estimate we shall introduce an auxiliary control problem which, loosely speaking, interpolates the two control problems. 
Since the original nonlinear problem is quasi-linear, there is a loss 
of derivative (this reflects the fact that the flow map is expected to be merely continuous and not Lipschitz on Sobolev spaces). 
We overcome this loss by proving and using a regularity property of the control, see statement $ii)$ in Proposition~\ref{PP11-intro}. 
This regularity result is proved by adapting an argument used by Dehman-Lebeau~\cite{DehmanLebeau} and Laurent \cite{Laurent}. 
We also need to study how the control depends on $T$ or on the 
function $M$.

\subsection{Outline of the paper}
In Section~\ref{P:S2} we recall how to use paradifferential analysis 
to symmetrize the water waves equations. As mentioned above, 
the control problem for the water waves equations is studied by means of a nonlinear scheme. 
This requires to solve a linear control problem at each step. 
We introduce in Section~\ref{SN:3} this linear equation and state 
the main result we want to prove for it. In Section~\ref{P:S3},Ê
we commute the equations 
with a well-chosen elliptic operator to obtain a regularized problem. Once 
this step is achieved, we further transform in Section~\ref{S:22} 
the equations by means of a change of 
variables and by conjugating the equation with some pseudo-differential operator. 
Ingham's type inequalities are proved in Section~\ref{S:4} and then used in 
Section~\ref{P:S6} to deduce an observability result which in turn is used in 
Section~\ref{P:S7} to obtain a controllability result. In Section~\ref{P:S7} we also study 
the way in which the control depends on the coefficients, which requires to introduce several auxiliary control problems. Eventually, in 
Sections \ref{section9} and~\ref{P:S8} we use the previous control results for linear equations 
to deduce our main result Theorem~\ref{T1} by means of a 
quasi-linear scheme. 

To keep the paper self-contained, 
we add an appendix which contains two sections about paradifferential calculus 
and Sobolev energy estimates for classical or paradifferential evolution equations. 
The appendix also contains 
the analysis of various changes of variables which are used to conjugate the equations to a simpler form.

\section{Symmetrization of the water waves equations}\label{P:S2}
Consider the system 
\begin{equation}\label{system}
\left\{
\begin{aligned}
&\partial_t \eta = G(\eta)\psi,\\
&\partial_t\psi +g\eta+  \frac{1}{2} (\px\psi)^2  
-\frac{1}{2} \frac{\bigl(G(\eta)\psi + (\px\eta)(\px\psi)\bigr)^2}{1+ (\px\eta)^2}
= \kappamuet H(\eta)+\Pext(t,x).
\end{aligned}
\right.
\end{equation}
In this section, following \cite{ABZ1,AM} we recall how to use paradifferential analysis to rewrite 
the above system as a wave type equation for some new unknown $u$. 
This analysis is performed in \S\ref{P:S22}. 
In \S\ref{P:S21} and \S\ref{P:S23}, 
we complement the analysis in \cite{ABZ1,AM} by proving 
that all the coefficients can be expressed in terms of $u$ only. 
 
We refer the reader to the appendix for the definitions and the main results of 
paradifferential calculus. 

\subsection{Properties of the Dirichlet-Neumann operator}\label{P:S21}
We begin by recalling that, if $\eta$ in $W^{1,\infty}(\xT)$ and $\psi$ in $H^\mez(\xT)$, then $G(\eta)\psi$ is well-defined and belongs to $H^{-\mez}(\xT)$. 
Moreover, if $(\eta,\psi)$ belongs to $H^s(\xT)\times H^s(\xT)$ for some $s>3/2$, then 
$G(\eta)\psi$ belongs to $H^{s-1}(\xT)$ together with the estimate
\be\label{p20}
\lA G(\eta)\psi\rA_{H^{s-1}}\le C\big( \lA \eta\rA_{H^{s}}\big)\lA \psi\rA_{H^s}.
\ee
Then, it follows from usual nonlinear estimates in Sobolev spaces that the following result holds. 

Following \cite{AM,ABZ1}, the analysis is based on the so-called good unknown of Alinhac 
defined in the next lemma and denoted by $\omega$ 
as the notation for the control domain (both notations will not be used simultaneously).

\begin{lemm}
Let 
$s>3/2$ and $(\eta,\psi)$ in $H^s(\xT)\times H^s(\xT)$, the functions 
\be\label{p24z}
\begin{aligned}
B(\eta)\psi&\defn \frac{G(\eta)\psi+(\partial_x\eta)( \partial_x\psi)}{1+(\partial_x\eta)^2},
\qquad 
V(\eta)\psi&\defn \partial_x\psi-(B(\eta)\psi)\partial_x\eta,\\
\omega(\eta)\psi &\defn \psi-T_{B(\eta)\psi}\eta
\end{aligned}
\ee
belong, respectively, to $H^{s-1}(\xT)$, $H^{s-1}(\xT)$, $H^s(\xT)$ and satisfy
\begin{equation}\label{p25}
\lA B(\eta)\psi\rA_{H^{s-1}}+\lA V(\eta)\psi\rA_{H^{s-1}}+\lA \omega(\eta)\psi\rA_{H^s}
\le C\big( \lA \eta\rA_{H^{s}}\big)\lA \psi\rA_{H^s}.
\end{equation}
\end{lemm}
\begin{proof}
The estimate for $B(\eta)\psi$ and $V(\eta)\psi$ follow \e{p20}, by 
applying the usual nonlinear estimates in Sobolev spaces, see 
\e{esti:F(u)} and \e{prtame}. 
The Sobolev embedding then implies that $B(\eta)\psi$ belongs to 
$L^\infty(\xT)$. 
As a paraproduct with an $L^\infty$-function acts on any Sobolev space (see~\e{esti:quant0}), we deduce that
\be\label{p27}
\lA T_{B(\eta)\psi}\eta\rA_{H^{s}}
\les \lA B(\eta)\psi\rA_{L^\infty}
\lA \eta\rA_{H^s} \le C\left(\lA \eta\rA_{H^s}\right)\lA \psi\rA_{H^s}\lA \eta\rA_{H^s}.
\ee
This immediately implies the estimate for $\omega(\eta)\psi$ in \e{p25}. 
\end{proof}

Consider a Banach space $X$ and an operator $A$ whose operator norm is strictly smaller than $1$. 
Then it is well-known that $I-A$ is invertible. Now write 
$\omega(\eta)\psi$ under the form $(I-A)\psi$ with $A\psi=T_{B(\eta)\psi}\eta$. 
By applying the previous 
argument, it follows from \e{p27} that we have the following result. 

\begin{lemm}\label{LP:3}
Let $s>3/2$. There exists $\eps_0>0$ such that the following property holds. If 
$\lA \eta\rA_{H^s}<\eps_0$, then there exists a linear operator $\Psi(\eta)$ such that:

$i)$ for any $\psi$ in $H^\mez(\xT)$, 
$$
\Psi(\eta)\omega(\eta)\psi=\psi;
$$ 

$ii)$ if $\omega$ in $H^s(\xT)$ then $\Psi(\eta)\omega$ belongs to $H^{s}(\xT)$ and
\begin{equation}\label{p30}
\lA \Psi(\eta)\omega\rA_{H^s}
\le C\big( \lA \eta\rA_{H^{s}}\big)\lA \omega\rA_{H^s}.
\end{equation}
\end{lemm}

\begin{nota}
Hereafter, we often simply write 
$B,V,\omega$ instead of $B(\eta)\psi$, $V(\eta)\psi$, $\omega(\eta)\psi$. It follows 
from the above lemma that, if $\eta$ is small enough in $H^s(\xT)$, then $B$ and $V$ can be express in terms of 
$\eta$ and $\omega$:
$$
B=B(\eta)\Psi(\eta)\omega,\quad V=V(\eta)\Psi(\eta)\omega.
$$
\end{nota}

We also record the following corollary of the analysis 
in~\cite{AM,ABZ1}.
\begin{prop}\label{ParaDN}
Let $s\ge s_0$ with $s_0$ fixed large enough. There exists 
$\theta\in (0,1]$ such that
\begin{equation}\label{p38}
G(\eta)\psi=G(0)\omega-\partial_x \bigl( T_{V}\eta\bigr)+F(\eta)\psi
\end{equation}
where $F(\eta)\psi$ satisfies
\be\label{p39}
\lA F(\eta)\psi\rA_{H^{s+\mez}}\\
\le C\left(\lA \eta\rA_{H^s}\right)\lA \eta\rA_{H^s}^\theta \lA \psi\rA_{H^s}.
\ee
\end{prop}
\begin{proof}
We prove that $F(\eta)\psi$ satisfies the following two estimates:
\begin{align}
&\lA F(\eta)\psi\rA_{H^{s+1}}
\le C\big(\lA \eta\rA_{H^{s}}\big)\lA \psi \rA_{H^{s}},\label{p39-a}\\
&\lA F(\eta)\psi\rA_{H^{s-2}}
\le C\big(\lA \eta\rA_{H^{s}}\big)\lA \eta\rA_{H^{s}}
\lA \psi \rA_{H^{s}}.\label{p39-b}
\end{align}
The estimate \e{p39} then follows by interpolation in Sobolev spaces. 

Let us prove \e{p39-a}. 
In \cite{AM,ABZ1} it is proved that, for any $N$, when $s$ is large enough, 
$
G(\eta)\psi=\la D_x\ra\omega-\partial_x \bigl( T_{V}\eta\bigr)+\tilde F(\eta)\psi$
where $\blA \tilde F(\eta)\psi\brA_{H^{s+N}}
\le C\big(\lA \eta\rA_{H^{s}}\big)\lA \psi \rA_{H^{s}}$. 
Now notice that \e{p38} holds with 
$F(\eta)\psi=(\la D_x\ra-G(0))\omega+\tilde{F}(\eta)\psi$. 
Since $G(0)=\la D_x\ra\tanh(b\la D_x\ra)$, the difference 
$\la D_x\ra-G(0)$ is a smoothing operator. So using the estimate 
\e{p25} for $\omega$, we find that $\lA F(\eta)\psi\rA_{H^{s+N}}$ is bounded by the right-hand side of \e{p39-a}. Taking $N=1$ gives the desired result.

We now prove \e{p39-b}. As for \e{p27}, using the paraproduct rule \e{esti:quant0} and \e{p25}, 
one has 
$$
\lA \omega-\psi\rA_{H^{s}}+\blA \px (T_V \eta)\brA_{H^{s-1}}\les 
(\lA B\rA_{L^\infty}+\lA V\rA_{L^\infty})\lA \eta\rA_{H^s}\le C\big(\lA \eta\rA_{H^{s}}\big)\lA \eta\rA_{H^s}\lA \psi \rA_{H^{s}},
$$
hence it is sufficient to prove that
$\lA G(\eta)\psi-G(0)\psi\rA_{H^{s-2}}$ is bounded by the rhs of \e{p39-b}. This in turn will be deduced from an estimate of 
$\lA \varphi'(\tau)\rA_{H^{s-2}}$ where $\varphi(\tau)=G(\tau\eta)\psi$. 
Set $B_\tau=B(\tau\eta)\psi$ and $V_\tau=V(\tau\eta)\psi$. 
It follows from the computation of the shape derivative of the Dirichlet-Neumann operator (see \cite{LannesJAMS}) that 
$\varphi'(\tau)=-G(\tau\eta)(B_\tau \eta)-\px(V_\tau \eta)$. Now the 
estimate \e{p25}Ê
implies that 
$\lA \varphi'(\tau)\rA_{H^{s-2}}\le C\big(\lA \eta\rA_{H^{s}}\big)\lA \eta\rA_{H^{s}}
\lA \psi \rA_{H^{s}}$. Integrating in $\tau$ we complete the proof.
\end{proof}

\subsection{Symmetrization}\label{P:S22}
As already mentioned, the linearized equations are 
$$
\left\{
\begin{aligned}
&\partial_t \eta = G(0)\psi,\\
&\partial_t\psi +g\eta -\kappamuet \px^2 \eta=\Pext,
\end{aligned}
\right.
$$
where $G(0)=\la D_x\ra\tanh(b\la D_x\ra)$. Introducing the Fourier multiplier (of order $3/2$)
$$
\Lk\defn  \big( (g-\kappamuet\px^2) G(0)\big)^{\frac{1}{2}},
$$
with symbol
\be\label{defi:ell-lambda}
\ell(\xi)\defn \big( (g+|\xi|^2)\lambda(\xi)\big)^\mez \quad\text{where}\quad \lambda(\xi)\defn \la \xi\ra\tanh(b\la \xi\ra)
\ee
(so that $L=\ell(D_x)$), and considering $u=\psi-i \Lk G(0)^{-1}\eta$, one obtains the equation
$$
\partial_t u+i\Lk u=P_{ext}.
$$
The following proposition contains a similar diagonalization of System \e{system}.

\begin{prop}\label{T24}
Let $\sigma,\sigma_0$ be such that $\sigma\ge \sigma_0$ 
with $\sigma_0$ large 
enough. Consider a solution $(\eta,\psi)$ of~\eqref{system} on the time interval $[0,T]$ with $0<T<+\infty$, 
such that 
$$
(\eta,\psi)\in C^0\big([0,T];H^{\sigma+\mez}_0(\xT)\times H^{\sigma}(\xT)\big).
$$
Introduce
\be\label{p34b}
\begin{aligned}
&c\defn (1+(\px\eta)^2)^{-\tq},\\
&p\defn c^{-\frac{1}{3}}+\frac{5}{18i}\frac{\chi(\xi)\partial_\xi \ell(\xi)}{\ell(\xi)}
c^{-\frac{4}{3}}\px c,\quad 
q= \chi(\xi)\Big(c^{\frac 2 3}\frac{\ell(\xi)}{\lambda(\xi)}+
\big(\px c^{\frac 2 3}\big) \frac{\ell(\xi)}{i\xi\lambda(\xi)}\Big),
\end{aligned}
\ee
where $\ell,\lambda$ are as in \e{defi:ell-lambda}, $\chi\in C^\infty$ satisfies $\chi(\xi)=1$ for $|\xi|\ge 2/3$ 
and $\chi(\xi)=0$ for $\la \xi\ra\le 1/2$. 
Then
$$
u\defn  T_p \omega-iT_q \eta
$$
satisfies
\be\label{p35}
\partial_t u +T_V \px u +iL^\mez \big(T_c L^\mez u\big)+R(\eta,\psi)=  T_p P_{ext},
\ee
for some remainder $R(\eta,\psi)=R_1(\eta)\psi+R_2(\eta)\eta$ with
\be\label{p35b}
\begin{aligned}
&\lA R_1(\eta)\psi\rA_{H^\sigma}\le 
C \big(\lA \eta\rA_{H^{\sigma+\mez}}
\big)\lA \eta\rA_{H^{\sigma+\mez}}^\theta\lA \psi\rA_{H^\sigma},\\
&\lA R_2(\eta_1)\eta_2\rA_{H^\sigma}
\le C \big(\lA \eta_1\rA_{H^{\sigma+\mez}}
\big)\lA \eta_1\rA_{H^{\sigma+\mez}}^\theta\lA \eta_2\rA_{H^{\sigma+\mez}},
\end{aligned}
\ee
for some $\theta\in (0,1]$ given by Proposition~$\ref{ParaDN}$. 
\end{prop}
\begin{rema}
Compared to a similar result proved in \cite{ABZ1}, 
there are two differences. We here obtain 
a super-linear remainder term (see \e{p35b}), 
and secondly we prove here that $q$ can be so chosen that 
$T_q=\px T_Q$ for some symbol $Q$; namely,
\be\label{p36}
T_q=\px T_Q \quad\text{with}\quad Q\defn \chi(\xi)c^{\frac{2}{3}}\frac{\ell(\xi)}{\lambda(\xi) i\xi}.
\ee
This will be used to obtain that $\int T_q \eta \, dx=0$. Since it is not a trivial 
task to obtain these additional properties, we shall recall the strategy of the proof from \cite{ABZ1} 
and give a detailed analysis of the required modifications.
\end{rema}
\begin{proof}
The first step consists in paralinearizing the equation. We use in particular the paralinearization of the Dirichlet-Neumann operator (see \e{p38}). Then, by using 
the paralinearization formula for products (replacing products 
$ab$ by $T_ab+T_ba+\RBony(a,b)$), it follows from direct computations (see \cite{ABZ1}) that
\begin{equation}\label{p40}
\left\{
\begin{aligned}
&\partial_t \eta+\partial_x(T_V \eta)- G(0)\omega =F^1,\\
&\partial_t \omega+T_V  \px \omega +T_{a} \eta -\kappamuet H(\eta)=F^2+\Pext,
\end{aligned}
\right.
\end{equation}
where $a$ denotes the Taylor coefficient, which is
$$
a = g+ \partial_t B +V  \partial_x B,
$$
and $F^1$ and $F^2$ are given by 
(see \e{defi:RBony}Ê
for the definition of $\RBony(a,b)$)
\begin{align*}
F^1&=F(\eta)\psi,\\
F^2&=(T_{V}  T_{\partial_x\eta}-T_{V \partial_x\eta})B
+(T_{V \partial_x B}-T_{V}  T_{\partial_x B})\eta\\
&\quad +\mez \RBony(B,B)-\mez\RBony(V,V)+T_V\RBony(B,\partial_x\eta)
-\RBony(B,V\px\eta).
\end{align*}

On the other hand, the paralinearization estimate~\ref{FBony} applied with $\alpha=\sigma-1/2$ implies that
$$
\frac{\px\eta}{\sqrt{1 + (\px\eta)^2}} 
=T_{r} \px \eta + \widetilde{f} , \qquad r\defn (1+(\px\eta)^2)^{-\tdm},
$$
where  $\tilde{f}\in L^\infty(0,T;H^{2\sigma-\tdm})$ is such that
\begin{equation*}
\blA \widetilde{f}\brA_{H^{2\sigma-\tdm}} \le C ( \lA \eta\rA_{H^{\sigma+\mez}}) \lA \eta\rA_{H^{\sigma+\mez}}^2,
\end{equation*}
for some non-decreasing function $C$. 
Hence, directly from \e{p40}, we obtain that
\begin{equation*}
\left\{
\begin{aligned}
&\partial_t \eta+T_V \partial_x\eta- G(0)\omega =f^1,\\
&\partial_t \omega+T_V  \px \omega + g\eta -\kappamuet 
\px(T_{r}\px \eta)=f^2+\Pext,
\end{aligned}
\right.
\end{equation*}
where
\[ 
f^1 \defn F^1-T_{\px V}\eta, \qquad 
f^2 \defn F^2+\kappamuet \px \widetilde{f}+T_{g-a}\eta.
\]
Then introduce 
$\zeta\defn T_q \eta$ and $\theta\defn T_p \omega$. It is found that
\begin{equation}\label{p40a}
\left\{
\begin{aligned}
&\partial_t \zeta+T_V \partial_x\zeta- T_q G(0)\omega =\tilde f^1,\\
&\partial_t \theta+T_V  \px \theta +T_p\big(g\eta -\kappamuet 
\px(T_{r}\px \eta)\big)=\tilde f^2+T_p \Pext,
\end{aligned}
\right.
\end{equation}
where
\begin{align*}
\tilde f^1&\defn T_q f^1 +T_{\partial_t q}\eta +[T_V\px ,T_q]\eta,\\
\tilde f^2&\defn T_p f^2 +T_{\partial_t p}\omega+[T_V\px,T_p]\omega.
\end{align*}

Assuming that $q$ and $p$ are as in the statement of the proposition, 
it easily follows from \e{p39} and the paradifferential rules \e{esti:quant1}, \e{esti:quant0} and \e{esti:quant2} (applied with $\rho=1$ to bound the operator norm of the commutators 
$[T_V\px ,T_q]$ and $[T_V\px ,T_p]$) that
$$
\lA (f^1,f^2)\rA_{H^\sigma}\le C \big(\lA \eta\rA_{H^{\sigma+\mez}}
\big)\lA \eta\rA_{H^{\sigma+\mez}}^\theta\big\{ \lA \psi\rA_{H^\sigma}+\lA \eta\rA_{H^{\sigma+\mez}} \big\}.
$$

It remains to compute $T_q G(0)\omega $ and $T_p\big(g\eta -\kappamuet 
\px(T_{r}\px \eta)\big)$. 
More precisely, it remains to establish that
\begin{align} \label{p40ab}
&\blA T_q G(0)\omega - L^\mez T_c L^\mez T_p\omega\brA_{H^\sigma}
\le C\big(\lA \eta\rA_{H^{\sigma+\mez}}\big)
\lA \eta\rA_{H^{\sigma+\mez}}^\theta
\lA \omega\rA_{H^\sigma},\\
&\blA T_p\big(g\eta -\kappamuet 
\px(T_r\px \eta)\big) -L^\mez T_c L^\mez T_q\eta\brA_{H^\sigma}
\le C\big(\lA \eta\rA_{H^{\sigma+\mez}}\big) \lA \eta\rA_{H^{\sigma+\mez}}^\theta\lA \eta\rA_{H^{\sigma+\mez}}.
\notag \end{align}
(We prove below these estimates with $\theta=1$.) 
Then the estimates \e{p35b}Ê
follow from \e{p25} which gives a bound for $\lA \omega\rA_{H^\sigma}$ in terms of 
$\lA \psi\rA_{H^\sigma}$.

To prove \e{p40ab}, it is convenient 
to introduce the following notation: Given two operators, the notation 
$A\sim B$ means that, 
for any $\mu\in\xR$ there is a constant $C\big(\lA \eta\rA_{H^{\sigma+\mez}}\big)$ such that
$$
\lA (A-B)u\rA_{H^\mu}\le C\big(\lA \eta\rA_{H^{\sigma+\mez}}\big) \lA \eta\rA_{H^{\sigma+\mez}}\lA u\rA_{H^\mu}.
$$
In words, $A\sim B$ means that $A$ equals $B$ modulo a remainder which is 
of order $0$ and quadratic. 

For instance consider real numbers $m,m'$ with $m+m'=2$ and two operators 
$A=T_{a^{(m)}+a^{(m-1)}}$ and $B=T_{b^{(m')}+b^{(m'-1)}}$ 
where
$$
a^{(m)}\in \Gamma^m_2,\quad a^{(m-1)}\in \Gamma^{m-1}_1,\quad b^{(m')}\in \Gamma^{m'}_2,\quad b^{(m'-1)}\in \Gamma^{m'-1}_1
$$
(see Definition \ref{defi Gamma}) with
(see \e{defi:norms}) 
\begin{align*}
&M_2^{m'}(b^{(m')})+M_1^{m'-1}(b^{(m'-1)})\le C\big( \lA \eta\rA_{H^{\sigma+\mez}}\big),\\
&M_2^{m}(a^{(m)})+M_1^{m-1}(a^{(m-1)})\le C\big( \lA \eta\rA_{H^{\sigma+\mez}}\big)\lA \eta\rA_{H^{\sigma+\mez}}.
\end{align*}
By using \e{esti:quant2sharp} applied with $\rho=2$ and \e{esti:quant2} applied with 
$\rho=1$, we obtain that
\begin{alignat*}{2}
&T_{a^{(m)}}T_{b^{(m')}}\sim T_{a^{(m)}b^{(m')}+\frac{1}{i}\partial_\xi a^{(m)} \px b^{(m')}}, 
\qquad & 
& T_{a^{(m)}}T_{b^{(m'-1)}}\sim T_{a^{(m)}b^{(m'-1)}}, 
\\
& T_{a^{(m-1)}}T_{b^{(m')}}\sim T_{a^{(m-1)}b^{(m')}}, \qquad & 
& T_{a^{(m-1)}}T_{b^{(m'-1)}}\sim 0,
\end{alignat*}
so
\be\label{pb40}
AB\sim T_{a^{(m)}b^{(m')}+\frac{1}{i}\partial_\xi a^{(m)} \px b^{(m')}+a^{(m)}b^{(m'-1)}
+a^{(m-1)}b^{(m')}}.
\ee

Using the previous notation, to prove \e{p40ab} we have to prove that 
\be\label{pb1}
\begin{aligned}
&T_q G(0) \sim  L^\mez T_c L^\mez \chi(D_x)T_p\\
&T_p\big(gI -\kappamuet 
\px(T_{r}\px \cdot)\big)\chi(D_x) \sim  L^\mez T_c L^\mez \chi(D_x)T_q.
\end{aligned}
\ee
Notice that $\chi(D_x)\eta=\eta$ and 
$L^\mez \chi(D_x)u=L^\mez u$ for any periodic function $u$. 
This is why we can introduce the cut-off function $\chi$ in the calculations. 
This cut-off function is used to handle symbols which are not 
smooth at $\xi=0$.

We remark that, by definition of paradifferential operators, we have
$$
T_q G(0)=T_{q\lambda(\xi)},\quad 
gI -\kappamuet 
\px(T_r\px \cdot)=T_{g+\kappamuet r \xi^2-\kappamuet(\px r)(i\xi)}.
$$
{\em Study of the first identity in} \e{pb1}. 
It follows from symbolic calculus (see \e{esti:quant2sharp}) that
\be\label{pb41}
L^\mez T_c L^\mez\chi(D_x) \sim T_\gamma
\qquad \text{with} \quad 
\gamma=\chi c\ell +\frac{\chi}{i}\big(\partial_\xi \sqrt{\ell}\big)\sqrt{\ell} \px c.
\ee 
Now we seek $q$ under the form $q=q^{(1/2)}+q^{(-1/2)}$ where $q^{(1/2)}$ is of order $1/2$ in $\xi$ 
(more precisely, $q\in \Gamma_{2}^{1/2}$) and 
$q^{(-1/2)}$ is of order $-1/2$ (in $\Gamma^{-1/2}_{1}$). 
Similarly, we seek $p=p^{(0)}+p^{(-1)}$ 
with $p\in \Gamma_{2}^{0}$ and 
$p^{(-1)}\in \Gamma^{-1}_{1}$.

Also, it follows from \e{pb40} 
that 
$L^\mez T_c L^\mez \chi(D_x) T_p\sim T_\gamma T_p\sim T_{\wp_1}$ with 
$$
\wp_1=\gamma p^{(0)}+\chi c\ell p^{(-1)} +\frac{1}{i}\chi c(\partial_\xi \ell)\px p^{(0)}
$$
(the contribution of $(\partial_\xi\chi)\px p^{(0)}$ is in the remainder term). The first identity in \e{pb1} will be satisfied if
$$
q^{(1/2)}\defn \chi cp^{(0)}\frac{\ell(\xi)}{\lambda(\xi)},\quad q^{(-1/2)}=\frac{\chi}{i}\frac{\partial_\xi \ell(\xi)}{\lambda(\xi)}\Big[\mez (\px c)p^{(0)}+c\px p^{(0)}\Big]
+\chi c p^{(-1)}\frac{\ell(\xi)}{\lambda(\xi)}.
$$

{\em Study of the second identity in} \e{pb1}. 
As above, it follows from symbolic calculus that 
$L^\mez T_c L^\mez \chi(D_x)T_q\sim T_\gamma T_q\sim T_{\wp_2}$ with (see \e{pb40}) 
$$
\wp_2=\gamma q^{(1/2)}+\frac{1}{i}\chi c(\partial_\xi \ell)\px q^{(1/2)}+\chi c \ell q^{(-1/2)}.
$$
With $q^{(1/2)}$ and $q^{(-1/2)}$ as given above, we compute that
$$
\wp_2=\chi\Big\{c\ell q^{(1/2)}+\frac{\chi}{i}\frac{\partial_\xi \ell^2}{\lambda(\xi)}
\Big(cp^{(0)}\px c+c^2\px p^{(0)}\Big)+\chi c^2 p^{(-1)}\frac{\ell(\xi)^2}{\lambda(\xi)}\Big\}.
$$
Moreover, by definition of $\ell(\xi)$ one has
$$
\frac{\partial_\xi\ell^2}{\lambda(\xi)} 
=
3\xi+r_1,
\quad 
\frac{\ell(\xi)^2}{\lambda(\xi)}=\kappamuet\xi^2 +r_2,
\qquad r_1,r_2\text{ are of order }0.
$$
Notice that the contribution of the term 
$r_1(cp^{(0)}\px c+c^2\px p^{(0)})$ to $T_{\wp_2}$ and the one of 
$r_2c^2 p^{(-1)}$ 
can be handled as remainder terms and hence  
$$
L^\mez T_c L^\mez \chi(D_x)T_q\sim 
T_{\tilde\wp_2}
$$
with 
$$
\tilde\wp_2=\chi \Big\{ c\ell q^{(1/2)}+\frac{3\chi\kappamuet}{i}\xi\Big(
cp^{(0)}\px c+c^2\px p^{(0)}\Big)+\chi \kappamuet c^2 p^{(-1)}\xi^2\Big\}.
$$
On the other hand, 
$$
T_p\big(gI -\kappamuet 
\px(T_{r}\px \cdot)\big)\chi(D_x)
\sim T_{\chi p\big(g+\kappamuet r \xi^2-\kappamuet(\px r)(i\xi)\big)}.
$$

By definition of $c,\ell,q^{(1/2)}$, recall that 
$r=c^2$ and $\ell^2=(g+\kappamuet\xi^2)\lambda(\xi)$ 
and hence
\begin{align*}
p\big(g+\kappamuet r \xi^2-\kappamuet(\px r)(i\xi)\big) 
&=p\kappamuet c^2 \xi^2+gp -\kappamuet p (\px r)(i\xi)\\
&=p c^2 \frac{\ell(\xi)^2}{\lambda(\xi)}
-gpc^2+gp -\kappamuet p (\px r)(i\xi).
\end{align*}
Since $q^{(1/2)}\defn \chi cp^{(0)} \frac{\ell(\xi)}{\lambda(\xi)}$, we deduce that
$$
p\big(g+\kappamuet r \xi^2-\kappamuet(\px r)(i\xi)\big)\chi
=c \ell q^{(1/2)}+\chi\Big\{p^{(-1)} c^2 (g+\kappamuet\xi^2)+gp(1-c^2) -\kappamuet i p (\px r) \xi\Big\}.
$$ 
Since $1-c^2$ and $\px r$ depend at least linearly on $\eta$ and since $p$ and $p^{(-1)}\xi$ are symbols of order $0$, 
it follows from the estimate \e{esti:quant1} for the operator norm of a paradifferential operator that
\begin{align*}
&\lA T_{p(1-c^2)}u\rA_{H^\mu}\le C \big( \lA \eta\rA_{H^{\sigma+\mez}}\big) 
\lA \eta\rA_{H^{\sigma+\mez}}^2\lA u\rA_{H^\mu},\\
&\blA T_{ p^{(-1)} (\px r) \xi}u\brA_{H^\mu}\le C\big(\lA \eta\rA_{H^{\sigma+\mez}}\big)\lA \eta\rA_{H^{\sigma+\mez}}\lA u\rA_{H^\mu}.
\end{align*}
Similarly, assuming that $p^{(-1)}$ is a symbol of order $-1$ depending linearly on $\eta$ (as this will be true, see \e{p34b}),
we have
$$
\blA T_{p^{(-1)} c^2 g}u\brA_{H^\mu}\le C\big(\lA \eta\rA_{H^{\sigma+\mez}}\big)\lA \eta\rA_{H^{\sigma+\mez}}\lA u\rA_{H^\mu}.$$
Therefore, 
$$
T_p\big(gI -\kappamuet 
\px(T_{r}\px \cdot)\big)\chi(D_x)\sim T_{c \ell q^{(1/2)} 
+\chi p^{(-1)} c^2 \xi^2-\chi p^{(0)} (\px r)(i\xi)} .
$$
Now since $r=c^2$, with $p^{(0)}=c^{-\frac{1}{3}}$, we have 
$$
-p^{(0)} (\px r)(i\xi)=
+\frac{3}{i}\xi\Big(
cp^{(0)}\px c+c^2\px p^{(0)}\Big),
$$
as can be verified by a direct calculation, so 
the second identity in \e{pb1} holds.

It remains to compute $q$. We have
$$
q=\chi\bigg\{
cp^{(0)}\frac{\ell(\xi)}{\lambda(\xi)}+\frac{1}{i}\frac{\partial_\xi \ell(\xi)}{\lambda(\xi)}\Big[\mez (\px c)p^{(0)}+c\px p^{(0)}\Big]
+c p^{(-1)}\frac{\ell(\xi)}{\lambda(\xi)}\bigg\}.
$$
Observe that
$$
\Big[\mez (\px c)p^{(0)}+c\px p^{(0)}\Big]=\frac{1}{6}c^{-\frac{1}{3}}\px c.
$$
We now seek $p^{(-1)}$ such that
$$
c p^{(-1)}\frac{\ell(\xi)}{\lambda(\xi)}=\alpha\frac{1}{i}\frac{\partial_\xi \ell(\xi)}{\lambda(\xi)}
c^{-\frac{1}{3}}\px c
$$
for some constant $\a$ to be determined. We thus set
$$
p^{(-1)}\defn \alpha \frac{\chi(\xi)}{i}\frac{\partial_\xi \ell(\xi)}{\ell(\xi)}
c^{-\frac{4}{3}}\px c.
$$
Then (replacing $\chi^2$ by $\chi$, to the price of adding a 
smoothing operator in the remainder), we have
$$
q\defn 
\chi\bigg\{c^{\frac 2 3}\frac{\ell(\xi)}{\lambda(\xi)}+\frac{1}{i}\frac{\partial_\xi \ell(\xi)}{\lambda(\xi)}\Big[
\Big(\alpha+\frac 1 6\Big) c^{-\frac 1 3}\px c\Big]\bigg\}.
$$
Since
$$
\chi(\xi)\xi\partial_\xi\ell=\tdm \chi \ell+\tau(\xi)
$$
with $\tau(\xi)$ is a smooth symbol of order $1/2$, we have
$$
\frac 2 3 \chi(\xi)\frac{\ell}{\lambda(\xi) i\xi} =
\frac{4}{9i}\chi(\xi)
\frac{\partial_\xi \ell}{\lambda(\xi)}+r'
$$
where $r'$ is of order $-3/2$. Then, 
choosing $\alpha$ such that $\alpha+1/6=4/9$, we find that
$$
q= c^{\frac 2 3}\frac{\ell(\xi)}{\lambda(\xi)}+
\big(\px c^{\frac 2 3}\big) \frac{\ell(\xi)}{i\xi\lambda(\xi)}+\tilde r
$$
where $\tilde r$ is such that
$$
\lA T_{\tilde r}u\rA_{H^{\mu+\tdm}}
\le C\big(\lA \eta\rA_{H^{\sigma+\mez}}\big)\lA \eta\rA_{H^{\sigma+\mez}}\lA u\rA_{H^\mu}.
$$
In particular, the contribution of $\tilde r$ can be handled 
as a remainder term and the same results hold when $q$ is replaced 
by the same expression without $\tilde r$, thereby obtaining \e{p34b}. 
This completes the proof of \e{pb1} and hence the proof of the proposition.
\end{proof}

\subsection{Invertibility of the change of unknowns}\label{P:S23}

We thus obtained an equation of the form
$$
\partial_t u +T_V \px u +iL^\mez\big(T_c L^\mez u\big) + R(\eta,\psi)= T_p P_{ext},
$$
where the coefficients $V$ and $c$ depend on the original unknowns $(\eta,\psi)$. 
We conclude this section by proving that $V$ and $c$ can be expressed in terms of $u$ only. 
We have already seen in Lemma~\ref{LP:3} that these coefficients can be expressed in terms 
of $\eta$ and~$\omega$. So it remains only to express $(\eta,\omega)$ in terms of 
$u$. 

In this paragraph, the time is seen as a parameter and we skip it.

\begin{nota}\label{N:26}
Introduce the space $\tilde H^\sigma(\xT;\xC)$ of complex-valued functions $u$ satisfying
$$
\int_\xT \IM u (x)\, dx=0.
$$
\end{nota}

Recall (see \e{esti:quant1}) that a paradifferential operator with symbol in 
$\Gamma^m_0$ is bounded from any Sobolev space $H^\mu(\xT)$ to 
$H^{\mu-m}(\xT)$. Recall also that 
$\omega\in H^\sigma(\xT)$ 
whenever $(\eta,\psi)\in H^{\sigma+\mez}_0(\xT)\times H^\sigma(\xT)$. 
Since, as already mentioned, 
$T_q\eta=T_{q\chi}\eta$ where 
$\chi$ is as defined after \e{p34b} and since 
$q\chi\in \Gamma^{1/2}_0$, we deduce that 
$u$ belongs to $H^\sigma(\xT)$. Moreover, 
it follows from \e{p36}Ê
that $T_p \omega-iT_q \eta$ belongs to 
$\tilde H^\sigma(\xT;\xC)$. 

We now introduce the mapping  
$U\colon H^{\sigma+\mez}_0(\xT)\times H^\sigma(\xT) \rightarrow \tilde H^\sigma(\xT;\xC)$ so that
$$
U(\eta,\psi)\defn T_p \omega-iT_q \eta.
$$
The following result shows 
that this nonlinear mapping can be inverted. 
\begin{lemm}\label{T27}
Let $\sigma_0>5/2$. There exists $\eps_0>0$ and $K$ such that the following properties holds. If 
$\lA \eta\rA_{H^{\sigma_0}}<\eps_0$, then there exists
$$
Y\colon \tilde H^{\sigma_0}(\xT;\xC)\rightarrow H^{\sigma_0+\mez}_0(\xT)\times H^{\sigma_0}(\xT),
$$
such that $Y(u)=(\eta,\psi)$ with $u=U(\eta,\psi)$. 
Moreover, for any 
$\sigma>5/2$, 
\be\label{p44a}
\lA \eta\rA_{H^{\sigma+\mez}}\le 2 \lA u\rA_{H^\sigma}, \quad 
\lA \psi\rA_{H^\sigma}\le 2 \lA u\rA_{H^\sigma}.
\ee
\end{lemm}
\begin{proof}Set $u=U(\eta,\psi)\defn T_p \omega-iT_q \eta$. 
Then $T_q \eta=-\IM u$ and $T_p\omega=\RE u$,
where $q$ and $p$ depend on $\eta$. The only difficulty is to express $\eta$ in terms of 
$\IM u$. Once this will be granted, to invert the equation $T_p \omega=\RE u$, 
we use the fact that $T_p$ is a small bounded perturbation 
of the identity so that $T_p$ is invertible, indeed (recalling that $M_\rho^m(a)$ is defined by \e{defi:norms}) 
$$
\lA T_p-I\rA_{\mathcal{L}(H^\sigma)}\les M^0_0(p-1)\le C\left(\lA \eta\rA_{H^{\sigma+\mez}}\right)\lA \eta\rA_{H^{\sigma+\mez}}.
$$

Now to solve the equation $T_q \eta=-\IM u$, we use the Banach fixed point theorem. 
Denote by $Q$ the Fourier multiplier with symbol 
$Q(\xi)\defn \chi(\xi)\ell(\xi)/\lambda(\xi)=\chi(\xi)\sqrt{g+\kappamuet \xi^2}/\sqrt{\lambda(\xi)}$. 
The reason to introduce this symbol is that, with $q$ as given by \e{p34b} one has
\be\label{p45}
M_0^{1/2}\big(q(x,\xi)-Q(\xi)\big)\le C(\lA \eta\rA_{H^{\sigma+\mez}})\lA \eta\rA_{H^{\sigma+\mez}},
\ee
which is obtained by considering separately the principal and sub-principal terms in the definition of $q$. 
Then seek 
$\eta$ in $H^{\sigma+\mez}_0(\xT)$ 
such that $\Phi(\eta)=\eta$ with 
$$
\Phi(\eta)\defn -(g-\kappamuet\px^2)^{-\mez}G(0)^{\mez}\big((T_q-Q) \eta+\IM u\big).
$$
It is easily verified that if $\Phi(\eta)=\eta$ then $T_q\eta=-\IM u$ and also that 
$\Phi$ maps $H^{s+\mez}_0(\xT)$ into itself. To see that $\Phi$ is a contraction, we use 
\e{p45}Ê
to obtain
\begin{align*}
\lA \Phi(\eta_1)-\Phi(\eta_2)\rA_{H^{\sigma+\mez}}
&\les 
\blA \big(T_{q_1}-Q\big)(\eta_1-\eta_2)\brA_{H^\sigma}
+\lA \left(T_{q_1}-T_{q_2}\right)\eta_2\rA_{H^\sigma}\\
&\les M_0^{1/2}\big(q_1-Q\big)\lA \eta_1-\eta_2\rA_{H^{\sigma+\mez}}
+M_0^{1/2}(q_1-q_2)\lA \eta_2\rA_{H^{\sigma+\mez}}\\
&\le C(M)M\lA \eta_1-\eta_2\rA_{H^{\sigma+\mez}}
\end{align*}
where $M\defn \lA \eta_1\rA_{H^{\sigma+\mez}}
+\lA \eta_2\rA_{H^{\sigma+\mez}}$. 
If $M$ is small enough, then $\Phi$ is a contraction. 
\end{proof}

\section{The linear equation}\label{SN:3}

As mentioned in the introduction, we shall study the control problem for the water waves 
equations by means of a nonlinear scheme. This requires to solve a linear control problem at each step. 
We introduce in this section the linear equation we are going to study until section~\ref{P:S8}, 
emphasize one key property of this equation and state the main result we want to prove.

We have seen in the previous section that one can express $V=V(\eta)\psi$ in terms of $u$ only. 
To simplify notations, we write $V=V(u)$, and similarly we write $c=c(u)$. Also, one can write the remainder 
$R(\eta,\psi)$ under the form $R(u)u$ 
where, for any $\underline{u}$, the mapping $u\mapsto R(\underline{u})u$ is linear. 

We have proved that, for $\sigma$ large enough 
and a solution $(\eta,\psi)$ of~\eqref{system} on the time interval $[0,T]$, satisfying 
$$
(\eta,\psi)\in C^0\big([0,T];H^{\sigma+\mez}_0(\xT)\times H^{\sigma}(\xT)\big),
$$
the new unknown $u$ satisfies $u\in C^0([0,T];\tilde H^\sigma(\xT;\xC))$ (where 
$\tilde H^\sigma(\xT;\xC)$ is defined in Notation~\ref{N:26}) and
$$
\partial_t u +T_{V(u)}\px u +iL^\mez\big(T_{c(u)}L^\mez u\big) +R(u)u = T_{p(u)} P_{ext}.
$$

We now fix $\underline{u}\in C^0([0,T];\tilde H^\sigma(\xT;\xC))$ and then set
\be\label{ps31}
V=V(\underline{u}), \quad c=c(\underline{u}), \quad R=R(\underline{u}), \quad p=p(\underline{u})
\ee
and consider the linear operator
$$
P=\partial_t +T_{V}\px  +iL^\mez\big(T_{c}L^\mez \cdot\big)+R.
$$
Except for the second condition in Assumption~\ref{T31A} below, we shall not use the way in which the coefficients depend on $\underline{u}$  
and hence we shall state all the assumptions on $V,c,p,R$ forgetting their  dependence on 
$\underline{u}$ through \e{ps31}.

\begin{assu}\label{T31A}
\begin{enumerate}[i)]
\item Consider two real-valued functions $V,c$ 
in $C^0([0,T];H^{s_0}(\xT))$ for some $s_0$ large enough, with 
$c$ bounded from below by $1/2$. 
The symbol $p$ is given by $p\defn c^{-\frac{1}{3}}+\frac{5}{18i}\frac{\chi(\xi)\partial_\xi \ell(\xi)}{\ell(\xi)}
c^{-\frac{4}{3}}\px c$ with $\chi$ as in \e{p34b}. It is always assumed that the $\eC{\tdm}$-norm of $c-1$ is small enough.

\item \label{T31A-cond3} If $Pu$ is a real-valued function then
$$
\frac{d}{dt}\int_\xT \IM u (t,x)\, dx=0.
$$
\end{enumerate}
\end{assu}
Fix an open domain $\omega\subset \xT$ and denote by $\chi_\omega$ a $C^\infty$ 
cut-off function such that $\chi_\omega(x)=1$ for 
$x\in\omega$. We want to study the following control problem: given an initial data 
$v_{in}$ find $f$ such that the unique solution to 
\be\label{p50}
Pv=T_{p}\chi_\omega \RE f,\quad v_{\arrowvert t=0}=v_{in}
\ee
satisfies $v_{\arrowvert t=T}=0$. The fact that the Cauchy problem for \e{p50} 
admits a unique solution is proved in the appendix, see Proposition~\ref{P:10}. 

Our main goal until Section~\ref{P:S8} will be to prove the following control result.

\begin{prop}\label{PP10}
There exists $s_0$ large enough such that, for all $T\in (0,1]$ and 
all $s\ge s_0$, if Assumption~\ref{T31A} holds then 
there exist two positive constants 
$\wide\delta=\wide\delta(T,s)$ and $K=K(T,s)$ 
such that, if
\be\label{pn200}
\begin{aligned}
&\lA V\rA_{\bo([0,T];H^{s_0})}+\lA c-1\rA_{\bo([0,T];H^{s_0})}\le \wide\delta, \\
& \blA \partial_t^k V\brA_{\bo([0,T];H^1)}
+\blA \partial_t^k c\brA_{\bo([0,T];H^1)}\le \wide\delta \qquad (1\le k\le 3),\\
&\lA R\rA_{\bo([0,T];\mathcal{L}(H^s))}\le \wide\delta,
\end{aligned}
\ee
then 
for any initial data 
$v_{in}\in \tilde H^s(\xT;\xC)$  
there exists $f\in C^0([0,T];H^s(\xT))$ such that: 
\begin{enumerate}
\item \label{PP10-prop1}
the unique solution $v$ to $Pv=T_p\chi_\omega \RE f,\quad v_{\arrowvert t=0}=v_{in}$ satisfies $v(T)=0$;

\item $\lA f\rA_{\bo([0,T];H^s)}\le K\lA v_{in}\rA_{H^s}$. 
\end{enumerate}
\end{prop}
\begin{rema}
Notice that the smallness assumption on $V$ and $c$ involves only 
some $H^{s_0}$-norm, while the result holds for 
initial data in $H^s$ with $s\ge s_0$. 
We shall use this property with $s_0=s-2$ in the analysis of the quasi-linear scheme. 
This is possible only because we consider a paradifferential equation. 
\end{rema}

We conclude this section by proving that the second condition in Assumption~\ref{T31A} holds when $V,c,p,R$ are given by \e{ps31}. 

\begin{lemm}
Consider $\underline{u}\in C^0([0,T];\tilde H^{s_0}(\xT;\xC))$ with $s_0$ large enough and assume 
that $V,c,p,R$ are given by \e{ps31}. 
If $Pu$ is a real-valued function, then
$$
\frac{d}{dt}\int_\xT \IM u (t,x)\, dx=0.
$$
\end{lemm}
\begin{proof}
Set $\zeta=\IM u$. It follows from \e{p40a} that $\zeta$ satisfies
\begin{align*}
&\partial_t \zeta+T_V \partial_x\zeta- T_q G(0)\omega =\tilde f^1,\\
&\tilde f^1= T_q \big(F(\eta)\psi-T_{\px V}\eta\big) +T_{\partial_t q}\eta +[T_V\px ,T_q]\eta,
\end{align*}
where $F(\eta)\psi$ is given by \e{p38}.  
One can write this equation under the form
\be\label{pn201a}
\partial_t \zeta+T_q \big( \px (T_V \eta)\big)- T_q G(0)\omega =T_q F(\eta)\psi+T_{\partial_t q}\eta.
\ee
Notice that $T_q G(0)\omega$ and $T_q F(\eta)\psi$ are well-defined since 
$\widehat{G(0)\omega}(0)=0=\widehat{F(\eta)\psi}(0)$ (this follows from the definition \e{p38} 
and the fact that the mean values of 
$G(\eta)\psi$, $G(0)\omega$ and $\px (T_V\eta)$ are all $0$). Using 
\e{p36}, one obtains that $\int_{\xT} T_q v \, dx=0=\int_{\xT} T_{\partial_t q} v \, dx$ for any function $v$. So integrating \e{pn201a} we obtain 
the desired result.
\end{proof}

\section{Reduction to a regularized equation}\label{P:S3}

In this section, we reduce the proof of Proposition~\ref{PP10}Ê
to that of a simpler result. We shall prove that:
\begin{itemize}
\item it is enough to consider a classical equation instead of a paradifferential equation (this observation 
will be used below to simplify the computation of a change of variable);
\item it is enough to prove an $L^2$-result instead of a result in higher order Sobolev spaces (this plays a crucial role). 
\end{itemize}

As explained in the introduction, the idea is to 
commute the equation with an elliptic semi-classical operator 
$\Lambda_{h,s}$ of order $s$. To choose this elliptic operator, 
the key point is to prove that $\Lambda_{h,s}$ 
can be so chosen that it satisfies  
the following commutator estimates:
\begin{align*}
&\blA [\Lambda_{h,s},P]\Lambda_{h,s}^{-1}\brA_{\mathcal{L}(L^2)}=O(1),\\
&\blA [\Lambda_{h,s},\chi_\omega]\Lambda_{h,s}^{-1}\brA_{\mathcal{L}(L^2)}=O(h),
\end{align*}
which is the reason to introduce the small parameter $h$. 
Some care is required to do so, 
and we introduce
\be\label{pf1}
\Lambda_{h,s}=I+h^{s} T_{c^{\frac{2s}{3}}}L^{\frac{2s}{3}}
\ee
\begin{lemm}
$i)$ Assume that the $L^\infty_{t,x}$-norm of $c-1$ is small enough. Then 
$\Lambda_{h,s}$ is invertible from $H^s$ to $L^2$ and its inverse is denoted by 
$\Lambda_{h,s}^{-1}$. 

$ii)$ Moreover, for any $s'\le s$, $h^{s'} \Lambda_{h,s}^{-1}$ is uniformly bounded from $L^2$ to $H^{s'}$: 
there is $K>0$ such that for any $h\in (0,1]$ and any $u$ in $L^2(\xT)$, 
\be\label{pna0}
\lA h^{s'} \Lambda_{h,s}^{-1} u\rA_{H^{s'}}\le K \lA u\rA_{L^2}.
\ee
\end{lemm}
\begin{proof}
Set $r=2s/3$. Statement $i)$ is obtained writing $\Lambda_{h,s}$ as $(I+B)(I+h^{s}L^r)$ where $B$ is a bounded operator from $L^2$ into itself. To do so, write
$$
\Lambda_{h,s}=I+h^{s}T_{c^r}L^r=
I+h^{s}L^r+h^sT_{c^r-1}L^r,
$$
to obtain the desired result with 
$B\defn h^s T_{c^r-1}L^r (I+h^sL^r)^{-1}$. 
We now claim that $B$ is a bounded operator on $L^2$, 
with operator norm $O(\lA c-1\rA_{L^\infty})$. This follows easily from \e{esti:quant0} 
(which implies that $T_{c^r-1}$ is of order $0$ with operator norm $O(\lA c-1\rA_{L^\infty})$) and, on the other hand, from the fact that 
the operator $h^{s}L^{r}\big(I+h^{s}L^{r}\big)^{-1}$ is 
bounded from $L^2$ into itself uniformly in $h$ (as can be verified 
using the Fourier transform). 

Now for $\lA c-1\rA_{L^\infty}$ small enough, one has $\lA B\rA_{\Lr(L^2)}\le 1/2$ and 
one can invert $(I+B)$ to obtain
\be\label{pna19}
\Lambda_{h,s}^{-1}=(I+h^{s}L^{r})^{-1}(I+B)^{-1}
\ee
and statement $ii)$ follows from the fact that $h^{s'}(I+h^{s}L^{r})^{-1}$ is uniformly bounded in $\mathcal{L}(L^2;H^{s'})$ for $s'\le s$. 
\end{proof}

The key property is that one has good estimates for the commutators of 
$\Lambda_{h,s}$ and the various operators appearing in the equation. 

\begin{lemm}\label{P:L42}
Assume that the $\eC{\tdm}$-norm of $c-1$ is small enough. Then there is $K>0$ such that for any $h\in (0,1]$ and any $u$ in $L^2(\xT)$, there holds
\begin{align}
&\lA \big[ \Lambda_{h,s}, T_V\px\big] \Lambda_{h,s}^{-1}u\rA_{L^2}\le K 
\lA V\rA_{\eC{1}}\lA u\rA_{L^2},\label{pna1}\\
&\blA \big[ \Lambda_{h,s}, \chi_\omega \big] \Lambda_{h,s}^{-1}u\brA_{L^2}\le K h 
\lA \chi_\omega\rA_{H^{s+1}}\lA u\rA_{L^2},\label{pna2}\\
&\blA \big[ \Lambda_{h,s}, L^\mez (T_c L^\mez\cdot) \big] \Lambda_{h,s}^{-1}u\brA_{L^2}\le K h 
\lA u\rA_{L^2}.\label{pna3}
\end{align}
\end{lemm}
\begin{proof}
Write
$$
\big[ \Lambda_{h,s}, T_V\px\big] \Lambda_{h,s}^{-1}=
\big[ T_{(c\ell)^{2s/3}}, T_V\px\big]h^s \Lambda_{h,s}^{-1},
$$
to obtain
$$
\blA \big[ \Lambda_{h,s}, T_V\px\big] \Lambda_{h,s}^{-1}u\brA_{L^2}
\le K \lA  \big[ T_{(c\ell)^{2s/3}}, T_V\px\big]\rA_{\mathcal{L}(H^s,L^2)}
\blA  h^s \Lambda_{h,s}^{-1}\brA_{\mathcal{L}(L^2,H^s)}
\lA u\rA_{L^2}.
$$
It follows from \e{pna0} that $\blA  h^s \Lambda_{h,s}^{-1}\brA_{\mathcal{L}(L^2,H^s)}$ is uniformly bounded in $h$. On the other hand, the commutator estimate~\e{esti:quant4} implies that
$$
\blA  \big[ T_{(c\ell)^{2s/3}}, T_V\px\big]\brA_{\mathcal{L}(H^s,L^2)}\le 
K\lA V\rA_{\eC{1}},
$$
where $K$ depends on $\lA c\rA_{\eC{\tdm}}$ (which by assumption can be bounded by $2$). 

To estimate the second commutator, we begin by establishing that
\be\label{pna5}
\blA \big[ \Lambda_{h,s}, T_{\chi_\omega} \big] \Lambda_{h,s}^{-1}u\brA_{L^2}\le K h 
\lA \chi\rA_{H^s}\lA u\rA_{L^2}.
\ee
To see this write
$$
\big[ \Lambda_{h,s}, T_{\chi_\omega} \big] \Lambda_{h,s}^{-1}=
h\big[ T_{(c\ell)^{2s/3}}, T_{\chi_\omega} \big]h^{s-1} \Lambda_{h,s}^{-1}.
$$
Then we notice that, as above,
$$
\lA \big[ T_{(c\ell)^{2s/3}}, T_{\chi_\omega} \big]\rA_{\mathcal{L}(H^{s-1},L^2)}\le 
K \lA \chi_\omega\rA_{\eC{1}},
$$
and we use that, thanks to \e{pna0}, 
$h^{s-1} \Lambda_{h,s}^{-1}$ is uniformly bounded from $L^2$ to $H^{s-1}$. 

Now it remains to estimate 
$\big[ \Lambda_{h,s}, (\chi_\omega-T_{\chi_\omega}) \big]$. 
It follows from 
Proposition~\ref{lemPa} (applied with $(r,\mu,\gamma)=(s+1,0,s)$) that
\begin{align*}
\lA h^{s}T_{(c\ell)^{2s/3}} (\chi_\omega-T_{\chi_\omega})
\Lambda_{h,s}^{-1}u\rA_{L^2}&\les h^{s} \lA (\chi_\omega-T_{\chi_\omega})
\Lambda_{h,s}^{-1}u\rA_{H^s}\\
&\les h^{s} \lA \chi_\omega\rA_{H^{s+1}}\blA \Lambda_{h,s}^{-1} u\brA_{L^2}
\les h^{s}\lA \chi_\omega\rA_{H^{s+1}}\lA u\rA_{L^2},
\end{align*}
and similarly
$$
\lA  (\chi_\omega-T_{\chi_\omega})h^{s}T_{(c\ell)^{2s/3}}
\Lambda_{h,s}^{-1}u\rA_{L^2}\le K h^{s} 
\lA \chi_\omega\rA_{H^{s+1}}\lA u\rA_{L^2}.
$$
By combining these two estimates, we find that
\be\label{pna6}
\lA \big[ \Lambda_{h,s}, (\chi_\omega-T_{\chi_\omega}) \big] \Lambda_{h,s}^{-1}u\rA_{L^2}\le K h^{s}
\lA \chi_\omega\rA_{H^{s+1}}\lA u\rA_{L^2}.
\ee
By combining \e{pna5} and \e{pna6} we deduce \e{pna2}. 

We now prove the last property \e{pna3}. Write that 
$L^\mez (T_c L^\mez\cdot)=T_{c\ell}+T_{\wp}+R$ where 
$R$ is of order $0$ and 
$\wp=i^{-1} \sqrt{\ell}(\partial_\xi \sqrt{\ell})(\partial_x c)$. Since 
$\Lambda_{h,s}=I+h^sT_{(c\ell)^{2s/3}}$, by definition, $[\Lambda_{h,s},L^\mez (T_c L^\mez\cdot)] \Lambda_{h,s}^{-1}$ can be written as the sum 
$(I)+(II)+(III)$ with 
\begin{align*}
&(I)\defn \big[T_{(c\ell)^{2s/3}},T_{c\ell}\big] h^s \Lambda_{h,s}^{-1},
\quad 
(II)\defn 
\big[T_{(c\ell)^{2s/3}},T_{\wp}\big] h^s \Lambda_{h,s}^{-1},
\\
&(III)\defn 
\big[T_{(c\ell)^{2s/3}},R\big]h^s \Lambda_{h,s}^{-1}.
\end{align*}
Since $h^s \Lambda_{h,s}^{-1}$ belongs to $\Lr(L^2,H^s)$ uniformly in $h$, we need only 
estimate 
$$
\blA \big[T_{(c\ell)^{2s/3}},T_{c\ell}\big]\brA_{\Lr(H^s,L^2)},\quad 
\blA \big[T_{(c\ell)^{2s/3}},T_{\wp}\big] \brA_{\Lr(H^s,L^2)}.
$$
The second term is estimated by means of \e{esti:quant2sharp} applied with $\rho=1/2$. 
To estimate the first term we notice that the Poisson bracket of the symbols vanishes:
$$
\big\{ (c\ell)^{2s/3},c\ell\big\}=\frac{1}{i}\big( 
(\partial_\xi (c\ell)^{2s/3})\partial_x(c\ell)-(\partial_x (c\ell)^{2s/3}) \partial_\xi^\a (c\ell)\big)=0.
$$
Since $\lA c\rA_{\eC{\tdm}}\le 2$ by assumption, 
it follows from \e{esti:quant2sharp} applied with $\rho=3/2$ that 
$$
\blA \big[T_{(c\ell)^{2s/3}},T_{c\ell}\big]\brA_{\Lr(H^s,L^2)}\les 1.
$$
This completes the proof.
\end{proof}

Next we conjugate $P$ with $\Lambda_{h,s}$. Introduce
$$
\widetilde{P}_h\defn \Lambda_{h,s} P \Lambda_{h,s}^{-1}.
$$
Then
\begin{align*}
\widetilde{P}_h&=\partial_t  +T_V\partial_x +iL^\mez (T_c L^\mez\cdot)+\Run \qquad \text{where}\\
\Run^h&\defn \Lambda_{h,s} R\Lambda_{h,s}^{-1}
+ \bigl[\Lambda_{h,s},\partial_t\big]\Lambda_{h,s}^{-1}
+  \bigl[\Lambda_{h,s},T_V\px\big] \Lambda_{h,s}^{-1}
+ i\bigl[\Lambda_{h,s},L^\mez (T_c L^\mez\cdot)\big] \Lambda_{h,s}^{-1}.
\end{align*}

\begin{lemm}\label{PP10.5}
Assume that the $\eC{\tdm}$-norm of $c-1$ is small enough. There holds
\be\label{pna10}
\blA \Run^h u\brA_{L^2}\le K \big(\lA V\rA_{\eC{1}}+
\lA \partial_t c\rA_{L^\infty}
+h^{-s}\lA R\rA_{\Lr(H^s)}\big)\lA u\rA_{L^2},
\ee
for some constant $K$ independent of $h$.
\end{lemm}
\begin{rema}The constant $h^{-s}$ is harmless since 
at the end of this section, $h$ will be fixed depending only on $T$.
\end{rema}
\begin{proof}
We have
$$
\blA \Lambda_{h,s} R\Lambda_{h,s}^{-1}\brA_{\Lr(L^2)}\le 
\blA \Lambda_{h,s} \brA_{\Lr(H^s;L^2)}
\blA R\brA_{\Lr(H^s;H^s)}
\blA \Lambda_{h,s}^{-1}\brA_{\Lr(L^2;H^s)}
\le K h^{-s}  \blA R\brA_{\Lr(H^s;H^s)},
$$
since $\blA \Lambda_{h,s} \brA_{\Lr(H^s;L^2)}\les 1$ and 
$\blA \Lambda_{h,s}^{-1}\brA_{\Lr(L^2;H^s)}\les h^{-s}$. 

On the other hand, $\bigl[\Lambda_{h,s},L^\mez (T_c L^\mez\cdot)\big] \Lambda_{h,s}^{-1}$ and  
$\bigl[\Lambda_{h,s},T_V\px\big] \Lambda_{h,s}^{-1}$ 
are estimated by means of Lemma~\ref{P:L42} and $\bigl[\Lambda_{h,s},\partial_t\big]\Lambda_{h,s}^{-1}$ is estimated by similar arguments. 
\end{proof}

We further transform the equation by replacing $T_V\px$ and 
$L^\mez (T_c L^\mez\cdot)$ by $V\px$ and $L^\mez(cL^\mez\cdot)$ 
modulo remainder terms. Namely, write $\widetilde{P}_h$ as
\be\label{Pew2}
\widetilde{P}_h\defn \partial_t +V\px +iL^\mez \big( c L^\mez \cdot\big) +\Rdeux^h
\ee
where $c$ stands for the multiplication operator by $c$ and 
$$
\Rdeux^h u=\Run^h u+T_V\px u -V\px u+i\big(L^\mez T_cL^\mez u 
-L^\mez \big(c L^\mez u\big) \big).
$$
\begin{lemm}\label{PP10.5deux}
Let $s_0>2$ and assume that the $\eC{\tdm}$-norm of $c-1$ is small enough. 
There holds
\be\label{pna15}
\blA \Rdeux^h u\brA_{L^2}\le K \big(\lA V\rA_{H^{s_0}}+\lA c-1\rA_{H^{s_0}}
+\lA \partial_t c\rA_{H^{1}} 
+h^{-s}\lA R\rA_{\Lr(H^s)}\big)\lA u\rA_{L^2},
\ee
for some constant $K$ independent of $h$.
\end{lemm}
\begin{proof}We have already estimated $\Run^h$, and the right-hand side 
of \e{pna10} is less than the one of \e{pna15} provided that $s_0>3/2$. To estimate 
$T_V\px u -V\px u$, we apply 
Proposition~\ref{lemPa} with $(r,\mu,\gamma)=(s_0,-1,0)$ (and $s_0>3/2$) 
to obtain
$$
\lA T_V\px u -V\px u\rA_{L^2}\les \lA V\rA_{H^{s_0}}\lA \px u\rA_{H^{-1}}\le 
\lA V\rA_{H^{s_0}}\lA u\rA_{L^2}.
$$
The estimate for $\mathcal{L}-L^\mez\big(cL^\mez\cdot\big)=L^\mez\big((T_c-cI)L^\mez\cdot\big)$ 
follows in the same way, assuming that $s_0>2$. 
\end{proof}

We are now ready to give the main reduction. Our goal in this section is to prove that one can deduce Proposition~\ref{PP10} from 
the following proposition. 

\begin{prop}\label{PP11}
Consider an operator of the form
\be\label{Pew1}
\widetilde{P}\defn \partial_t +V\px +iL^\mez \big( c L^\mez \cdot\big) +\Rdeux.
\ee
Let $T\in (0,1]$ and consider an open subset $\omega\subset \xT$. 
There exist an integer $s_0$ large enough and two positive constants 
$\delta=\delta(T)$ and $K=K(T)$ 
such that, if
\be\label{pn200b}
\begin{aligned}
&\lA V\rA_{\bo([0,T];H^{s_0})}+\lA c-1\rA_{\bo([0,T];H^{s_0})}\le \delta, \\
& \blA \partial_t^k V\brA_{\bo([0,T];H^1)}
+\blA \partial_t^k c\brA_{\bo([0,T];H^1)}\le \delta \qquad (1\le k\le 3),\\
&\lA \Rdeux\rA_{\bo([0,T];\mathcal{L}(L^2))}\le \delta,
\end{aligned}
\ee
then  
for any initial data 
$v_{in}\in L^2(\xT)$  
there exists $f\in C^0([0,T];L^2(\xT))$ such that: 
\begin{enumerate}
\item 
the unique solution $v$ to $\widetilde{P}v=\chi_\omega \RE f,\quad v_{\arrowvert t=0}=v_{in}$ is such that 
$v(T)$ is an imaginary constant:
$$
\exists b\in \xR/~\forall x\in \xT, \quad v(T,x)=ib.
$$
\item \label{PP10-prop2} $\lA f\rA_{\bo([0,T];L^2)}\le K\lA v_{in}\rA_{L^2}$. 
\end{enumerate}
\end{prop}
\begin{rema}
Notice that the final state $v(T)$ is not $0$ but an imaginary constant. 
\end{rema}

This result will be proved later. In the end of this section, 
we assume that Proposition~\ref{PP11} is true and prove Proposition~\ref{PP10}. 

\begin{proof}[Proof of Proposition~\ref{PP10} given Proposition~\ref{PP11}] 
Proposition~\ref{PP11} holds for any $\wide P$ of the form 
\e{Pew1}. In particular, in view of \e{Pew2}, it holds for 
$\wide P$ replaced 
by $\widetilde{P}_h\defn \Lambda_{h,s} P \Lambda_{h,s}^{-1}$. 
Let us mention that $h$ will be fixed at the end of the proof 
by asking that $K'(T)h<1/4$ where $K'(T)$ depends only on $T$. 

The idea is to apply control property for $\widetilde{P}_h$ 
associated with an unknown initial data to be determined. 

We shall prove that Proposition~\ref{PP11} implies that Proposition~\ref{PP10} holds with the conclusion 
$\ref{PP10-prop1}$ replaced by $v(T)\in i\xR$. 
Then one deduces that $v(T)=0$ by using condition~$\ref{T31A-cond3})$ in Assumption \ref{T31A} and the fact that $\int_{\xT}v_{in}(x)\, dx=0$.

Assume that $\wide\delta \le h^s\delta$ 
where $\wide\delta$ appears in 
the statement of Proposition~\ref{PP10}Ê
and $\delta$ is as given by 
Proposition~\ref{PP11}. Then one has 
$h^{-s}\wide\delta \le \delta$. 
Therefore, if the smallness condition \e{pn200} holds, then Lemma~\ref{PP10.5} 
implies that $\lA \Rdeux^h\rA_{\bo([0,T];\mathcal{L}(L^2))}$ is small, and hence the 
smallness assumption \e{pn200b} holds. This explains why one may apply the conclusion of Proposition~\ref{PP11} under the assumption of Proposition~\ref{PP10}. 

The assumption that Proposition~\ref{PP11} holds implies that for any $y\in L^2(\xT)$ 
there is $\widetilde{f}\in C^0([0,T];L^2(\xT))$ satisfying
\be\label{n210}
\blA \widetilde{f}\brA_{\bo([0,T];L^2)}\le K(T) \lA y\rA_{L^2},
\ee
and such that the unique solution $u_1$ to
$$
\widetilde{P}_h u_1=\chi_\omega \RE \widetilde{f},\quad u_1 |_{t=0}=y,
$$
is such that $u_1(T,x)=ib$ for some $b\in\xR$ and all $x\in \xT$.

Now introduce $u_2$ which is the unique solution of the Cauchy problem (with data at time $T$), 
$$
\widetilde{P}_h u_2=\big( \Lambda_{h,s}T_p \chi_\omega \Lambda_{h,s}^{-1}-\chi_\omega\big)\RE\widetilde{f},\quad u_2(T)=0.
$$
Again, the fact that the above Cauchy problem has a unique solution follows from Proposition~\ref{P:10}. One can then define 
the linear operator $\opk$ by
\be\label{n211}
\opk y=u_2(0).
\ee 
The reason to introduce $u_2$ and the operator $\opk$ is that the function $u$ defined by 
$u\defn u_1+u_2$ satisfies
$$
\widetilde{P}_h u=\Lambda_{h,s} \big(T_p\chi_\omega 
\Lambda_{h,s}^{-1}
\RE\widetilde{f}\big),\quad u(T)=ib,\quad u_{|t=0}=y+\opk y.
$$
Now, assume that $I+\opk$ is invertible with $(I+\opk)^{-1}\in \mathcal{L}(L^2)$. Then 
$y$ can be so chosen that $y+\opk y=\Lambda_{h,s}v_{in}$. Using that 
$\Lambda_{h,s}b=b$ and hence $\Lambda_{h,s}^{-1}b=b$ for any constant $b$, it follows that, with $f\defn \Lambda_{h,s}^{-1}\widetilde{f}$ and 
$v\defn \Lambda_{h,s}^{-1}u$,
$$
Pv=T_p\chi_\omega \RE f,\quad v(T)=ib,\quad v(0)=v_{in},
$$
where $P$ is the original operator, so that $\widetilde{P}_h=\Lambda_{h,s}P\Lambda_{h,s}^{-1}$. 
Moreover, it follows from the conclusion $\ref{PP10-prop2}$ of 
Proposition~\ref{PP11} that
$\blA \widetilde{f}\brA_{\bo([0,T];L^2)}\le K \lA y\rA_{L^2}$ from which 
we deduce that 
$\lA f\rA_{\bo([0,T];H^s)}\le K(h) \lA y\rA_{H^s}$. The fact that the last constant depends on $h$ is not a problem since $h$ is fixed, depending on $T$. Now to see that 
Proposition~\ref{PP10} holds, it remains to check that 
$v(T)=0$. As already mentioned, 
the fact that $v(T)=0$ follows from the fact that $v(T)=ib$ as well as condition~$\ref{T31A-cond3})$ in Assumption \ref{T31A} and the fact that $\int_{\xT}v_{in}(x)\, dx=0$. 

Thus it remains only to prove that $I+\opk$ is a bijection from 
$L^2$ into itself. To see this, it is sufficient to prove that 
$\opk$ is a bounded operator whose operator norm 
in $\mathcal{L}(L^2)$ is strictly smaller than $1$. In this direction, we first use the energy 
estimate \e{pn27b} for the operator $\widetilde{P}$:
$$
\lA u_2(t)\rA_{L^2}\le e^{CT}
\left( \lA u_2(T)\rA_{L^2}+\int_0^T \blA \widetilde{P}u_2 \brA_{L^2}\, dt'\right),
$$
for some constant $C$ depending only on 
$$
M_{s_0}\defn \sup_{t'\in [0,T]}\big\{\lA V(t')\rA_{H^{s_0}}+\lA c(t')-1\rA_{H^{s_0}}
+\lA \Rdeux(t')\rA_{\mathcal{L}(L^2)}\big\}.
$$
Since $u_2(T)=0$, this implies that,
$$
\lA u_2\rA_{\bo([0,T];L^2)}\le e^{CT}
\int_0^T \blA\big( \Lambda_{h,s}T_p \chi_\omega \Lambda_{h,s}^{-1}-\chi_\omega\big)\widetilde{f}(t')
\brA_{L^2}\, dt'.
$$
To estimate the term $\big( \Lambda_{h,s}T_p \chi_\omega \Lambda_{h,s}^{-1}-\chi_\omega\big)\widetilde{f}$ we write it as
$$
[ \Lambda_{h,s},\chi_\omega]\Lambda_{h,s}^{-1}\widetilde{f}+
\Lambda_{h,s}(T_p-I) \chi_\omega \Lambda_{h,s}^{-1}\widetilde{f}.
$$
It follows fromÊ
\e{pna2} that
$$
\blA [ \Lambda_{h,s}, \chi_\omega]\Lambda_{h,s}^{-1}\widetilde{f} \brA_{L^2}
\leq K h \lA \chi_\omega \rA_{H^{s+1}}\blA \widetilde{f} \brA_{L^2}.
$$
It remains to estimate $ \Lambda_{h,s}(T_p-I) \chi_\omega \Lambda_{h,s}^{-1}\widetilde{f}$. 
To do so, we write $\Lambda_{h,s}=I+h^sT_{c^{(2s)/3}}L^{\frac{2s}{3}}$ to split this term as
$$
(T_p-I) \chi_\omega\Lambda_{h,s}^{-1}\widetilde{f}+T_{c^{(2s)/3}}L^{\frac{2s}{3}}(T_p-I) \chi_\omega (h^s \Lambda_{h,s}^{-1})\widetilde{f}.
$$
For the first term we have (using \e{esti:quant1} and \e{pna0} with $s'=0$)
$$
\blA (T_p-I) \chi_\omega \Lambda_{h,s}^{-1}\widetilde{f}\brA_{L^2}\les 
M^0_0(p-1)\lA \chi_\omega\rA_{L^\infty}\blA  \widetilde{f}\brA_{L^2}.
$$
For the second term write (using \e{esti:quant0}, \e{prtame2} 
and \e{pna0} with $s'=s$)
\begin{align*}
\blA T_{c^{(2s)/3}}L^{\frac{2s}{3}}(T_p-I) \chi_\omega (h^s \Lambda_{h,s}^{-1})\widetilde{f}\brA_{L^2}
&\les \blA  (T_p-I) \chi_\omega (h^s \Lambda_{h,s}^{-1})\widetilde{f}\brA_{H^s}\\
&\les \lA T_{p-1}\rA_{\Lr(H^s)}\lA \chi_\omega\rA_{H^s}\blA h^s \Lambda_{h,s}^{-1}\widetilde{f}\brA_{H^s}\\
&\les M^0_0(p-1)\lA \chi_\omega\rA_{H^s}\blA \widetilde{f}\brA_{L^2}.
\end{align*}
It is found that
$$
\blA \Lambda_{h,s}(T_p-I) \chi_\omega \Lambda_{h,s}^{-1}\widetilde{f}\brA_{L^2}
\les  \big(\lA c-1\rA_{L^\infty}+\lA \px c\rA_{L^\infty}\big) \lA \chi_\omega\rA_{H^{s}}\blA \widetilde{f}\brA_{L^2}
\les \wide\delta \blA \widetilde{f}\brA_{L^2}.
$$
This yields 
$$
\lA u_2\rA_{\bo([0,T];L^2)}\les (h+\wide\delta) e^{CT}
\int_0^T \blA \widetilde{f}\brA_{L^2}\, dt'.
$$
In view of \e{n210}, we conclude that
$$
\lA u_2\rA_{\bo([0,T];L^2)}\le K'(T)(h+\wide\delta)\lA y\rA_{L^2},
$$
for some constant $K'(T)$. Then chose $h,\wide\delta$ such that 
$K'(T)h<1/4$, $K'(T)\wide\delta<1/4$.  
We conclude that
\be\label{n217}
\forall t\in [0,T], \quad \lA u_2(t)\rA_{L^2}\le \mez \lA y\rA_{L^2}.
\ee
By applying this inequality with $t=0$, one obtains $\lA \opk y\rA_{L^2}\le \mez\lA y\rA_{L^2}$ which proves that 
$I+\opk$ is invertible in $\mathcal{L}(L^2)$. This completes the proof of Proposition~\ref{PP10}.
\end{proof}

\section{Further reductions}\label{S:22}

Recall that until now we have reduced the study of the control problem in Sobolev spaces for
$P\defn \partial_t  +T_V\partial_x +iL^\mez T_c L^\mez+ R$
to the one of the control problem in $L^2$ for
$\widetilde{P}=\partial_t +V\px +i L^\mez (c L^\mez \cdot)+\Rdeux$.

\subsection{Change of variables}

The goal of this subsection is to reduce the analysis to an equation 
where $ L^\mez (c L^\mez \cdot)$ is replaced with an operator 
with constant coefficient. To do so, we use three change of variables, 
which preserves the $L^2(dx)$ scalar product. 
This allows us to conjugate $\widetilde{P}$ to an operator of the form
$$
\pa_t + W \pa_x + i L+ R
$$
where $R$ is of order $0$ and furthermore 
$W = W(t,x)$ satisfies $\int_{\xT} W(t,x)\, dx=0$.

\begin{prop}\label{P:38}
There exist universal constants $\delta_0 \in (0,1)$, $r \geq 2$, $C > 0$ such that the following properties hold. 
Assume that $c,V,R_2$ satisfy
\be \label{basic smallness}
\| c - 1 \|_{\NormP{T}{L^\infty}} < \delta_0, \quad 
\mN_0 \leq 1,
\ee
where 
\[
\mN_0 := \| c-1 \|_{\NormP{T}{H^r}} + \| V \|_\Norm{T}{H^1}
+ \| \pa_t c \|_\Norm{T}{H^1} + \| R_2 \|_\Norm{T}{\mL(L^2)}.
\]
Then there exist a constant $T_1 > 0$ and a bounded, invertible linear map 
\[
\Phi \colon C^0([0,T];L^2(\xT)) \to C^0([0,T_1];L^2(\xT))
\]
with bounded inverse $\Phi^{-1}$ such that 
\[
\wide P u = m \Phi^{-1}\big( \wide P_3 ( \Phi u )\big), 
\]
where $m = m(t)$ is a function of time only, defined for $t \in [0,T]$, and 
\[
\wide P_3 = \pa_t + W \pa_x + i L + \Rtrois.
\]
The function $W = W(t,x)$ is defined for $t \in [0,T_1]$,  
it satisfies $\int_{\xT} W(t,x)\, dx=0$, and 
\be\label{p214}
\| W \|_\Norm{T_1}{H^2} \leq C \big( \| (c-1,V) \|_\Norm{T}{H^2} 
+ \| \pa_t c \|_\Norm{T}{H^1} 
\big).
\ee
The operator $\Rtrois$ maps $C^0([0,T_1] ; L^2(\T))$ into itself, with 
\be\label{p215}
\| R_3 \|_{\Norm{T_1}{\mL(L^2)}} \leq C \mN_0.
\ee
The constant $T_1$ and the function $m$ satisfy
\[
\Big| \frac{T_1}{T}\, - 1 \Big| + \| m - 1 \|_{{C^0([0,T])}} 
\leq C \| c-1 \|_\Norm{T}{L^\infty}.
\]
The map $\Phi$ is the composition of three local transformations 
$\Phi = \ph_*^{-1} \psi_*^{-1} \Psi_1$, where 
\be\label{S5n15}
\begin{aligned}
(\Psi_1 h)(t,x) & := (1 + \pa_x \tilde\b_1(t,x))^{\frac12} h(t,x + \tilde\b_1(t,x)),
\\
(\psi_*^{-1} h)(t,x) & := h(\psi^{-1}(t),x), 
\qquad (\ph_*^{-1} h)(t,x) := h(t,x-p(t)).
\end{aligned}
\ee

\end{prop}

\begin{proof}
This proposition is proved in Appendix \S\ref{S:AC1}.
\end{proof}

\begin{rema}
$i)$ The proof is based on computations similar to the ones used 
in \cite{AB}. However, the analysis in \cite{AB} used some special properties of the Hilbert transform which cannot be applied in the present setting. Instead, 
we shall rely on Egorov theorem. 
Moreover, we need in this paper to introduce a change of variables which preserves 
the skew-symmetric structure of the operator $iL^{\mez}(cL^\mez\cdot)$. This allows us to prove that 
some operator of order $1/2$ vanishes, which plays an essential role below. This in turn forces us to revisit the analysis of changes of variables, which explains why the proof is done 
in details in \S\ref{S:AC1}. 

$ii)$ In sharp contrast with other transformations that will be performed below, notice that a change of variable is a local transformation, thereby transforming a localized control in another localized control 
(this is used below to prove Lemma~\ref{P:51-w}).
\end{rema}

In addition to Proposition \ref{P:38}, 
higher regularity and stability estimates are given in Proposition \ref{P:38 App C}.

\subsection{Conjugation}\label{S:42}

To study the control problem for 
the new equation
$$
\partial_t +W\px +iL +\Rtrois
$$
we will use 
the HUM method. A key point is then to prove an observability 
inequality for solutions of the dual equation, which reads
$$
(-\partial_t -\px (\V \cdot)   -iL+\Rtrois^*)w=0.
$$
This equation can be written under the form $\mathcal{P}w=0$ with
$$
\mathcal{P}w\defn \partial_t w+W\px w+iL w+\Rquatre w,
$$
where
\be\label{p229}
\Rquatre w\defn -\Rtrois^*w+(\px \V)w.
\ee
The observability inequality will be proved later. 
As a preparation, in this section, 
we prove that $\mathcal{P}$ is conjugated to 
a simpler operator where $\partial_t w+W\px w$ 
is replaced by $\partial_t$. 
To do so, we use the analysis in \cite{AB}. For the sake of completeness, we recall the strategy and the main steps of the proof.

We often use below the following notation: given a function $f$ with zero mean, 
$\px^{-1}f$ is the zero-mean primitive of $f$, defined by
$$
\px^{-1}f=\sum_{j\neq 0}\frac{f_j}{ij}e^{ijx},\quad f(x)=\sum_{j\neq 0} f_je^{ijx}.
$$
We seek an operator $A$ such that
$$
\big(\partial_t +W\px+iL +\Rquatre\big) A  
=A   \big(   \partial_t +iL+\Rcinq\big)
$$
where $\Rcinq$ 
is a remainder term of order $0$. By definition
\be\label{n230}
\Rcinq\defn A^{-1}\Big(\left[ \partial_t,A\right] + \Rquatre A 
+W\px  A+i\left[ L,A \, \right] \Big).
\ee
Seeking $A$ as 
a pseudo-differential operator, and trying 
to cancel the leading order terms (that is 
$W\px A+i\left[L \,,A \right]$), it is natural to introduce 
$A$ as follows. Let 
\[
\phi(t,x,k) := k x + \beta(t,x)|k|^\mez,
\]
for some function $\beta$ to be determined. Consider also 
an amplitude $q(t,x,k)$ to be determined. Then 
define the operator $A(t)$ by setting
\be\label{n242}
A u(t,x) = \sum_{k \in \xZ} \hat u_k \, q(t,x,k) e^{i\phi(t,x,k)},
\ee
for periodic functions $u$, where $\hat u_k(t)$ are the Fourier coefficients of $u$, so that
$u(t,x) = \sum_{k \in \xZ} \hat u_k(t)e^{ikx}$. 

Below $t$ is seen as a parameter and we omit it in most expressions. 

\begin{assu}\label{pA:41}
Set
$$
\Nr\defn \lA V\rA_{\bo([0,T];H^{s_0})}+\lA c-1\rA_{\bo([0,T];H^{s_0})}+
\blA \partial_t c\brA_{\bo([0,T];H^1)}
+\lA \Rdeux\rA_{\bo([0,T];\mathcal{L}(L^2))},
$$
where $s_0$ is some fixed large enough integer. 
In this section, we always assume that $\Nr$ is small enough without recalling this assumption in all the statements. 
\end{assu}

Hereafter, $s_0$ always refers to an index large enough whose value may vary from one statement to another.

\begin{lemm}[from \cite{AB}]
\label{L:26}
There exists a universal constant $\delta > 0$ with the following properties. 

$i)$ Consider the case when the amplitude $q$ is of order zero in $k$ and is a perturbation of 1,  
\[
q(x,k) = 1 + b(x,k)\,.
\]
Denote $|b|_s :=\sup_{k \in \xZ} \| b(\cdot\,,k) \|_{H^s(\xT)}$. 
If 
\[
\| \beta \|_{H^3} + |b|_{3} \leq \delta\,,
\]
then $A$ and $A^*$ are invertible from $L^2(\xT)$ onto itself, with 
\[
\| A u \|_{L^2} + \| A^{-1} u \|_{L^2} + \| A^*\, u \|_{L^2} + 
\| (A^*)^{-1}u \|_{L^2} \leq C \, \| u \|_{L^2} \,,
\]
where $C>0$ is a universal constant. 

$ii)$ Consider the case when the amplitude $q$ is small.That is, assume that
\[
\| \beta \|_{H^3} + |q|_{3} \leq \delta\,,
\]
then
\[
\| A u \|_{L^2} \le  C\delta  \, \| u \|_{L^2} \,,
\]
where $C>0$ is a universal constant. \end{lemm}

\begin{prop}[from \cite{AB}]\label{L27}
Assume that $\|\beta\|_{W^{1,\infty}} \leq 1/4$ and $\| \beta \|_{H^2} \leq 1/2$. 
Let 
\[
r,m,s_0 \in \xR, \quad 
m \geq 0, \quad
s_0 > 1/2, \quad 
M \in \xN, \quad 
M \geq 2(m + r + 1) + s_0. 
\] 
Then 
\[
|D_x|^r A u = \sum_{\alpha=0}^{M-1} \Op \left( \frac{1}{i^\alpha \alpha!}
\left(\partial_\xi^\alpha |\xi|^r\right)\px^\alpha \left(q(x,k) e^{i |k|^\mez \beta(x)}\right)\right)
u
+ R_M u,
\]
where, for every $s \geq s_0$, the remainder satisfies
\begin{equation}\label{n250}
\| R_M \Dx^m u \|_{H^s} 
\leq 
C(s) \Big\{ \mathcal{K}_{2(m+r+s_0+1)} \, \| u \|_{H^s} 
+\mathcal{K}_{s + M + m + 2} \, \| u \|_{H^{s_0}} \Big\},
\end{equation}
where $\mathcal{K}_\mu := | q-1 |_{\mu} + | q |_1  \| \beta \|_{H^{\mu + 1}}$ and $| q |_\mu := \sup_t \sup_{k \in \xZ} \| q(t,\cdot,k) \|_{H^\mu}$. 
\end{prop}

We now deduce the following result (which is a variant of a result 
proved in \cite{AB} with a slightly different estimate for the remainder).

\begin{coro}\label{CoroR4}
There exists a universal constant $\delta > 0$ with the following property. 
Assume that 
\[
| q-1 |_{14} + \| \beta \|_{H^{14}} \leq \delta,
\]
and let $A$ be the operator $A\defn \Op\Big(q(x,\xi)e^{i|k|^\mez \beta(x)}\Big)$. 
For any $u$ in $L^2$, there holds
\begin{multline}\label{p251a}
i [ |D_x|^{\frac32}, A ] u \\ = \frac32 (\pa_x \beta) \pa_x (Au) 
+ \Op \Big( \Big( \frac32 \frac{\xi}{|\xi|}\pa_x q 
- \frac{9i}{8}(\px\beta)^2  q \Big) |\xi|^{\frac12} e^{i|\xi|^{\frac12} \beta} \Big) u 
+ R_A u
\end{multline}
where $R_A$ satisfies 
\[
\| R_A u \|_{L^2} \leq C \delta \| u \|_{L^2}.
\]
\end{coro}
\begin{proof}Denote by $p$ the symbol 
$p=q(x,\xi)e^{i|k|^\mez \beta(x)}$. 
Set $M=8$ and write
\be\label{p255a}
i\Dx^\tdm A=\Op \Big(\sum_{\alpha=0}^{2}\frac{i}{i^\alpha \alpha!} \partial_\xi^\alpha |\xi|^\tdm \, 
\partial_x^\alpha p \Big)+R_0+R_M,
\ee
where 
$$
R_0\defn \Op \Big(\sum_{\alpha=3}^{M-1}\frac{i}{i^\alpha \alpha!} \partial_\xi^\alpha |\xi|^\tdm \, 
\partial_x^\alpha p \Big).
$$
For any $3\le \alpha\le M-1$, the symbol $\partial_\xi^\alpha |\xi|^\tdm \, 
\partial_x^\alpha p$ is a linear combination of terms of the form $m(x,\xi)b(x) e^{i|k|^\mez \beta}$ 
where $m$ is of order $0$ (that is $\partial_\xi^l m(x,\xi)\les |\xi|^{-l}$) 
and $b(x)$ is of the form 
$(\px^{\alpha_0}q)(\px^{\alpha_1}\beta)\cdots (\px^{\alpha_m} \beta)$. 
It follows from statement $ii)$ in Lemma~\ref{L:26} that $R_0$ is an operator or order $0$, satisfying
$$
\lA R_0 u\rA_{L^2}\le C(\delta)\delta \lA u\rA_{L^2}.
$$

We now estimate the operator norm of $R_M$. By applying \e{n250} with $s=s_0=1$, $m=1$ and $M=8$, 
then $ \mathcal{K}_{2(m+r+s_0+1)} \le \mathcal{K}_{s + M + m + 2}=\mathcal{K}_{12}$ 
and the inequality simplifies to
$$
\forall u\in L^2(\xT),\quad 
\| R_M \Dx u \|_{H^1} 
\leq C (1)\mathcal{K}_{12} \, \| u \|_{H^1}.
$$
Now we estimate the $L^2$-norm of $R_M v$ for $v$ in $L^2$. 
We can assume without loss of generality that $v$ has zero mean (since $R_M C=0$ for any constant $C$) 
and then set $u=\Dx^{-1}v$. The previous inequality yields
$$
\| R_M v \|_{L^2}\le \| R_M v \|_{H^1} =\| R_M \Dx u\|_{H^1}
\leq 
C(\delta)\delta\, \| v \|_{L^2}.
$$
Therefore one has 
$$
\lA (R_0 +R_M)u\rA_{L^2}\le C(\delta)\delta \lA u\rA_{L^2}.
$$

It remains to study the sum for $0\le |\a|\le 2$ in the right-hand side of \e{p255a}. 
One can 
split this sum into two symbols such that the contribution of the first symbol 
is the two terms in the right-hand side of \e{p251a} while 
the other symbol is of the form 
$Q(x,\xi)e^{i|\xi|^\mez \beta}$ with $Q$ of order $0$. Therefore the contribution of the second symbol 
can be estimated by means of Lemma~\ref{L:26}, so it can be added to $R_0+R_M$ to obtain an operator 
$R_A$ satisfying the estimate in the statement of the lemma.
\end{proof}

\begin{nota}
Set
$$
\Nr'\defn \lA W\rA_{\bo([0,T];H^{s_0-d})} 
+\lA \Rtrois\rA_{\bo([0,T];\Lr(L^2))},
$$
where $s_0$ is the large enough integer which appears in the definition of $\Nr$ (see Assumption~\ref{pA:41}) and 
$d$ is an absolute number independent of $s_0$ (as in the statement of Proposition~\ref{P:38}). 
\end{nota}

We now chose $\beta$ under the form $\beta_0(t)+\beta_1(t,x)$ for some function coefficient $\beta_0(t)$ to be determined later and with $\beta_1=\frac{2}{3\sqrkappa}\px^{-1}\V$. Then $\beta$ is such that
$$
\sqrkappa\tdm\px\beta=\sqrkappa\tdm\px\beta_1=\V.
$$
Recall from \e{n230} that 
\be\label{n230bis}
\Rcinq\defn A^{-1}\Big(\left[ \partial_t,A\right] + \Rquatre A 
+W\px  A+i\left[ L,A \, \right] \Big).
\ee
Now we split the last term as 
$i\left[ L,A \, \right]=i\left[ \sqrkappa\Dx^\tdm \, ,A\right]+i\left[ L-\sqrkappa\Dx^\tdm \, ,A\right]$. Then it follows from the previous corollary that the remainder $\Rcinq$ 
(as defined by \e{n230}) satisfies
\be\label{n260}
\begin{aligned}
\Rcinq\defn A^{-1}\Big(& \left[ \partial_t,A\right] 
-\Op\Big(\Big(\sqrkappa\tdm\frac{\xi}{|\xi|}\px q -
\frac{9\sqrkappa i}{8}(\px\beta)^2  q\Big)|\xi|^\mez 
e^{i|\xi|^\mez\beta}\Big)
\\
&+ \Rquatre  A+i\left[ L-\sqrkappa\Dx^\tdm \, ,A\right]
-\sqrkappa R_A
\Big)
\end{aligned}
\ee
where $R_A$ is as given by Corollary~\ref{CoroR4}. 
Recall that $\Rquatre$ is an operator of order $0$. On the other hand 
$$
\left[ \partial_t,A\right]=\Op(\partial_t p)=\Op\Big(\big(\partial_t q +i|\xi|^\mez (\partial_t \beta)q\big)e^{i|\xi|^\mez \beta}\Big).
$$
So one can write $\Rcinq$ under the form 
$\Rcinq=\Rcinq^{(1/2)}+\Rcinq^{(0)}$ where $\Rcinq^{(1/2)}$ (resp.\ $\Rcinq^{(0)}$) is of order $1/2$ (resp.\ $0$),
\begin{align*}
\Rcinq^{(1/2)}&\defn A^{-1} \Op\Big( i|\xi|^\mez \big(\partial_t \beta+\frac{9\sqrkappa}{8}(\px\beta)^2\big) p-\tdm\sqrkappa\frac{\xi}{|\xi|}(\px q)
|\xi|^\mez 
e^{i|\xi|^\mez\beta}
\Big),\\
\Rcinq^{(0)}&\defn A^{-1}\Big(  \Rquatre  A-\sqrkappa R_A
+i\left[ L-\sqrkappa\Dx^\tdm \, ,A\right]+\Op\Big((\partial_t q)e^{i|\xi|^\mez \beta}\Big)\Big).
\end{align*}
We claim that
\be\label{p261}
\blA \Rcinq^{(0)}\brA_{\bo([0,T];\mathcal{L}(L^2))}\les 
\Nr'.
\ee
Indeed, $R_A$ has already been estimated and, directly from its definition (see \e{p229}), the Sobolev embedding 
$\lA \px \V\rA_{L^\infty}\le \lA \V\rA_{H^2}$ 
and \e{p215}, one has $\lA \Rquatre\rA_{\bo([0,T];\mathcal{L}(L^2))}\les 
\Nr'$. The last term is estimated by means of Lemma~\ref{L:26} 
and to estimate the commutator $\big[ A,L-\sqrkappa\Dx^\tdm \, \big]$ we notice that $L-\sqrkappa\Dx^\tdm$ is a smoothing operator. 

Now, in view of \e{p214}Ê
and \e{p215} one has $\Nr'\les \Nr$ and hence 
$\blA \Rcinq^{(0)}\brA_{\bo([0,T];\mathcal{L}(L^2))}\les 
\Nr$. 

It remains to prove that $\beta$ and $q$ can be so chosen that $\Rcinq^{(1/2)}=0$.
To do so, we first fix $\beta_0(t)$ such that 
\be\label{n265}
2\pi \partial_t\beta_0 = -\int_{\xT} \Big(\partial_t \beta_1+\frac{9\sqrkappa}{8}(\px\beta_1)^2\Big)(t,x) \, dx,
\ee
where recall that $\beta_1=-\frac{2}{3\sqrkappa}\px^{-1}\V$, 
so that
$$
\int_{\xT} \Big(\partial_t \beta+\frac{9\sqrkappa}{8}(\px\beta)^2\Big)(t,x) \, dx=0.
$$
Now define $q$ as $q=e^\gamma$ where $\gamma$ is such that
\be\label{n267}
\gamma=\frac{2}{3\sqrkappa} i\frac{\xi}{|\xi|}\px^{-1}\Big(\partial_t \beta+\frac{9\sqrkappa}{8}(\px\beta)^2\Big).
\ee
(Notice that the previous cancellation for the mean implies that $\gamma$ is periodic in~$x$.)  
With this choice one has $\Rcinq^{(1/2)}=0$.

By combining the previous results, we end up with the following proposition. 
\begin{prop}\label{P:39}
Assume that $s_0$ is large enough. Consider the operator
$$
A\defn \Op\big(q(t,x,\xi)e^{i\beta(t,x)|\xi|^\mez}\big)
$$
with
$$
\beta=\beta_0(t)+\frac{2}{3\sqrkappa}\px^{-1}\V
$$
where $\beta_0$ determined by \e{n265}, and $q=e^{\gamma}$ where $\gamma$ is given by \e{n267}. 
Then 
$$
\big(\partial_t +W\px+iL +\Rquatre\big) A  
=A   \big(   \partial_t +iL+\Rcinq\big)
$$
with
$$
\lA \Rcinq\rA_{\bo([0,T];\mathcal{L}(L^2))}\les 
\Nr
$$
where $\Nr$ is as defined in Assumption~\ref{pA:41}.\end{prop}

\section{Ingham type inequalities}\label{S:4}

\renewcommand{\th}{\vartheta}

As already mentioned, the controllability of the linearized equation around the null solution is based on 
(a modification of) the following Ingham's inequality: 
for every $T>0$ there exist two positive constants $C_1=C_1(T)$ and $C_2=C_2(T)$ such that, 
for all $(w_n)_{n \in \xZ} \in \ell^2(\xZ;\xC)$,
$$
C_1\sum_{n\in\xZ}|w_n|^2
\le \int_{0}^T \bigg| \sum_{n\in\xZ} w_n e^{i n |n|^{\frac12} t}\bigg|^2\, dt
\le C_2\sum_{n\in\xZ}|w_n|^2.
$$
Hereafter, $(w_n)_{n\in\xZ}$ always refers to an arbitrary complex-valued
sequence in $\ell^2(\xZ)$.

For our purposes, we need to consider more general phases that do not depend linearly on $t$. 
For some given real-valued function $\beta\in C^3(\xR)$, set
$$
\mu_n(t) 
= \sign(n) \left[ \ell(n) t + \beta(t) |n|^\mez \right], \quad 
\ell(n) = (g + n^2)^\mez |n|^\mez \tanh^\mez(b|n|),
$$
with $\mu_0=0$ and $\sign(n)=n/|n|$ for $n\neq 0$.
We recall that $\ell$ is the symbol of the linear operator $L=(g-\px^2)^\mez G(0)^\mez$ obtained by linearizing the water waves system around the null solutions, see Section \ref{P:S22}.
We begin by proving a lower bound which holds for any $T>0$ provided that the functions 
contain only large enough frequencies.

\begin{prop}[High frequencies]\label{P41}
Let $T > 0$. 
Let $| \pa_t \beta | \leq \frac12 \tanh^\mez(b)$ and $| \pa_t^2 \beta | \leq 1$ 
for all $t \in [0,T]$. 
Then there exists $N_0 \geq 0$ such that, for all $N \ge N_0$,
\be\label{n42}
\frac{T}{2}\sum_{\substack{n\in\xZ \\ |n| \geq N}} |w_n|^2\le \int_0^T 
\bigg\lvert \sum_{\substack{n\in\xZ \\ |n| \geq N}} w_n e^{i\mu_n(t)}\bigg\rvert^2\, dt.
\ee
\end{prop}

\begin{rema}
$i)$ For $T$ small, one can take $N_0 = C T^{-2-\eps}$ for some $\eps>0$. 
See \eqref{NT prec} for more details on this estimate. 
$ii)$ For $\| \pa_t^2 \b \|_{L^\infty}$ small enough and $T$ large enough, 
the result holds with $N_0 = 0$.
\end{rema} 
\begin{proof} 
Splitting the sum into $n=m$ and $n \neq m$, we write
$$
\int_0^T 
\bigg| \sum_{\substack{n\in\xZ \\ |n| \geq N}} w_n e^{i\mu_n(t)} \bigg|^2\, dt
\geq 
T \sum_{\substack{n\in\xZ \\ |n| \geq N}} |w_n|^2 
+ \sum_{\substack{n\neq m\\ |m|\ge N, |n|\ge N}} w_n \overline{w_m} \, \int_0^T e^{i(\mu_n(t)-\mu_m(t))}\, dt.
$$
We have to estimate
$$
K(n,m)\defn \int_0^T 
e^{i(\mu_n(t)-\mu_m(t))}\, dt.
$$
Integrating by parts, 
\[
K(n,m) = \bigg[ \frac{ e^{i(\mu_n(t)-\mu_m(t))} }{i (\mu_n'(t) - \mu_m'(t)) } \bigg]_{t=0}^{t=T}
+ \int_0^T e^{i(\mu_n(t)-\mu_m(t))} \frac{\mu_n'' - \mu_m''}{i(\mu_n' - \mu_m')^2} \, dt,
\]
and therefore 
\[
|K(n,m)| \leq \kappa(n,m) := 
\bigg\| \frac{2}{\mu_n' - \mu_m'} \bigg\|_{L^\infty([0,T])} 
+ \int_0^T \frac{\la \mu_n'' - \mu_m'' \ra}
{|\mu_n' - \mu_m' |^2}\, dt.
\]
Since $\kappa(n,m)=\kappa(m,n)$, we have
$$
\bigg| \sum_{n\neq m}w_n\overline{w_m} \, K(n,m) \bigg| 
\le \frac12 \sum_{n\neq m} \big( \la w_n\ra^2+\la w_m\ra^2 \big) \kappa(n,m)
\le \sum_{n\neq m}\la w_n\ra^2\kappa(n,m).
$$
Hence 
\begin{equation*}
\int_0^T
\bigg| \sum_{\substack{n\in\xZ \\ |n| \geq N}} w_n e^{i\mu_n(t)}\bigg|^2\, dt
\ge \sum_{\substack{n\in\xZ \\ |n| \geq N}} \Big( T - \sum_{\substack{m\in \xZ\setminus \{n\} \\ |m| \geq N}} \kappa(n,m) \Big) |w_n|^2.
\end{equation*}
We have to prove that $N$ can be so chosen that
\begin{equation}\label{n47}
T - \sum_{\substack{m\in \xZ\setminus \{n\} \\ |m| \geq N}} \kappa(n,m) \ge \frac{T}{2}.
\end{equation}
To do so, we use the following lemma.

\begin{lemm} \label{lemma:Ingham 1}
Assume that $\la \partial_t\beta(t)\ra \leq \frac12 \tanh^\mez(b)$ for all $t$. 
Let $\eps \in (0, \frac12)$. 
Then 

$i)$ There exists a positive constant $K_\eps$ such that, 
for all integers $N \geq 0$ and all $n \in \xZ$ with $|n| \geq N$, 
\begin{equation}\label{n48}
\sum_{\substack{m\in \xZ\setminus \{n\}\\ |m| \geq N}}
\bigg\| \frac{1}{\mu_n' - \mu_m'} \bigg\|_{L^\infty([0,T])} 
\le \frac{K_\eps}{(1+N)^{\mez-\eps}}\,.
\end{equation}

$ii)$ For all integers $n,m$ in $\xZ$ with $n\neq m$, and all $t$,
\begin{equation}\label{n49}
\frac{\la \mu_n''-\mu_m''\ra}{|\mu_n'-\mu_m'|}\le 2\tanh^{-\mez}(b)\, \la \partial_t^2\beta\ra.
\end{equation}
\end{lemm}

\begin{proof} 
Let us prove statement $i)$. 
Since $\kappa(-n,m) = \kappa(n,-m)$, we can assume, without loss of generality, that $n \geq 0$. 
Let $|\pa_t \beta| \leq \frac12 \tanh^\mez(b)$, and note that $\tanh(b) < 1 \leq 1+g$. 
Then for all $n \geq 0$ 
\[
\tanh^\mez(b) n^{\frac32} \leq \ell(n) \leq (1+g)^{\frac12} n^\tdm, \quad 
\frac12\, \tanh^\mez(b) n^\tdm \leq \mu_n'(t) \leq \frac32\,(1+g)^\mez n^\tdm. 
\]
For $m \leq 0$, $m \neq n$, one has
\[
|\mu'_n-\mu_m'| 
= \mu_n' + \mu_{-m}' 
\geq \frac12 \tanh^\mez(b) \, (n^\tdm + |m|^{\frac32}) 
\geq \frac14 \tanh^\mez(b) \, (1 + |m|^\tdm)
\]
and therefore
\[
\sum_{m \leq -N} \bigg\| \frac{1}{\mu_n' - \mu_m'} \bigg\|_{L^\infty([0,T])} 
\leq \sum_{m \leq -N} \frac{C}{1 + |m|^{3/2}}
\leq \frac{C'}{\sqrt{1 + N}}
\] 
for some constant $C' > 0$.
We now consider the case $m > 0$ and split the sum into two pieces. 
For $m \ge A n$ with $A := (36 (1+g) / \tanh(b))^{1/3}$ one has 
$\mu_m' \geq 2 \mu_n'$, and
$$
\la \mu_n'-\mu_m'\ra 
\ge  
\mu_m' \Big| 1 - \frac{\mu_n'}{\mu_m'} \Big| 
\geq \frac12 \, \mu_m' 
\geq \frac14 \, \tanh^\mez(b) \, m^\tdm
$$
which again leads to a convergent series 
\[
\sum_{\substack{m > 0 \\ m \ge A n, \, m \geq N}} \bigg\| \frac{1}{\mu_n' - \mu_m'} \bigg\|_{L^\infty([0,T])} 
\le \frac{C}{\sqrt{1+N}}\,.
\]
It remains to consider the sum over all $m > 0$ such that $ N \leq m < An$. 
Denote $\s(n) := \sqrt{(g + n^2)n}$. Then
$$
\frac{ \s(m) - \s(n) }{ m - n } \, 
= \frac{\s(m)^2 - \s(n)^2 }{(m-n) (\s(m) + \s(n))} 
= \frac{m^2 + n^2 + nm + g}{ \s(m) + \s(n) }\,.
$$
Using the elementary inequality $ab \leq \frac12 (a^2 + b^2)$, one has
\[
\big( \s(n) + \s(m) \big) \sqrt{n} 
= \sqrt{n^2 + g} \, n + \sqrt{m^2 + g} \sqrt{nm}
\leq m^2 + n^2 + nm + g 
\]
for all $m,n \geq 0$. Therefore 
\begin{equation}  \label{sigma nm} 
|\s(m) - \s(n)| 
\geq \sqrt{\max\{n,m\}} \, |m-n|
\end{equation}
for all $m,n \geq 0$. 
Now suppose that $m \leq n$, with $m \geq 0$, $n \geq 1$. Then
\begin{align*}
\mu_n' - \mu_m' & = (\s(n) - \s(m)) \tanh^\mez(bn) 
+ \s(m) (\tanh^\mez(bn) - \tanh^\mez(bm)) 
\\ & \qquad + (n^\mez - m^\mez) \pa_t \b
\\ & \geq (\s(n) - \s(m)) \tanh^\mez(bn) - (n^\mez - m^\mez) |\pa_t \b|
\\ & \geq n^\mez (n-m) \tanh^\mez(b) \Big( 1 - \frac{|\pa_t \b|}{\sqrt{n} (\sqrt{n} + \sqrt{m}) \sqrt{\tanh(b)} }  \Big) 
\\ & \geq \frac12 \tanh^\mez(b) \sqrt{n} (n-m)
\end{align*}
if $|\pa_t \b| \leq \frac12 \sqrt{\tanh(b)}$. We deduce that 
\be \label{mu nm}
|\mu_n' - \mu_m'| \geq C \sqrt{\max\{n,m\}} \, |n-m| 
\ee
for all $m,n \geq 0$, with $C = \frac12 \tanh^\mez(b)$. 
Now, for $n \geq 1$, we obtain
\begin{equation}  \label{sum region 3}
\sum_{\substack{m > 0, \, m \neq n \\ N \leq m < An}} \bigg\| \frac{1}{\mu_n' - \mu_m'} \bigg\|_{L^\infty([0,T])} 
\le 
\frac{1}{C \sqrt{n}} 
\sum_{\substack{m > 0, \, m \neq n \\ N \leq m < An}} \frac{1}{|n-m|}
\leq \frac{c \log(cn)}{\sqrt{n}}
\end{equation} 
for some $c > 0$. 
For $n \geq 1$ and $n \geq N$, one has $n \geq \frac12(1+N)$, and 
\[
c \log(cn) n^{-\frac12} 
\leq C_\eps n^{-\frac12 + \eps}
\leq C_\eps 2^{\frac12 - \eps} (1+N)^{-\frac12 + \eps}
\]
for $\eps \in (0,\frac12)$, for some $C_\eps > 0$.  
On the other hand, for $n=0$ the first sum in \eqref{sum region 3} is zero because it has no terms. 
Thus the first sum in $\eqref{sum region 3}$ is $\leq C_\eps (1+N)^{-\mez + \eps}$ for any $n \geq 0$. 
This completes the proof of statement $i)$. 
Statement $ii)$ is proved using \eqref{mu nm}.
\end{proof}

The previous lemma and the definition of $\kappa(n,m)$ imply that
\[
\sum_{\substack{m \in \xZ\setminus \{n\} \\ |m| \geq N}} \kappa(n,m) 
\leq 
\frac{2 K_\eps}{(1+N)^{\frac12 - \eps}}\,(1 + T \| \pa_t^2 \b \|_{L^\infty}).
\]
Hence \eqref{n47} is satisfied provided that
\begin{equation} \label{NT prec}
\frac{4 K_\eps}{T} \big( 1 + T \| \pa_t^2 \b \|_{L^\infty} \big) 
\leq (1+N)^{\frac12 - \eps},
\end{equation}
and Proposition \ref{P41} is proved.
\end{proof}

From \e{mu nm}, applied for $\b=0$, we deduce that 
\be \label{ell nm}
|\ell(n) - \ell(m)| \geq C \sqrt{\max\{n,m\}} \, |n-m| 
\ee
for all $m,n \geq 0$, with $C = \frac12 \tanh^\mez(b)$. 

We now prove upper bounds. By contrast with the previous proposition, 
we shall see that these estimates hold for any function (not only for high frequencies). 
Also, a key point for later purpose is that 
one can add some amplitudes $\zeta_n$ depending on time (and whose derivatives 
in time of order $k$ can grow with $n$ as $|n|^{k/2}$). 

\begin{prop}\label{P43}
There exists $C > 0$ with the following property. 
Let $T > 0$. Let $| \pa_t \b| \leq \frac12 \tanh^\mez(b)$, 
and $|\pa_t^k \b| \leq 1$, $k = 2,3$ on $[0,T]$. 
Then, for all $(w_n) \in \ell^2(\xZ;\xC)$,
\be\label{n42b}
\int_0^T 
\bigg| \sum_{n\in\xZ} w_n\zeta_n(t) e^{i\mu_n(t)} \bigg|^2\, dt
\le C M(\zeta)^2 (1+T) \sum_{n\in\xZ}|w_n|^2
\ee
where
\be\label{n42c}
M(\zeta) \defn \sup_{n\in\xZ} \| \zeta_n \|_{L^\infty} 
+ \sup_{n \in \xZ} \frac{ \| \pa_t \zeta_n \|_{L^\infty} }{\sqrt{1 + |n|}} 
+ \sup_{n \in \xZ} \frac{ \| \pa_t^2 \zeta_n \|_{L^\infty} }{1 + |n|}\,.
\ee
\end{prop}

\begin{proof}
Splitting the sum into $n=m$ and $n \neq m$, we write
\[ 
\int_0^T 
\bigg| \sum_{n\in\xZ} w_n \zeta_n(t) e^{i\mu_n(t)} \bigg|^2 \, dt
= \sum_{n\in\xZ} \Big(\int_{0}^{T} |\zeta_n(t)|^2\, dt \Big) |w_n|^2
+\sum_{n\neq m} w_n \overline{w_m} \, E(n,m)
\] 
with
$$
E(n,m) \defn \int_{0}^{T} \zeta_n(t) \overline{\zeta_m(t)} \, e^{i(\mu_n(t)-\mu_m(t))}\, dt.
$$
The first sum on the right-hand side is easily estimated. 
It remains to bound the sum for $n\neq m$.
Integrating by parts twice, one has 
\begin{align*}
E(n,m) = \int_0^T f e^{ih} dt 
& = \big[ e^{ih} (- i f p + f' p^2 - f h'' p^3 ) \big]_0^T
\\ & \quad 
+ \int_0^T e^{ih} (f'' p^2 - 3 f' h'' p^3 + 3 f h''^2 p^4 - f h''' p^3)\, dt
\end{align*} 
with 
\[
f := \zeta_n \overline{\zeta_m}, \qquad 
h := \mu_n - \mu_m, \qquad 
p := \frac{1}{\mu_n' - \mu_m'}\,. 
\]
Thus $|E(n,m)|\le e(n,m)$, where 
\begin{align} \label{n415}
e(n,m) & \defn 
2 \| fp \|_{L^\infty}
+ 2 \| f' p^2 \|_{L^\infty}
+ 2 \| f h'' p^3 \|_{L^\infty}
\\ & \quad \ 
+ T ( \| f'' p^2 \|_{L^\infty} 
	+ 3 \| f' h'' p^3 \|_{L^\infty}
	+ 3 \| f h''^2 p^4 \|_{L^\infty}
	+ \| f h''' p^3 \|_{L^\infty} ).
\notag 
\end{align}
We have to estimate the sum $\sum_{m \in \xZ \setminus \{ n \}} e(n,m)$, uniformly in $n$.
First, we note that
\[
\| \pa_t^k(\zeta_n \overline{\zeta_m}) \|_{L^\infty}
= \| \pa_t^k f \|_{L^\infty} \leq \{ (1+|n|)^\mez + (1+|m|)^\mez \}^k M(\zeta)^2, \quad k=0,1,2.
\]
We have already seen in \eqref{n49} that $| h'' p| \leq 2 |\pa_t^2 \b|$. 
Similarly, $| h''' p| \leq 2 |\pa_t^3 \b|$. 
Also, applying \eqref{n48} with $N = 0$, $\eps = \frac14$, 
we deduce that $\sum_{m \in \xZ \setminus \{n\}} \| p \|_{L^\infty} \leq C$ 
for some absolute constant $C$. 
Therefore the first, the third and the last two terms in \eqref{n415} (i.e.\ those with $f$)
are all bounded by $C M(\zeta)^2 ( 1 + T ) $. 
The remaining three terms of \eqref{n415} are also bounded by $C M(\zeta)^2 (1 + T)$ 
provided that 
\begin{equation}\label{n48b}
\sum_{m\in \xZ\setminus \{n\}} \bigg\| \frac{|n| + |m|}{(\mu_n' - \mu_m')^2} \bigg\|_{L^\infty} \, \leq C
\end{equation}
for all $n \in \xZ$, for some $C$ independent of $n$.
The bound \e{n48b} is proved using the same splitting and estimates as in the proof of Lemma 
\ref{lemma:Ingham 1}. 
\end{proof}

By combining the two previous propositions with an induction argument 
(following \cite{BallSlemrod,Haraux,MicuZuazua}), we now deduce the following result.

\begin{prop}[Sharp Ingham type inequality]\label{P44}
Let $T > 0$. Then there exist two positive constants $C(T)$ and $\delta(T)$ such that, if 
\be\label{n48c}
\lA \beta\rA_X\defn \sup_{t\in [0,T]}\la (\partial_t\beta,\partial_t^2\beta,\partial_t^3\beta)\ra\le \delta(T),
\ee
then, for all $(w_n) \in \ell^2(\xZ;\xC)$, 
$$
C(T)\sum_{n\in\xZ}|w_n|^2
\le \int_0^T \bigg| \sum_{n\in\xZ} w_n e^{i\mu_n(t)} \bigg|^2\, dt.
$$
\end{prop}

\begin{proof}
This proposition will be deduced from Proposition~\ref{P41}, 
the following claim and an immediate induction argument (with a finite number of steps).

\begin{claim}
Consider two subsets $\mathcal{A},\mathcal{A}'$ of $\xZ$ with $\mathcal{A}'=\mathcal{A} \cup \{N\}$ 
for some $N \in \xZ$, and with $|n| \ge |N|$ for all $n$ in $\mathcal{A}$. 
Assume that for every $T > 0$ there exist two positive constants $\delta(T)$ and $K(T)$ such that 
\be\label{n49bis}
\lA \beta\rA_{X}\le \delta(T) ~ \Rightarrow ~ K(T)\sum_{n\in\mathcal{A}}|w_n|^2
\le \int_0^T \bigg| \sum_{n\in\mathcal{A}} w_n e^{i\mu_n(t)} \bigg|^2\, dt.
\ee
Then for every $T > 0$ there exist two positive constants $\delta'(T)$ and $K'(T)$ such that 
\be\label{n49ter}
\lA \beta\rA_{X}\le \delta'(T) ~ \Rightarrow ~ K'(T)\sum_{n\in\mathcal{A}'}|w_n|^2\le \int_0^T 
\bigg| \sum_{n\in\mathcal{A}'} w_n e^{i\mu_n(t)} \bigg|^2\, dt.
\ee
\end{claim}

Let us prove the claim. Introduce 
$$
f(t)\defn \sum_{n\in\mathcal{A}} w_n e^{i\mu_n(t)}, \quad 
f'(t)\defn \sum_{n\in\mathcal{A}'} w_n e^{i\mu_n(t)}, \quad 
f_1(t)\defn \sum_{n\in\mathcal{A}'} w_n e^{i\mu_n(t)-i\mu_N(t)},
$$
so that $f' = f + w_N e^{i \mu_N}$, 
$f_1 = e^{-i\mu_N} f' = fe^{-i\mu_N}+w_N$, and
$$
\int_0^T |f_1(t)|^2\, dt=\int_0^T |f'(t)|^2\, dt
=\int_0^T 
\bigg| \sum_{n\in\mathcal{A}'} w_n e^{i\mu_n(t)}\bigg|^2\, dt.
$$ 
We prove that there exist two constants $C_1, C_2$ (both depending on $T$) such that
\begin{equation}  \label{C12 claim}
C_1 \sum_{n\in\mathcal{A}}|w_n|^2\le \int_0^T  |f'(t)|^2\, dt,
\quad 
C_2 |w_N|^2\le \int_0^T  |f'(t)|^2\, dt.
\end{equation}
Then \eqref{C12 claim} implies the second inequality of \eqref{n49ter} 
with $K'(T) := \frac12 \min\{ C_1, C_2 \}$. 
Let us begin with the first inequality of \eqref{C12 claim}. 
Let $\tau := \frac12 \min \{ 1, T \}$, 
and remark that
\be\label{n49c}
\int_0^\tau \left(f_1(t+\eta)-f_1(t)\right)\, d\eta=e^{-i\mu_N(t)}
\sum_{n\in\mathcal{A}} w_n e^{i\mu_n(t)} \theta_n(t),
\ee
(notice that the sum is over $\mathcal{A}$ and not $\mathcal{A}'$) with
$$
\theta_n(t)=\int_0^\tau \left(e^{i\left( \mu_n(t+\eta)-\mu_n(t)-\mu_N(t+\eta)+\mu_N(t)\right)}-1\right)\, d\eta.
$$
Assume that $n,N$ are positive. 
We split $\theta_n = c_n + \zeta_n$, 
where $c_n$ is a constant, independent of time 
(such that $c_n = \theta_n$ for $\b = 0$), 
and $\zeta_n$ is defined by difference, namely 
\begin{align*}
c_n & \defn 
\int_0^\tau \big( e^{i [ \ell(n) - \ell(N) ] \eta } - 1 \big)\, d\eta
= \frac{e^{i [ \ell(n) - \ell(N) ] \tau} - 1}{i [ \ell(n) - \ell(N) ]} - \tau, 
\\
\zeta_n & \defn 
\int_0^\tau e^{i [ \ell(n) - \ell(N) ] \eta} 
\Big( e^{i [ \beta(t+\eta)-\beta(t)] ( \sqrt{n} - \sqrt{N} )} - 1 \Big)\, d\eta.
\end{align*}
Now we use the following elementary inequality: 
there exists an absolute constant $c_0 > 0$ such that, for all $\th \in \R$,
\[
| e^{i \th} - 1 - i \th|^2 \geq 
c_0 \min \{ \th^2, \th^4 \}.
\]
This inequality holds because $|e^{i \th} - 1 - i \th|^2 
= (1 - \cos \th)^2 + (\th - \sin \th)^2$ is positive for all $\th \neq 0$ 
and it has asymptotic expansion
$\th^2 + o(\th^2)$ for $|\th| \to \infty$, and $\frac14 \th^4 + o(\th^4)$ for $\th \to 0$. 
We apply this inequality with $\th = [\ell(n) - \ell(N)] \tau$, and, using \e{ell nm}, we get 
\[
|c_n|^2 \geq c \t^4
\]
for some $c > 0$ (note that $\min \{ \t^2, \t^4 \} = \t^4$ because, by assumption, $\t < 1$). 

It remains to estimate $\zeta_n$ and its derivatives. From the definition, 
\begin{align*}
&|\zeta_n|\le 2 \tau , \quad 
| \partial_t \zeta_n | 
\le 2 \| \partial_t\beta \|_{L^\infty} \tau \sqrt{n}, \quad 
| \partial_t^2 \zeta_n | 
\le 4 (\| \partial_t\beta \|_{L^\infty}^2 
+ \| \partial_t^2 \beta \|_{L^\infty}) \tau n.
\end{align*}
However, we need a sharper bound on $\zeta_n$ which shows 
that $\zeta_n$ is small when $\beta$ is small. 
Such a bound could be easily obtained by estimating $|e^{if}-1| \leq |f|$. 
However, this would make appear an extra factor $\sqrt{n}$. 
Instead, we integrate by parts to obtain 
\begin{align*}
\zeta_n(t) = {} &
\frac{e^{i [\ell(n) - \ell(N)] \tau}}{i [ \ell(n) - \ell(N) ] } \,  
\Big( e^{i [ \beta(t+\tau)-\beta(t)] ( \sqrt{n} - \sqrt{N} )} - 1 \Big)
\\
& - \int_0^\tau \frac{e^{i [\ell(n) - \ell(N)] \eta} }{i [ \ell(n) - \ell(N)] } \,
\partial_{\eta} \Big( e^{i [ \beta(t+\eta)-\beta(t)] ( \sqrt{n} - \sqrt{N} )} - 1 \Big)\, d\eta,
\end{align*}
and it is easily checked, using \eqref{ell nm}  and the bound 
$| \beta(t+\t) - \b(t)| \leq \t \| \pa_t \b \|_{L^\infty}$,  
that $|\zeta_n| \le C \t \| \partial_t \beta \|_{L^\infty}$. 
By combining the previous estimates, we have 
$M(\zeta)\le C \tau \| \beta \|_{X}$ where $M(\zeta)$ is given by \e{n42c}, 
and $C$ is independent on $T,\t$. 

Set $F(t)\defn \sum_{n\in\mathcal{A}} w_n e^{i\mu_n(t)} \theta_n(t)$ 
and split $F=F_1+F_2$ with 
$$
F_1(t)\defn\sum_{n\in\mathcal{A}} w_n e^{i\mu_n(t)} 
 c_n,\quad 
F_2(t)\defn  \sum_{n\in\mathcal{A}} w_n e^{i\mu_n(t)} 
 \zeta_n(t).
$$
Since $|c_n|^2 \ge c \tau^4$, the assumption \e{n49bis} implies that, 
if $\lA \beta\rA_{X}\le \delta(T-\tau)$, then
$$
c \tau^4 K(T-\tau) \sum_{n\in\mathcal{A}} |w_n|^2 
\leq K(T-\tau) \sum_{n\in\mathcal{A}} |w_n c_n|^2 
\leq \int_0^{T-\tau} |F_1(t)|^2\, dt.
$$
On the other hand, Proposition~\ref{P43} applied with 
$M(\zeta)\le C \tau \| \beta\|_{X}$ implies that, 
if $\| \beta \|_X \le \frac12 \tanh^\mez(b)$, then 
$$
\int_0^{T-\tau} |F_2(t)|^2\, dt
\le C_0 \tau^2 \| \beta \|_{X}^2 (1 + T - \tau) \sum_{n\in \mathcal{A}} |w_n|^2
$$
where $C_0$ is  independent of $T,\t$.  
Therefore, if 
\begin{equation} \label{smallness Ingham} 
4 C_0 \t^2 \| \beta \|_{X}^2 (1 + T - \tau) 
\le c \tau^4 K(T-\t), 
\end{equation}
then $\int_0^{T-\tau} |F_2|^2 dt \leq \frac14 \int_0^{T-\tau} |F_1|^2 dt$,   
whence 
$\int_0^{T-\tau} |F|^2 dt \geq \frac14 \int_0^{T-\tau} |F_1|^2 dt$. 
By \eqref{n49c}, this implies that 
\begin{align*}
\frac{1}{4} \, c \t^4 K(T-\tau) \sum_{n\in\mathcal{A}}|w_n|^2 
& \le \int_0^{T-\tau} |F(t)|^2\, dt 
\\
&\le \int_0^{T-\tau} \la \int_0^\tau \left(f_1(t+\eta)-f_1(t)\right)\, d\eta\ra^2\, dt.
\end{align*}
The condition \eqref{smallness Ingham}  holds if 
\begin{equation}  \label{bla}
\lA \beta\rA_X\le \frac{\tau \sqrt{ c \, K(T-\t)}}{2 \sqrt{C_0 (1+T)} }\,,
\end{equation}
and we set $\delta'(T)$ as the minumum among $\frac12 \tanh^\mez(b), \delta(T)$, 
and the constant on the right in \eqref{bla}.  Moreover,
\begin{align*}
&\int_0^{T-\tau} 
\la \int_0^\tau \left(f_1(t+\eta)-f_1(t)\right)\, d\eta\ra^2\, dt\\
&\qquad\qquad\le \int_0^{T-\tau} 
\tau \int_0^\tau \la  \left(f_1(t+\eta)-f_1(t)\right)\ra^2 \, d\eta dt\\
&\qquad\qquad\le  
2 \tau \int_0^{T-\tau} \int_0^\tau \la f_1(t+\eta)\ra^2 \, d\eta dt 
+ 2 \tau \int_0^{T-\tau} \int_0^\tau \la f_1(t) \ra^2 \, d\eta dt \\
&\qquad\qquad  \le 2 T \tau \int_0^T|f_1(t)|^2\, dt 
= 2 T \tau \int_0^T \la f'(t)\ra^2\, dt, 
\end{align*}
and we infer that  the first inequality in \e{C12 claim} holds 
with $C_1 = \frac18 c \t^3 K(T-\t) T^{-1}$. 

Now we prove the second inequality in \e{C12 claim}.
We have $|w_N|^2=\la f'(t)-f(t)\ra^2$ for any $t$, and so 
$$
|w_N|^2=\frac{1}{T}\int_0^T \la f'(t)-f(t)\ra^2 \, dt
\le \frac{2}{T}\Big( \int_0^T |f'(t)|^2\, dt +\int_0^T|f(t)|^2\, dt \Big).
$$
It follows from Proposition~\ref{P43} (applied with $\zeta_n = 1$) that
$$
\int_0^T|f(t)|^2\, dt\le (1+T) C \sum_{n\in \mathcal{A}}|w_n|^2.
$$
Using  the first inequality in \e{C12 claim},  we deduce that
$$
\int_0^T|f(t)|^2\, dt\le  \frac{(1+T)C}{C_1}  \, \int_0^T \la f'(t)\ra^2\, dt,
$$
where $C$ is the constant of Proposition \ref{P43} 
and $C_1$ has been found above.  
Consequently  the second inequality in \eqref{C12 claim} holds 
with $C_2 = \frac12 T C_1 [C_1 + (1+T)C]^{-1}$. 
We set $K'(T) = \frac12 \min \{ C_1, C_2 \}$ and obtain \eqref{n49ter}.
This completes the proof of the claim in the case of $n,N$ positive. 
The other cases are analogous.
\end{proof}

\section{Observability} \label{P:S6}

We now use the previous inequalities for sums of oscillatory functions 
in order to prove an observability property. 
In particular, we prove that it is sufficient to control the real part of the solution 
to bound the initial data.

\begin{prop}[Observability]\label{P45}
Let $T > 0$. Consider an open subset $\omega \subset \T$ and a constant $0 < c \leq 1$. 
Then there exist positive constants $ K, \eps_1$ 
such that the following property holds. 
Consider a pseudo-differential $A_0$ with symbol $\exp\big(i\beta(t,x)|\xi|^\mez\big)$ 
for some function $\beta$ satisfying
$$
\sup_{t\in[0,T]}\sup_{x\in [0,2\pi]} 
\la (\partial_t\beta(t,x),\partial_t^2 \beta(t,x),\partial_t^3\beta(t,x))\ra 
\le \delta(T), 
$$
where $\delta(T)$ is the constant in Proposition \ref{P44}. 
Then for every 
initial data $v_0 \in L^2(\T)$ whose mean value 
$\langle v_0 \rangle = \frac{1}{2\pi} \int_{\xT} v_0(x)\, dx$ satisfies 
\be\label{n420i}
\la \RE \langle v_0 \rangle \ra \ge c \la \langle v_0\rangle \ra
-\eps_1\lA v_0\rA_{L^2},
\ee
the solution $v$ of 
\be\label{n421}
\partial_t v+iL v=0,\quad v(0)=v_0,
\ee
satisfies
\be\label{n422}
\int_0^T\int_\omega  \la \RE (A_0 v)(t,x)\ra^2 \, dxdt\ge K \int_0^{2\pi} \la v_0(x)\ra^2 \, dx.
\ee
\end{prop}

\begin{rema}
The condition \eqref{n420i} cannot be eliminated. 
To see it, consider the simplest case $\beta=0$, so $A_0 = I$, 
and consider a constant solution $v(t,x)=C$ of \eqref{n421}. 
Then \eqref{n422} holds for some $K$ if and only if 
the real part of $C$ is non zero. This suggests to assume that
\be\label{n420f}
\la \RE \langle  v_0\rangle \ra\ge c \la  \langle v_0\rangle \ra.
\ee
In fact, it is sufficient to consider the weaker assumption \eqref{n420i}. 
The advantage of assuming \e{n420i} instead of \e{n420f} is used below (see \e{n430}). 
\end{rema}

\begin{proof}
Write
$$
v(t,x) = \sum_{n\in\xZ} a_n e^{inx} e^{i \ell(n) t}, \quad 
a_n = \frac{1}{2\pi} \int_0^{2\pi} e^{-inx}v_0(x)\, dx,
$$
where $\ell(n) = (g + n^2)^{\frac12} (|n|\tanh(b|n|))^\mez$ is the symbol of $L$. 
Then set $w=A_0 v$, given by
$$
w(t,x)=\sum_{n\in\xZ} a_n e^{inx} e^{i(\ell(n) t + \beta(t,x) |n|^\mez)}.
$$
For $n\in \xZ$, set
$$
\lambda_n = \ell(n) t + \beta(t,x)|n|^\mez, \quad 
\mu_n = \sign (n) \lambda_n, \quad 
c_n(x) = a_ne^{inx}.
$$
Since $\mu_n=\sign(n)\lambda_n$ and $\mu_{-n}=-\mu_n$, we write
\begin{align*}
2\RE w &= 2\RE a_0+\sum_{n>0}c_n e^{i\lambda_n }+ \sum_{n>0}\overline{c_{n}}e^{-i\lambda_n }
+ \sum_{n<0}c_n e^{i\lambda_n }+ \sum_{n<0}\overline{c_n}e^{-i\lambda_n }\\
&= 2\RE c_0+\sum_{n>0}c_n e^{i\mu_n}+\sum_{n<0}\overline{c_{-n}}e^{i\mu_n}
+\sum_{n>0}c_{-n}e^{i\mu_n}+\sum_{n<0}\overline{c_n}e^{i\mu_n}
\end{align*}
to obtain
$$
2\RE w=\sum_{n\in\xZ}\gamma_n e^{i\mu_n}
\quad \text{with}\quad 
\gamma_n=
\left\{
\begin{aligned}
&c_n+c_{-n} \quad&&\text{for }n>0,\\
&2\RE c_0\quad&&\text{for }n=0,\\
&\overline{c_n}+\overline{c_{-n}}\quad &&\text{for }n<0.
\end{aligned}
\right.
$$
Consider an interval $\omega_0 = [a,b] \subset \omega$. 
By Proposition \ref{P44}, 
\begin{align}
\int_0^T\int_\omega | \RE (w(t,x))|^2\, dxdt
& \ge \int_{\omega_0} \int_0^T| \RE (w(t,x))|^2\, dtdx
\notag \\
\label{n424}
& \ge \frac{C(T)}{4} 
\int_{\omega_0}\sum_{n\in\xZ}|\gamma_n(x)|^2\, dx,
\end{align}
where 
$C(T)$ is the constant given in Proposition \ref{P44}. 
For $n\neq 0$ we write
$$
|\gamma_n(x)|^2 = |a_n|^2 + |a_{-n}|^2 + a_n \overline{a_{-n}}\, e^{2inx}
+ \overline{a_n} \, a_{-n} e^{-2inx},
$$
so that 
\begin{align*}
\int_{\omega_0}|\gamma_n(x)|^2\, dx
&\ge |\omega_{0}| \big\{ |a_n|^2+|a_{-n}|^2\big\} 
\\& \quad 
-|a_n|\, |a_{-n}|\bigg( \la \int_{\omega_0}e^{2inx}\, dx\ra+\la \int_{\omega_0}e^{-2inx}\, dx\ra\bigg).
\end{align*}
Now 
$$
\la \int_{\omega_0}e^{2inx}\, dx\ra=\la \int_{\omega_0}e^{-2inx}\, dx\ra=
\la \frac{\sin (n(b-a))}{n}\ra.
$$
Moreover there is a small universal constant $\delta_0 > 0$ such that, 
for all $\delta \in (0, \delta_0)$, 
$$
\forall |x|\ge \delta,\quad 
\la \frac{\sin(x)}{x}\ra \le \frac{\sin(\delta)}{\delta}.
$$
We can assume that $0 < b-a \le \delta_0$, so that
$$
\forall n\in\xZ^*,\quad (b-a)- \la \frac{ \sin (n(b-a)) }{n} \ra 
\ge (b-a) - \sin (b-a).
$$
As a consequence, for all $n \neq 0$, 
$$
\int_{\omega_0}|\gamma_n(x)|^2\, dx \ge c' \big( |a_n|^2+|a_{-n}|^2\big), 
$$
where $c' := (b-a)-\sin (b-a) > 0$. 
Then, recalling that $\g_0 = 2 \RE a_0$, it follows from \e{n424} that
$$
\int_0^T\int_\omega | \RE (w(t,x))|^2\, dxdt
\ge 
C(T) \Big[ (b-a)|\RE a_0|^2 + \frac{c'}{2}
\sum_{n \in \xZ \setminus  \{ 0 \} } |a_n|^2 \Big]. 
$$
Now, using $(x+y)^2 \geq \frac12 x^2 - y^2$ and \e{n420i}, one has 
$\la \RE \langle  v_0 \rangle \ra^2 
\ge \frac{1}{2}\, c^2 \la \langle v_0\rangle \ra^2
- \eps_1^2 \lA v_0\rA_{L^2}^2$, namely 
$$
|\RE a_0|^2\ge \frac{c^2}{2} \, |a_0|^2 - 2 \pi \eps_1^2 \sum_{n\in \xZ}\la a_n\ra^2, 
$$
and therefore 
$$  
\int_0^T\int_\omega | \RE (w(t,x))|^2\, dxdt \ge K \sum_{n\in\xZ} |a_n|^2, 
$$
with 
$K = C(T) \min \{ (b-a) (\frac12 c^2 - 2\pi \eps_1^2 ) , \, 
\frac12 c' - (b-a) 2 \pi \eps_1^2 \}$.
If $\eps_1$ is small enough, then $K > 0$, which completes the proof. 
\end{proof}

\begin{coro}\label{P45c}
Let $T > 0$, let $\omega \subset \T$ be an open subset and let $0 < c \leq 1$. 
Then there exist positive constants 
$\eps_0,\eps_1,r,K$ such that the following property holds. 
Assume that 
$$
\langle \V(t)\rangle=0
$$
for all $t\in [0,T]$ and 
$$
\sup_{t\in[0,T]}\sum_{1\le k\le 3} \lA \partial_t^k \V(t)\rA_{H^1}
+\sup_{t\in [0,T]} \lA \V(t)\rA_{H^r}\le \eps_0,
$$
and consider the pseudo-differential 
operator $A$, given by Proposition~\ref{P:39}, with symbol $q(t,x,\xi)\exp ( i\beta(t,x)\la \xi\ra^\mez )$. 
Then for every initial data $v_0 \in L^2(\T)$ whose mean value $\langle v_0\rangle 
= \frac{1}{2\pi} \int_{\xT}v_0(x)\, dx$ satisfies 
\be\label{n420ib}
\la \RE \langle  v_0\rangle \ra\ge c \la  \langle v_0\rangle \ra
-\eps_1\lA v_0\rA_{L^2},
\ee
the solution $v$ of 
\be\label{n421b}
\partial_t v+iL v=0,\quad v(0)=v_0,
\ee
satisfies
\be\label{n422b}
\int_0^T\int_\omega  \la \RE (A v)(t,x)\ra^2 \, dxdt\ge K \int_0^{2\pi} \la v_0(x)\ra^2 \, dx.
\ee
(The constants $\eps_0, \eps_1, K$ depend on $T,c$, while $r$ is a universal constant.) 
\end{coro}

\begin{proof}
Split $A$ as $A_0+A_1$ with
$$
A_0\defn \Op\big(\exp\big(i\beta(t,x)|\xi|^\mez\big)\big),\quad 
A_1\defn \Op\big((q(t,x,\xi)-1)\exp\big(i\beta(t,x)|\xi|^\mez\big)\big).
$$
The contribution due to $A_0$ is estimated by Proposition \ref{P45}. 
Notice that, for $\eps_0$ small enough, 
the smallness assumption on $\beta$ of Proposition \ref{P45} is satisfied because 
$$
\sup_{t\in[0,T]}\sup_{x\in [0,2\pi]} 
\la (\partial_t\beta(t,x),\partial_t^2 \beta(t,x),\partial_t^3\beta(t,x))\ra
\les \sup_{t\in[0,T]}\sum_{1\le k\le 3} \blA \partial_t^k \V(t)\brA_{H^1}
\les \eps_0.
$$
On the other hand, 
it follows from the definition of $q$ and $\beta$ and the estimate given by statement $ii)$ in Lemma~\ref{L:26}  that $A_1$ is bounded from $L^2$ onto itself, with an operator norm of size 
$O(\lA \V\rA_{H^r})=O(\eps_0)$. Then
$$
\int_0^T\int_\omega  \la \RE (A_1 v)(t,x)\ra^2 \, dxdt \le \int_0^T\lA A_1 v(t)\rA_{L^2}^2 \, dt\les \int_0^T \eps_0^2\lA v(t)\rA_{L^2}^2\, dt.
$$
Since $\lA v(t)\rA_{L^2}=\lA v(0)\rA_{L^2}$, by taking $\eps_0$ small enough, 
the desired estimate follows from the triangle inequality.
\end{proof}

We now want to deduce an observability result for equations of the form
$$
\partial_t w +W\px w+iL w+\mathcal{R} w=0,
$$
where $\mathcal{R}$ is an operator of order $0$. 
In the appendix we prove that the Cauchy problem for this equation is well-posed 
(see Lemma~\ref{L2}). 

\begin{coro}\label{C:54}
Let $T > 0$, $\omega \subset \T$ be a non-empty open domain 
and let $0 < c \leq 1$. 
Then there exist positive constants $\eps_2,\eps_3,r,K$ 
such that the following property holds. Assume that 
$$
\langle \V(t)\rangle=0
$$
for all $t\in [0,T]$ and 
\be \label{WR eps2}
\sup_{t\in[0,T]}\sum_{1\le k\le 3} \lA \partial_t^k \V(t)\rA_{H^1}
+\sup_{t\in [0,T]} \lA \V(t)\rA_{H^r}+\sup_{t\in [0,T]}\lA \mathcal{R}(t)\rA_{\Lr(L^2)}\le \eps_2.
\ee
Then for every initial data $w_0\in L^2(\xT)$ whose mean value 
$\langle  w_0\rangle = \frac{1}{2\pi} \int_{\xT} w_0(x)\,dx$ satisfies 
\be\label{n426}
\la \RE \langle  w_0\rangle \ra 
\ge c \la  \langle w_0\rangle \ra
- \eps_3 \lA w_0\rA_{L^2}, 
\ee
the solution $w$ of 
\be\label{n427}
\partial_t w +W\px w+iL w+\mathcal{R} w=0, \quad 
w(0)=w_0
\ee
satisfies 
\be\label{n428}²
\int_0^T\int_\omega  \la \RE w\ra^2 \, dxdt\ge K \int_0^{2\pi} \la w_0(x)\ra^2 \, dx.
\ee
\end{coro}

\begin{rema}\label{R:75}
Corollary \ref{C:54} also holds for data at time $T$, that is:   
If $w_0 \in L^2(\T)$ satisfies \eqref{n426}, 
then the solution $w$ of 
\be\label{n427-T}
\partial_t w +W\px w+iL w+\mathcal{R} w=0,
\quad w(T) = w_0
\ee
also satisfies \eqref{n428}. Note that the datum in \eqref{n427-T} is at time $T$ instead of $0$. 
To prove it, notice that the function $\wide w(t,x):= w(T-t,x)$ satisfies 
$$
-\partial_t \wide w +\wide W\px \wide w+iL \wide w+\wide{\mathcal{R}} \wide w=0,
$$
where $\wide W(t), \wide{\mR}(t)$ stands for $W(T-t), \mR(T-t)$. 
Since $\wide \V$ and 
$\wide{\mathcal{R}}$ satisfy the same assumptions as $\V,\mathcal{R}$, one can apply \e{n428} with $w$ replaced by $\wide w$, noticing that 
$$
\int_0^T\int_\omega  \la \RE w\ra^2 \, dxdt=\int_0^T\int_\omega  \la \RE \wide w\ra^2 \, dxdt.
$$
\end{rema}

\begin{proof}[Proof of Corollary \ref{C:54}]
It follows from Proposition~\ref{P:39} that there is a change of unknown $w=Av$ such that 
$v$ satisfies an equation of the form
$$
\partial_t v +iL v+\mathfrak{R} v=0,
$$
for some operator $\mathfrak{R}$ of order $0$, satisfying 
$\lA \mathfrak{R}(t) v \rA_{L^2} \le C\eps_2 \lA v\rA_{L^2}$ for all $t \in [0,T]$. 
By a perturbative argument, we shall deduce observability for this equation from observability 
for the equation without $\mathfrak{R}$. To do so, split $v$ as $v=v_1+v_2$ where $v_1$ and $v_2$ are given by the Cauchy problems
\[ 
\begin{cases} 
\partial_t v_1 +iL v_1=0 \\ 
v_1(0) = v_0 
\end{cases}
\qquad 
\begin{cases} 
\partial_t v_2 +iL v_2+\mathfrak{R} v_2 = - \mathfrak{R} v_1 \\
v_2(0) = 0
\end{cases}
\]
and $v_0 := v(0) = (A^{-1} w)(0)$. 
We begin by estimating $v_1$, claiming that its initial datum $v_0$ satisfies the hypothesis \eqref{n420ib} of Corollary \ref{P45c}, which is 
\be\label{n430}
\la \RE \langle  v_0 \rangle \ra 
\ge c \la \langle v_0 \rangle \ra
- \eps_1 \lA v_0 \rA_{L^2}
\ee
(where $\eps_1$ is given in Corollary \ref{P45c}). 
To prove \eqref{n430}, we write 
$v = w + (I - A)v$ to obtain, at time $t=0$,
\[
\la \RE \langle  v_0 \rangle \ra 
= \la \RE\langle w_0 \rangle + \RE \langle (I - A) v_0 \rangle\ra 
\ge 
\la \RE \langle w_0 \rangle \ra - \la \langle (I-A) v_0 \rangle\ra.
\]
Thus, using the assumption \e{n426}, 
\[
\la \RE \langle  v_0 \rangle \ra 
\geq c | \langle w_0 \rangle | - \eps_3 \| w_0 \|_{L^2} - \la \langle (I-A) v_0 \rangle\ra.
\]
Since $w = v + (A-I)v$, we have 
$\langle w_0 \rangle = \langle v_0 \rangle + \langle (A-I)v_0 \rangle$, and 
\[
\la \RE \langle  v_0 \rangle \ra 
\geq c | \langle v_0 \rangle | - (c+1) | \langle (A-I)v_0 \rangle | 
- \eps_3 \| w_0 \|_{L^2}.
\]
By \eqref{WR eps2}, 
$| \langle (A-I) v_0 \rangle | \leq C \eps_2 \| v_0 \|_{L^2}$ (see Lemma~\ref{T:65} below). 
Also, $\| w_0 \|_{L^2} \leq C \| v_0 \|_{L^2}$ 
because $A$ is bounded on $L^2$ (see Lemma \ref{L:26}). 
Thus 
\[
\la \RE \langle  v_0 \rangle \ra 
\geq c | \langle v_0 \rangle | - \big( (c+1) C \eps_2 + C \eps_3 \big) \| v_0 \|_{L^2},
\]
and the claim is satisfied if $\eps_2, \eps_3$ are small enough.
As a consequence, from Corollary~\ref{P45c} we deduce that 
\be\label{n432}
\int_0^T\int_\omega  \la \RE (A v_1 ) \ra^2 \, dxdt 
\ge K \int_0^{2\pi} \la v_0(x)\ra^2 \, dx.
\ee
On the other hand, it follows from \e{n26} (applied with $V=0$, $c=1$ and $R=\mathfrak{R}$) that
$$
\lA v_2\rA_{\bo([0,T];L^2)}\le C \lA \mathfrak{R} v_1\rA_{L^1([0,T];L^2)}.
$$
Since $\| \mathfrak{R}(t) v \|_{L^2}  \leq  C \eps_2 \| v \|_{L^2}$, by using  \eqref{WR eps2}, we find that the last quantity is bounded by 
$C \eps_2 T \| v_0 \|_{L^2}$. 
Since $A$ is bounded on $L^2$, we deduce that 
\be\label{n435}
\begin{aligned}
\int_0^T\int_\omega  \la \RE ( A v_2 )\ra^2 \, dxdt
& \le \int_0^T \lA A v_2(t)\rA_{L^2}^2 \, dt\le T \sup_{[0,T]} \lA Av_2(t) \rA_{L^2}^2 
\\
& \le C T \lA v_2\rA_{\bo([0,T];L^2)}^2 
\le C T^3 \eps_2^2 \| v_0 \|_{L^2}^2. 
\end{aligned}
\ee
Using the elementary inequality $(x+y)^2 \geq \frac12 x^2 - y^2$, 
for $\eps_2$ small enough we get 
\begin{align*}
\int_0^T\int_\omega \la \RE ( A v ) \ra^2 \, dxdt
&\ge \frac{K}{4} \int_0^{2\pi} \la v_0(x)\ra^2 \, dx.
\end{align*}
Since $Av=w$ and $\| w_0 \|_{L^2} = \| Av_0 \|_{L^2} \leq C \| v_0 \|_{L^2}$, 
we obtain 
$$
\int_0^T\int_\omega  \la \RE w\ra^2 \, dxdt
\ge \frac{K}{4} \int_0^{2\pi} \la v_0(x)\ra^2 \, dx
\ge K' \int_0^{2\pi} \la w_0(x)\ra^2 \, dx,
$$
which completes the proof.
\end{proof}

Now we prove a technical result used in the proof above.

\begin{lemm}\label{T:65}
Consider a pseudo-differential $A$ with symbol $q(x,\xi)\exp ( i \beta(x)|\xi|^\mez )$. 
There exist universal positive constants $\delta,C$ such that, 
if $\lA \beta\rA_{H^3}+\la q-1 \ra_{3}\le \delta$, then 
$| \langle (A-I) u \rangle | \leq C \delta \| u \|_{L^2}$ for all $u \in L^2(\T)$. 
\end{lemm}

\begin{proof}
Like in the proof of Corollary \ref{P45c}, we split $A = A_0 + A_1$, with 
$$
A_0 \defn \Op\big(\exp\big(i\beta(x)|\xi|^\mez\big)\big),\quad 
A_1 \defn \Op\big((q(x,\xi)-1)\exp\big(i\beta(x)|\xi|^\mez\big)\big).
$$
Directly from statement $ii)$ in Lemma~\ref{L:26} we have $\| A_1 \|_{\Lr(L^2)} \leq C \delta$, 
whence $| \langle A_1 u \rangle | \leq C \delta \| u \|_{L^2}$. 
To estimate $(A_0 - I)$, let $u(x) = \sum_{n \in \xZ} u_n e^{inx}$, and calculate 
\[
\int_\T (A_0 - I)u \, dx = \sum_{n \neq 0} u_n c_n, \quad 
c_n = \int_\T e^{i(nx+|n|^\mez \b(x))} \, dx.
\]
Integrating by parts, one has 
\[
c_n 
= \int_\T \frac{\pa_x \{ e^{i(nx+|n|^\mez \b(x))} \} }{i (n + |n|^\mez \pa_x \b(x))} \, dx 
= \int_\T \frac{- i \pa_{xx}\b(x) \, e^{i(nx+|n|^\mez \b(x))}}{|n|^\tdm ( 1 + |n|^\mez n^{-1} \pa_x \b(x))^2} \, dx 
\]
so that, for $|\pa_x \b| \leq 1/2$, 
\[
|c_n| \leq C \| \b \|_{H^2} |n|^{-\tdm} \quad \forall n \in \xZ \setminus \{ 0 \}.
\]
Thus $(\sum |c_n|^2)^\mez \leq C \| \b \|_{H^2}$, and by H\"older's inequality the lemma follows.
\end{proof}

\section{Controllability}\label{P:S7}

\def\widePtrois{Q}
\def\QR{R}
\def\QRquatre{\mathcal{R}}

Consider an operator of the form
$$
\widePtrois\defn \partial_t  +\V\px +iL +\QR,
$$
where $\V$ is a real-valued function and $\QR$ is an operator of order $0$. 
In this section we study the following control problem: 
given a time $T>0$, a subset $\omega \subset \T$ 
and an initial data $w_{in}\in L^2(\xT)$, 
find a (possibly) complex-valued function $f\in C^0([0,T];L^2)$ 
such that the unique solution $w\in C^0([0,T];L^2)$ of
\be\label{p510a}
\widePtrois w=\chi_\omega \RE f, \quad w(0)=w_{in}
\ee
satisfies $w(T)=0$. 
We study this control problem by means of an adaptation of the classical HUM method. 
We need to adapt the standard argument since we want to prove the existence of a real-valued control, while the unknown is complex-valued. 
In particular, for this reason, one cannot obtain $w(T)=0$. 
We prove instead that, for any real-valued function $M$ such that the $L^\infty$-norm of $M-1$ is small enough, one can find a control such that $w(T,x)=ibM(x)$ for some constant $b\in \xR$.
We remark that, given $f$ and $w_{in}$, the existence of a unique solution $w$ to \e{p510a}Ê
is proved in the appendix, see Lemma~\ref{L2}. 

We prove not only a control result 
but also a {\em contraction estimate}, which is the main technical result of this section. This means that we estimate the difference of two controls $f$ and $f'$ associated with different functions $W,W'$ or remainders $R,R'$. This contraction estimate is the key estimate to prove later that the nonlinear scheme converges (using a Cauchy sequence argument). To prove this contraction estimate we introduce an auxiliary control problem which, loosely speaking, interpolates the two control problems. Since the original nonlinear problem is quasi-linear, a loss of derivative appears. 
This means that to estimate the $C^0([0,T];L^2)$-norm of $f-f'$ we need to have a bound for the $C^0([0,T];H^1)$-norms of $f$ and $f'$. This is why we prove and use a regularity property of the control, namely the fact that the control is in $C^0([0,T];H^\mu(\xT))$ whenever the initial data $w_{in}$ is in $H^\mu(\xT)$. 
This regularity result is proved by an adaptation of an argument used by Dehman-Lebeau~(\cite{DehmanLebeau}) 
and Laurent (\cite{Laurent}).
Before stating the result, we recall the definition of the adjoint operator 
$\widePtrois^*$, namely 
\be \label{def cal Q}
\widePtrois^* = - \mathcal{Q}, \quad 
\mathcal{Q} := \partial_t + W \px + i L + \QRquatre, \quad 
\QRquatre \defn -\QR^* + (\px \V).
\ee

\begin{prop}\label{P:51}
Consider an open domain $\omega\subset \xT$.  
There exist $r$ and six increasing functions 
$\mathcal{F}_j\colon\xR_+^*\rightarrow \xR_+^*$ ($0\le j\le 5$), 
satisfying $\lim_{T\rightarrow 0}\mathcal{F}_j(T)=0$, 
such that, for any $T > 0$, for any real-valued function $M \in H^{3/2}(\T)$ with $ \| M-1 \|_{L^\infty} \leq \mathcal{F}_0(T)$, the following results hold.

$i)$ {\em Existence.} Consider $\QR\in C^0([0,T];\Lr(L^2))$ and a function  
$\V$ satisfying
$$
\int_\T \V(t,x) \, dx = 0
$$
for any $t\in [0,T]$. Assume that 
the norm 
$$ 
\| (W,R) \|_{r,T} 
\defn  
\sum_{1\le k\le 3} \| \partial_t^k \V \|_{C^0([0,T];H^1)}
+ \lA \V \rA_{C^0([0,T];H^r)} + \lA \QR \rA_{C^0([0,T] ; \Lr(L^2))} 
$$ 
satisfies
\be\label{n510}
\lA (W,R)\rA_{r,T}\le \mathcal{F}_1(T).
\ee
Then there exists an operator 
$\controltroisT\colon L^2\rightarrow C^0([0,T];L^2)$ 
such that for any $w_{in}\in L^2$, setting 
$f\defn\controltroisT(w_{in})$, 
the unique solution $w\in C^0([0,T];L^2)$ of
\be \label{Qw=control}
\widePtrois w=\chi_\omega \RE f, \quad w(0)=w_{in}
\ee
satisfies
\be\label{p510c}
w(T,x)=ibM(x)
\ee
for some constant $b\in\xR$, and  
\be\label{n510a}
\lA f\rA_{\bo([0,T];L^2)}\le \frac{\lA w_{in}\rA_{L^2}}{\mathcal{F}_2(T)} \,.
\ee

$ii)$ {\em Uniqueness.} For any $w_{in}\in L^2(\xT)$ and any $T > 0$, 
$\Theta_{M,T}(w_{in})$ is determined as the unique 
function $f\in C^0([0,T];L^2(\xT))$ satisfying the two following conditions:
\begin{enumerate}
\item There holds $\widePtrois^* f = 0$ and $\IM \int_\xT M(x)f(T,x)\, dx =0$. 
\item The solution $w$ of \eqref{Qw=control} satisfies \eqref{p510c} 
for some constant $b\in\xR$.
\end{enumerate}

$iii)$ {\em Regularity}. Let $\mu\in [0,3/2]$ and 
consider $w_{in}\in H^{\mu}(\xT)$. 
If
\be\label{n510-mu}
\| (W,R) \|_{r,T} + \lA \QR \rA_{C^0([0,T] ; \Lr(H^\mu))}  \le 
\mathcal{F}_1(T),
\ee
then $\controltroisT(w_{in})$ is in $C^0([0,T];H^{\mu}(\xT))$ and
\be\label{p51a}
\lA \controltroisT(w_{in})\rA_{\bo([0,T];H^\mu)}\le \frac{\lA w_{in}\rA_{H^\mu}}{\mathcal{F}_3(T)}.
\ee

$iv)$ {\em Stability}. Consider two pairs $(W,\QR)$ and $(W',\QR')$, 
where $(W,\QR)$ is defined for $t \in [0,T]$ and satisfies \eqref{n510-mu} with $\mu=3/2$, 
and $(W',\QR')$ is defined for $t \in [0,T']$ and satisfies \eqref{n510-mu} (with $\mu=3/2$ and $T'$ instead of $T$). 
Denote by $\Theta_{M,T}$ and $\Theta'_{M,T'}$ the operators associated to these two pairs. 
Consider the time-rescaling operator $\mT$ defined by 
\be \label{trop}
(\mT h)(t) := h(\lm t), \quad \lm := \frac{T}{T'}\,,
\ee  
and let $\tilde W := \mT W$, $\tilde R := \mT R$, namely $\tilde R(t) = R(\lm t)$. 
Then, given any $w_{in} \in L^2(\T)$, 
\begin{multline} \label{z534a}
\| \Theta'_{M,T'}(w_{in}) - \mT \Theta_{M,T}(w_{in})  \|_{C^0([0,T'];L^2)} 
\\ \le \frac{\| w_{in} \|_{H^{3/2}}}{\mathcal{F}_4(T)}
\Big( |1-\lm| + \| W' - \tilde W \|_{C^0([0,T'];H^2)} 
+ \| R' - \tilde R \|_{C^0([0,T'];\Lr(L^2))}\Big).
\end{multline}

$v)$ {\em Dependence in $M$}. Consider $M,M'$ 
in $H^{3/2}(\xT)$ with $\lA M\rA_{L^\infty}+\lA M'\rA_{L^\infty}
\le \mathcal{F}_0(T)$. 
If $\lA (W,R) \rA_{r,T} \le 
\mathcal{F}_1(T)$, then, for all $w_{in}\in H^1(\xT)$,
\be\label{z539M}
\lA (\Theta_{M,T}-\Theta_{M',T})(w_{in})\rA_{\bo([0,T];L^2)}\\
\le 
\frac{1}{\mathcal{F}_5(T)}\lA M-M'\rA_{L^\infty}\lA w_{in}\rA_{H^1}.
\ee
\end{prop}

In this section we often use the notation $A\les B$ to say that $A\le C B$ for some constant $C$ depending only on $T$. 
The key result is the following lemma. 

\begin{lemm}\label{L3}
Introduce the space
$$
L^2_{M}\defn \bigg\{\varphi\in L^2(\xT;\xC)\,;\, \IM \int_\xT M(x)\varphi(x)\, dx =0\bigg\}. 
$$
For any $w_{in}\in L^2(\xT)$, there exists a unique $f_1\in L^2_{M}$ such that,
$$
\forall \phi_1\in L^2_{M},\quad 
\RE \int_0^T \big(  \chi_\omega \RE f(t), \phi(t)\big) dt
=- \RE ( w_{in}, \phi(0)),
$$
where $f$ and $\phi$ are the unique functions in $C^0([0,T];L^2(\xT))$ satisfying
\be \label{p8n1} 
\begin{cases}
\mathcal{Q}f=0 \\ f(T) = f_1 
\end{cases}
\qquad 
\begin{cases} 
\mathcal{Q}\phi=0 \\ \phi(T) = \phi_1,
\end{cases}
\ee
where $\mathcal{Q}$ is given by \e{def cal Q} 
(the existence of $f$ and $\phi$ follows from Lemma~\ref{L2}.) 
We set
$$
\controltroisT (w_{in}) \defn f.
$$
Moreover \e{n510a} holds.
\end{lemm}
\begin{proof}
The space $L^2_{M}$ 
is an $\xR$-vector space. 
Introduce the $\xR$-bilinear symmetric map $a(\cdot,\cdot)$ defined by
\begin{align}
a(f_1,\phi_1)
\defn {} & \RE \int_0^T \int_\xT \chi_\omega(x)  \RE (f(t,x))\, \overline{\phi(t,x)}\, dx dt \label{p8n2}
\\
= {} & \int_0^T \int_\xT \chi_\omega(x)  \RE (f(t,x))\, \RE \phi(t,x)\, dx dt.\notag
\end{align}
This application is well defined and continuous. Indeed, 
it follows from the $L^2$-energy estimate (see~\e{n26}) that  
\begin{align}
|a(f_1,\phi_1)| 
& \le \int_0^T\int_\xT  \la f \ra \la \phi \ra \, dxdt \notag 
\\
& \le T \lA f \rA_{\bo([0,T];L^2)}\lA \phi \rA_{\bo([0,T];L^2)} 
\le C(T) \lA f_1 \rA_{L^2} \lA \phi_1 \rA_{L^2}.
\label{p510.5z} 
\end{align}
Since $\chi_\omega(x)=1$ for $x$ in an open subset $\omega_1\subset \omega$, one has
$$
a(f_1,f_1)\ge \int_0^T\int_{\omega_1}  ( \RE f )^2 \, dxdt.
$$
If $f_1\in L^2_{M}$ then $\IM \int_{\xT} M f_1 \, dx=0$ and we have
$$
\la \IM \int_\xT f_1(x)\,dx\ra=\la\IM\int_\xT (1-M(x))f_1(x)\, dx\ra\le  
\lA M-1\rA_{L^\infty}\sqrt{2\pi}\lA f_1\rA_{L^2}
$$
from which (using $| \RE z| \geq |z| - |\IM z|$) we deduce that
$$
\la \RE \langle  f_1\rangle \ra\ge  \la  \langle f_1\rangle \ra
-\lA M-1\rA_{L^\infty}\sqrt{2\pi}\lA f_1\rA_{L^2}.
$$
For $\lA M-1\rA_{L^\infty}$ small enough, one can apply the observability inequality proved in the previous section (see Corollary~\ref{C:54} and Remark~\ref{R:75}) to conclude that 
\be\label{p519a}
C_1(T)\lA f_1\rA_{L^2}^2\le a(f_1,f_1).
\ee
On the other hand, \e{p510.5z} implies that 
$a(f_1,f_1)\le C(T)\lA f_1\rA_{L^2}^2$. 
Hence $a(\cdot, \cdot)$ is a real scalar product on $L^2_M$ which induces the norm 
$N(f_1) = \sqrt{a(f_1,f_1)}$, which is equivalent to the norm  
$\lA \cdot \rA_{L^2(\xT,\xC)}$ on $L^2_{M}$. 
Now,  
Lemma~\ref{L2} implies that the mapping $\phi_1\mapsto \phi(0)$ is $\xR$-linear and bounded from 
$L^2_{M}$ into $L^2$ and hence $\phi_1\mapsto \Lambda(\phi_1)\defn -\RE ( w_{in}, \phi(0))$ is a bounded $\xR$-linear form on $L^2_{M}$. Therefore, the Riesz theorem implies that, for any $\xR$-linear form $\Lambda$ on $L^2_{M}$, there is a unique $f_1\in L^2_{M}$ such that 
$a(f_1,\phi_1)=\Lambda(\phi_1)$ for all $\phi_1\in L^2_{M}$, together with 
\be\label{n510c}
\lA f_1\rA_{L^2}\le \frac{\lA \Lambda\rA}{C_1(T)}.
\ee
Moreover \e{n510a} 
follows from \e{n510c} and the bound $\lA f\rA_{\bo([0,T];L^2)}\les \lA f_1\rA_{L^2}$ already used. 
\end{proof}

\begin{proof}[Proof of Proposition \ref{P:51}]
\emph{Proof of statement $i)$}.
We begin by proving that if $M\in H^{3/2}(\xT)$ then 
$H^{3/2}(\xT)\cap L^2_M$ is dense in $(L^2_M, \| \cdot \|_{L^2})$.  
To see this, 
let $\Pi_N$ be the Fourier truncation operator defined by 
$\Pi_N h(x) = \sum_{|j| \leq N} h_j e^{ijx}$ 
where $h(x) = \sum_{j \in \xZ} h_j e^{ijx}$. 
Given $u\in L^2_M$, define $u_N\defn \frac{1}{M} \, \Pi_N (Mu)$. 
Since the operator $\Pi_N$ preserves the mean, one has that $u_N \in L^2_M$.  
Moreover, since $u \in L^2$, one has $Mu \in L^2(\xT)$, $\Pi_N(Mu) \in C^\infty(\xT)$, and hence 
$M^{-1} \Pi_N(Mu) \in H^{3/2}(\xT)$ since $M^{-1}\in H^{3/2}(\xT)$. 
Since $(u_N)$ converges to $u$, this proves that
$H^{3/2}(\xT)\cap L^2_M$ is dense in $L^2_M$.

Now let $f$ be as given by the previous lemma. 
It is proved in the appendix that 
there is a unique solution $w$ in $C^0([0,T];L^2(\xT))$ of \eqref{Qw=control}. 
Our goal is to prove that $w(T)$ satisfies \e{p510c}. To do so we first check 
that \e{p510c}Ê will be proved if $\RE (w(T),\phi_1)=0$ for all $\phi_1$ in $L^2_{M}$. 
Indeed, one has
$$
\RE (w(T),\phi_1)=\int (\RE w(T,x))\RE \phi_1(x)\, dx
+\int (\IM w(T,x))\IM \phi_1(x)\, dx=0
$$
for all  $\phi_1\in L^2_M$. Therefore 
we obtain $\int (\RE w(T,x))f(x)\, dx=0$ for any 
real-valued function $f$ and 
$\int (\IM M(x)^{-1}w(T,x))g(x)\, dx=0$ for any real-valued function $g$ with 
$\int g(x)\, dx=0$. This implies that \e{p510c} holds.

We now have to prove that 
$\RE (w(T),\phi_1)=0$ for any $\phi_1$ in $L^2_{M}$. By the density argument 
proved above, it is enough to assume that $\phi_1\in L^2_M\cap H^{3/2}(\xT)$. 
Given $\phi_1\in L^2_{M}\cap H^{3/2}(\xT)$, 
let $\phi\in C^0([0,T];H^{3/2}(\xT))$ be such that
\be \label{cal Q phi}
\mathcal{Q}\phi=0, \qquad \phi(T) = \phi_1.
\ee
Since $\mathcal{Q}=-\widePtrois^*$, multiplying the equation \eqref{Qw=control} 
by $\overline{\phi}$ and integrating by parts, we find that
\be\label{p75a}
\big( w(T),\phi_1\big) = \big( w(0),\phi(0)\big) + \int_0^T \big( \chi_\omega \RE f\, ,\, \phi\big)\, dt+\int_0^T \big( w \, ,\, \mathcal{Q}\phi\big)\, dt.
\ee
Notice that the integration by parts is justified since 
$\phi\in C^1([0,T];L^2(\xT))$. 
By definition of $\phi$ the last term in the right-hand side vanishes and, by definition of $f$, 
the real part of the 
sum of the first and second terms vanishes. 
This proves 
that $\RE \big( w(T),\phi_1\big)=0$, which 
concludes 
the proof of statement $i)$.

\medskip

\emph{Proof of statement $ii)$.} 
Recall that $\mathcal{Q}=-Q^*$ is given by \e{def cal Q}. 
Consider $\phi_1\in L^2_M$ and denote by $\phi$ 
the unique function in $C^0([0,T];L^2(\xT))$ satisfying \eqref{cal Q phi}.  
As in \e{p75a}, multiplying both sides of the equation $\widePtrois w=\chi_\omega \RE f$ 
by $\phi$, integrating by parts 
one obtains \eqref{p75a}. 
Since $\phi_1\in L^2_M$ and $w(T,x)=ibM(x)$ for some constant $b\in\xR$, one has 
$\RE \big( w(T),\phi_1\big)=0$. Therefore, since $\mathcal{Q}\phi=0$,  
$$
\RE \int_0^T \big( \chi_\omega \RE f\, ,\, \phi\big)\, dt = -\RE \big( w_{in},\phi(0)\big).
$$
Since $\mathcal{Q}f=0$ and $f(T) \in L^2_M$ by assumption, 
and since the function $f_1$ whose existence is given by Lemma~\ref{L3} is unique, 
one deduces that $f(T) = f_1$. 
Hence $f=\Theta_{M,T}(w_{in})$ by uniqueness of the solution to the Cauchy problem~\e{p8n1}. 

\medskip

\emph{Proof of statement $iii)$.}
We prove \e{p51a}. 
In view of the energy estimate \e{pn27b}, it is sufficient to prove that 
$\lA f_1\rA_{H^\mu}$ is controlled by $\lA w_{in}\rA_{H^\mu}$. 
We prove only an {\em a priori} estimate, assuming that $f_1$ belongs to $H^\mu(\xT)$. 
To estimate $\lA f_1\rA_{H^\mu}$, we adapt to our setting 
an argument used by 
Dehman-Lebeau~(see \cite[Lemma 4.2]{DehmanLebeau})
and Laurent (see \cite[Lemma 3.1]{Laurent}). 

Consider the mapping
$$
S \colon L^2_M \to L^2(\xT), \quad 
S \colon f_1 \in L^2_M \mapsto f\mapsto w \mapsto w(0) \in L^2(\xT)
$$ 
where $f$ and $w$ are the unique functions in $C^0([0,T];L^2(\xT))$ 
successively determined by the backward Cauchy problems with data at time $T$
\[ 
\begin{cases}
\mathcal{Q}f=0 \\ f(T) = f_1, \end{cases}
\qquad 
\begin{cases} 
\widePtrois w=\chi_\omega \RE f \\ 
w(T) = 0. \end{cases}
\]
It follows from statements $i)$ and $ii)$ that $S$ is an isomorphism of 
$L^2_M$ onto $L^2(\xT)$. 
As a result, with $\Lambda^\mu=(I-\px^2)^{\mu/2}$, one can write
\be\label{p72e}
\lA f_1 \rA_{H^\mu}=\lA \Lambda^\mu f_1\rA_{L^2}\les \lA S \Lambda^\mu f_1\rA_{L^2}.
\ee
Now we have to commute $S$ and $\Lambda^\mu$. 
This amounts to compare $(\Lambda^\mu f,\Lambda^\mu w)$ with $(f',w')$ defined by 
\[
\begin{cases}
\mathcal{Q}f'=0 \\ f'(T) = \Lambda^\mu f_1, \end{cases} 
\qquad 
\begin{cases} 
\widePtrois w' = \chi_\omega \RE f' \\ w'(T) = 0. 
\end{cases}
\]
We first estimate $f'-\Lambda^\mu f$ 
and then deduce an estimate for $w'-\Lambda^\mu w$.  Write
$$
\mathcal{Q}(f'-\Lambda^\mu f)
=[\Lambda^\mu,\QRquatre]f+[ \Lambda^\mu, W]\px f,  \qquad 
(f'-\Lambda^\mu f)_{\arrowvert t=T}=0,
$$
and use the energy estimate \e{n26b} to find that
\be\label{p78e}
\lA f'-\Lambda^\mu f\rA_{C^0([0,T];L^2)}\les  \lA [\Lambda^\mu,\QRquatre]f + [\Lambda^\mu, W]\px f\rA_{L^1([0,T];L ^2)}.
\ee
Similarly,
\begin{align}
&\lA w'-\Lambda^\mu w \rA_{C^0([0,T];L^2)}\les  \lA \mathcal{F}\rA_{L^1([0,T];L ^2)}
\quad \text{where }\label{p79e}\\
&\mathcal{F}\defn  
\chi_\omega \RE (f'-\Lambda^\mu f) +
[\Lambda^\mu,\QR]w+[\Lambda^\mu,W]\px w
-[\Lambda^\mu,\chi_\omega]\RE f. \notag
\end{align}
By \e{p78e} and the obvious embedding $C^0([0,T];L^2)\subset 
L^1([0,T];L^2)$, we deduce that
\begin{align*}
\lA \mathcal{F}\rA_{L^1([0,T];L ^2)}
&\les \lA [\Lambda^\mu,\chi_\omega]\RE f\rA_{C^0([0,T];L^2)}
\\&\quad
+ \lA [\Lambda^\mu,\QRquatre]f\rA_{C^0([0,T];L^2)}
+ \lA [\Lambda^\mu,\QR]w \rA_{C^0([0,T];L^2)}
\\
& \quad \ + \lA [\Lambda^\mu,W]\px f\rA_{C^0([0,T];L^2)}
+\lA [\Lambda^\mu,W]\px w\rA_{C^0([0,T];L^2)}.
\end{align*}

To estimate the commutators $[\Lambda^\mu,\chi_\omega]$ and $[\Lambda^\mu,W]$, we use the classical estimate
$$
s>\tdm,~ 0\le \mu\le s~\Rightarrow ~\Vert [\Lambda^\mu, W] u \Vert_{L^2} \leq K \Vert W \Vert_{H^s}  \Vert u \Vert_{H^{\mu- 1}}.
$$
On the other hand, to estimate the commutator 
$[\Lambda^\mu,\QRquatre]$ (or $[\Lambda^\mu,\QR]$) we estimate separately 
$\Lambda^\mu \QRquatre$ and $\QRquatre \Lambda^\mu$. 
Recalling that $\QRquatre = -\QR^*+(\px \V)$, we conclude that
$$
\lA \mathcal{F}\rA_{L^1([0,T];L ^2)} 
\les \| f \|_{C^0([0,T];H^{\mu-1})} 
+ a \lA (f,w)\rA_{C^0([0,T];H^\mu)}
$$
where
$$
a \defn \lA W \rA_{C^0([0,T]; H^3)}
+ \lA \QR \rA_{C^0([0,T];\mathcal{L}(H^\mu)\cap \Lr(L^2))}.
$$
Notice that 
$a\leq \lA (W,R)\rA_{r,T}+ \lA \QR \rA_{C^0([0,T];\mathcal{L}(H^\mu))}$ 
where $\lA (W,R)\rA_{r,T}$ is as defined above \eqref{n510}. 

Now, using the energy estimate \e{n26-H1}, we have
$\lA f\rA_{C^0([0,T];H^\mu)}\les  \lA f_1\rA_{H^\mu}$. 
Using again \e{n26-H1} and the equation satisfied by $w$, 
we deduce that $\lA w\rA_{C^0([0,T];H^\mu)}\les  \lA f_1\rA_{H^\mu}$. 
Thus, by \e{p79e}, we find 
$$
\lA w'-\Lambda^\mu w \rA_{C^0([0,T];L^2)} 
\les \| f_1 \|_{H^{\mu-1}} + a \lA f_1\rA_{H^\mu}.
$$
In particular, at $t=0$, we get  
$
\lA w'(0)-\Lambda^\mu w(0) \rA_{L^2}
\les \| f_1 \|_{H^{\mu-1}} + a \lA f_1\rA_{H^\mu}.
$
Now, by definition, $w'(0)=S\Lambda^\mu f_1$ while 
$\Lambda^\mu w(0)=\Lambda^\mu w_{in}$. 
Therefore, by triangle inequality, 
\be \label{ppp 1}
\lA S \Lambda^\mu f_1\rA_{L^2} \les  
\lA w_{in}\rA_{H^\mu} + \| f_1 \|_{H^{\mu-1}} + a \lA f_1\rA_{H^\mu}.
\ee
For $\mu \in [0,1]$, one has $\| f_1 \|_{H^{\mu-1}} \leq \| f_1 \|_{L^2} \les \| w_{in} \|_{L^2}$, 
and therefore 
\be \label{ppp 2}
\lA S \Lambda^\mu f_1\rA_{L^2} \les  
\lA w_{in}\rA_{H^\mu} + a \lA f_1\rA_{H^\mu}. 
\ee
Plugging this bound into \e{p72e}Ê yields 
$\lA f_1\rA_{H^\mu}\les  \lA w_{in}\rA_{H^\mu} +a \lA f_1\rA_{H^\mu}$.
By taking $a$ small enough, 
we conclude that
\be\label{ppp 3}
\lA f_1\rA_{H^\mu}\les  \lA w_{in}\rA_{H^\mu}.
\ee

For $\mu \in (1,\frac32]$ we go back to \e{ppp 1} and note that
$\| f_1 \|_{H^{\mu-1}} \les \| w_{in} \|_{H^{\mu-1}}$ because $\mu-1 \in [0,1]$. 
Hence \e{ppp 2} holds, and we reach the same conclusion as above. 
This completes the proof of statement $iii)$.

\medskip

\emph{Proof of statement $iv)$.} 
Given $w_{in}$, let $f_1$ and $f_1'$ in $L^2_M$  
be as given by Lemma~\ref{L3}, so that $f := \Theta_{M,T}(w_{in})$ and 
$f' := \Theta'_{M,T'}(w_{in})$ are determined by the Cauchy problems
\[
\begin{cases} \mathcal{Q}f=0 \ \text{ on } [0,T] \\ 
f(T) = f_1 
\end{cases}
\qquad 
\begin{cases} \mathcal{Q}'f'=0 \ \text{ on } [0,T'] \\ 
f'(T') = f'_1 
\end{cases}
\]
where 
\[
\mathcal{Q}' := \partial_t + i L + W' \px + \QRquatre', \quad  
\QRquatre' \defn -(\QR')^* + (\px \V').
\] 
Similarly, we denote $Q' = - (\mathcal{Q}')^* = \pa_t + i L + W' \pa_x + R'$. 
By definition of $f, f'$, the unique solutions $w \in C^0([0,T];L^2(\xT))$ and 
$w' \in C^0([0,T'];L^2(\xT))$ of the two Cauchy problems
\be \label{two Cp}
\begin{cases} 
\widePtrois w = \chi_\omega \RE f \ \text{ on } [0,T] \\ 
w(0) = w_{in} 
\end{cases}
\qquad 
\begin{cases} 
\widePtrois' w' = \chi_\omega \RE f' \ \text{ on } [0,T'] \\ 
w'(0) = w_{in} 
\end{cases}
\ee
satisfy $w(T)=ibM$, $w'(T') = ib' M$ for some $b, b'\in\xR$.
The idea now is to introduce an auxiliary control problem. 
Let $f''\in C^0([0,T];L^2(\xT))$ be the unique solution of 
\be \label{reno 3}
\mathcal{Q}  f'' =0  \ \text{ on } [0,T], \qquad  f''(T) = f'_1,
\ee
so that $f''$ solves the same equation as $f$ and it has the same Cauchy data as $f'$.
Then introduce $w''$ as the unique solution to 
\be \label{reno 4}
\widePtrois w'' = \chi_\omega \RE  f'' \ \text{ on } [0,T], \qquad w''(T) = i b' M 
\ee
and set $w''_{in}\defn w''(0)$. 
By uniqueness (see statement $ii)$) we deduce  
that $f''$ is the control for the operator $Q$ associated to $w''_{in}$, that is
$$
f''=\controltroisT (w''_{in}).
$$
Then, by continuity (see statement $i)$) one has
$$
\lA f-f''\rA_{C^0([0,T];L^2)} =\lA \controltroisT (w_{in}-w''_{in})\rA_{C^0([0,T];L^2)}
\les \lA w_{in}-w''_{in}\rA_{L^2}.
$$
Let $\tilde f := \mT f$ and $\tilde f'' := \mT f''$. 
Then
\be \label{reno 1}
\| \tilde f - \tilde f'' \|_{C^0([0,T'];L^2)} = \| f - f'' \|_{C^0([0,T];L^2)}
\ee
because 
\be \label{reno 2}
\forall h \in C^0([0,T];L^2), \quad \| \mT h \|_{C^0([0,T'];L^2)} = \| h \|_{C^0([0,T];L^2)}.
\ee 
It remains to estimate $\| f' - \tilde f'' \|_{C^0([0,T'];L^2)}$ 
and $\lA w_{in}-w''_{in}\rA_{L^2}$. We begin with $f'-\tilde f''$. 
Since $f''$ solves \e{reno 3}, 
$\tilde f''$ satisfies
\[
\tilde \mQ \tilde f'' = 0 \ \text{ on } [0,T'], \qquad 
\tilde f''(T') = f'_1 
\]
where 
\[
\tilde \mQ := \partial_t + i \lm L + \lm \tilde W \px + \lm \tilde \mR,
\] 
and $\tilde W := \mT W$, $\tilde \mR := \mT \mR$ (namely $\tilde \mR(t) := \mR(\lm t)$). 
By difference, one has 
$$
\tilde \mQ( f' - \tilde f'') = F_0, \quad (f' - \tilde f'')(T') = 0, 
$$
where
\[
F_0 := (\lm-1) (i L + \tilde W \pa_x + \tilde \mR) f' + (\tilde W - W') \pa_x f' + (\tilde \mR - \mR') f'.
\] 
In order to apply the $L^2$-energy bound \e{n26b}, we estimate $F_0$. 
Using the regularity property of the control operator $\Theta_{M,T}$ (see statement $iii)$) we have 
\begin{align} \label{reno 6} 
\| L f' \|_{C^0([0,T'];L^2)} 
& \les \| f' \|_{C^0([0,T'];H^{\frac32})} 
\les \| w_{in} \|_{H^{\frac32}},
\\ 
\| (\tilde W - W')\px f' \|_{C^0([0,T'];L^2)}
& \les \| \tilde W - W' \|_{C^0([0,T'];H^1)} \| f' \|_{C^0([0,T'];H^1)}\\
&\les \| \tilde W - W' \|_{C^0([0,T'];H^1)} \| w_{in} \|_{H^1}.
\notag 
\end{align}
Similarly
\begin{align*}
&\| (\tilde \mR - \mR') f' \|_{C^0([0,T'];L^2)} \\
&\qquad \les \| \tilde \mR - \mR' \|_{C^0([0,T'];\Lr(L^2))} \| f' \|_{C^0([0,T'];L^2)} 
\notag \\
&\qquad \les \big( \| \tilde R - R' \|_{C^0([0,T'];\Lr(L^2))} + \| \tilde W - W' \|_{C^0([0,T'];H^2)} \big) 
\| w_{in} \|_{L^2}
\end{align*}
where 
$\tilde R := \mT R$ (namely $\tilde R(t) = R(\lm t)$). 
Using \e{n26b} we conclude that 
\begin{multline} \label{z100}
\| f'-\tilde f'' \|_{C^0([0,T'];L^2)} \\
\les  \| w_{in} \|_{H^{\frac32}}
\big( |\lm-1| + \| \tilde W - W' \|_{C^0([0,T'];H^2)} + \| \tilde R - R' \|_{C^0([0,T'];\Lr(L^2))} \big).
\end{multline}
It remains to estimate $\| w_{in}-w''_{in} \|_{L^2}$. 
Let $\tilde w'' := \mT w''$. 
At $t=0$ one has $w'(0) - \tilde w''(0) = w_{in} - w''_{in}$, hence we study the difference $w' - \tilde w''$. 
Since $w''$ solves \e{reno 4}, $\tilde w''$ satisfies
\be \label{reno 5}
\tilde Q \tilde w'' = \lm \chi_\omega \RE \tilde f'' \ \text{ on } [0,T'], \qquad 
\tilde w''(T') = i b' M,
\ee
where
\[
\tilde Q := \partial_t + i \lm L + \lm \tilde W \px + \lm \tilde R.
\] 
By difference, 
\[
\tilde Q (w'-\tilde w'') = F, \quad (w' - \tilde w'')(T') = 0, 
\]
where 
\[
F \defn \chi_\omega \RE ( f'- \lm \tilde f'') + (\lm-1) ( i L + \tilde W \pa_x + \tilde R) w' 
+ (\tilde W - W') \pa_x w' + (\tilde R - R') w'.
\]
In order to apply the $L^2$-energy bound \e{n26b}, we estimate $F$. 
First, $f' - \lm \tilde f'' = \lm (f' - \tilde f'') + (1-\lm) f'$, 
and we have already estimated both $f'-\tilde f''$ (see \e{z100}) and $f'$.
For the other terms in $F$ we proceed as above,  
recalling that $\| w' \|_{C^0([0,T'];L^2)} \les \| w_{in} \|_{L^2}$. 
Also, since $w'$ solves the Cauchy problem \eqref{two Cp}, 
we deduce from \e{n26-H1} and the second inequality in \e{reno 6} that 
$\blA w' \brA_{C^0([0,T'];H^{3/2})} \les \lA w_{in}\rA_{H^{3/2}}$.
As a consequence, also $\| F \|_{C^0([0,T'];L^2)}$ is bounded by the term on the right-hand side of \eqref{z100}. 
Then, applying the energy inequality \e{n26b}, we deduce 
that $\| w' - \tilde w'' \|_{C^0([0,T'];L^2)}$ satisfies the same bound. 
In particular, at time $ t = 0$, this yields the desired bound 
for $w_{in} - w''_{in} = (w'- \tilde w'')(0)$.

\emph{Proof of statement $v)$.} 
We begin by introducing some notations which are used in the proofs of statements 
$v)$ and $vi)$. 
As already mentioned, 
it follows from Lemma~\ref{L2} that there exists an operator 
$E_T\colon L^2(\xT)\rightarrow C^0([0,T];L^2(\xT))$ such that 
$v=E_T(v_1)$ is the unique solution to the Cauchy problem 
$\mathcal{Q}v=0$ with data $v(T)=v_1$. Moreover
\be\label{z102}
\lA E_T(v_1)\rA_{C^0([0,T];L^2(\xT))}\les \lA v_1\rA_{L^2}.
\ee
Now recall that by definition
\be\label{z103}
a_T(f_1,\phi_1)
\defn {}  \RE \int_0^T \int_\xT \chi_\omega(x)  \RE (E_T(f_1))\, \overline{E_T(\phi_1)}\, dx dt .
\ee
Also introduce the mapping $\Lambda\colon L^2(\xT)\rightarrow \xR$ defined by $\Lambda(v_1)=-\RE (w_{in},E_T(v_1)(0))$ where 
$E_T(v_1)(0)=E_T(v_1)\arrowvert_{t=0}$. It follows from Lemma~\ref{L3} that 
there exist two functions $f_1 \in L^2_M$ and $f_1' \in L^2_{M'}$ such that
\begin{align*}
&\forall \phi_1\in L^2_M,\quad a_T(f_1,\phi_1)=\Lambda(\phi_1) ,\\
&\forall \phi_1\in L^2_{M'},\quad a_T(f_1',\phi_1)=\Lambda(\phi_1).
\end{align*}
Then $\Theta_{M,T}(w_{in})-\Theta_{M',T} (w_{in})=E_T(f_1-f_1')$. In view 
of \e{z102}, to prove 
statement $v)$ it is sufficient to estimate $f_1-f_1'$. To do so, 
we need to compare elements in $L^2_M$ and elements in 
$L^2_{M'}$. Observe that, by definition of $L^2_M$, if 
$\varphi\in L^2_{M'}$ then $(M'/M)\varphi\in L^2_{M}$. 
Therefore $\varphi_1\defn f_1-\frac{M'}{M}f_1'$ 
belongs to $L^2_M$ and we can use \e{p519a} to deduce that 
\be\label{z103z}
\lA f_1-\frac{M'}{M}f_1'\rA_{L^2}^2\les 
a_T\left(f_1-\frac{M'}{M}f_1',f_1-\frac{M'}{M}f_1'\right)=
a_T\left(f_1-\frac{M'}{M}f_1',\varphi_1\right).
\ee
Now write the last term as the sum $(I) + (II) + (III)$, where 
\begin{align*}
(I)&=a_T(f_1,\varphi_1)-a_T\left(f_1',\frac{M}{M'}\varphi_1\right),\\
(II)&=a_T\left(f_1',\frac{M}{M'}\varphi_1\right)-a_T(f_1',\varphi_1),\\
(III)&=a_T(f_1',\varphi_1)-a_T\left(\frac{M'}{M}f_1',\varphi_1\right).
\end{align*}
(Notice that both $M/M'$ and $M'/M$ appear.) 
Since $(M/M')\varphi_1$ belongs to $L^2_{M'}$, we can write 
$a_T\left(f_1',(M/M')\varphi_1\right)=\Lambda((M/M')\varphi_1)$ to deduce that
$$
(I)=\Lambda(\varphi_1)-\Lambda\left(\frac{M}{M'}\varphi_1\right)=
\Lambda\left(\frac{M'-M}{M'}\varphi_1\right),
$$
so that $\la (I)\ra\les \lA M'-M\rA_{L^\infty}\lA w_{in}\rA_{L^2}\lA\varphi_1\rA_{L^2}$. On the other hand, it follows from the easy estimates \e{p510.5z} and \e{z102} 
that
\begin{align*}
\la (II)\ra +\la (III)\ra&\les 
\lA M-M'\rA_{L^\infty}\lA f_1'\rA_{L^2}
\lA\varphi_1\rA_{L^2}\\
&\les \lA M-M'\rA_{L^\infty}\lA w_{in}\rA_{L^2}
\lA\varphi_1\rA_{L^2}.
\end{align*}
By combining \e{z103z} with the previous estimates we conclude that 
$$
\lA f_1-\frac{M'}{M}f_1'\rA_{L^2}\les 
\lA M-M'\rA_{L^\infty}\lA w_{in}\rA_{L^2}.
$$
Now write
$$
\lA f_1-f_1'\rA_{L^2}\les \lA M-M'\rA_{L^\infty} \lA f_1'\rA_{L^2}
+\lA f_1-\frac{M}{M'}f_1'\rA_{L^2}
$$
and $\lA f_1'\rA_{L^2}\les \lA w_{in}\rA_{L^2}$ 
to complete the proof of statement $v)$.
\end{proof}

\section{Controllability for the paradifferential equation}\label{section9}

\newcommand{\om}{\omega}

We now deduce from the results proved in the previous sections that 
the original equation introduced in Section~\ref{SN:3} is controllable, together with Sobolev estimates for the control. 

Consider a paradifferential operator of the form
\be \label{def P in sec 9}
P = \partial_t  +T_V\partial_x +iL^\mez \big( T_c L^\mez \cdot\big)+ R
\ee
where $R$ is an operator of order $0$. Assume that $P$ satisfies Assumption~\ref{T31A}, 
so that as above $V$ and $c$ are real-valued, $c-1$ is small enough and $P$ satisfies the following structural property: 
\be\label{p517z}
Pu \text{ real-valued }Ê~\Rightarrow ~\frac{d}{dt}\int_\xT \IM u (t,x)\, dx=0.
\ee
Introduce the norm
\begin{multline} \label{norm Xs0s}
\| (c-1, V, R) \|_{X^{s_0,s}(T)}
: = \| (c-1,\pa_t c,V) \|_\Norm{T}{H^{s_0}} 
+ \sum_{k=2,3,4} \| \pa_t^k c \|_\Norm{T}{H^1} 
\\  
+ \sum_{k=1,2,3} \| \pa_t^k V \|_\Norm{T}{H^1} + \| R \|_\Norm{T}{\mL(H^s)} 
+ \| R \|_\Norm{T}{\mL(H^{s + \frac32})}.
\end{multline}
We recall that $p$ is the symbol given by 
$p\defn c^{-\frac{1}{3}}+\frac{5}{18i}\frac{\chi(\xi)\partial_\xi \ell(\xi)}{\ell(\xi)}
c^{-\frac{4}{3}}\px c$ (see \e{p34b}).

\begin{prop}\label{PP10-b}
Consider an open domain $\omega\subset \xT$. There exists $s_0$ large enough and for any 
$s\ge s_0$  
there exist three increasing functions $\mathcal{F}_j\colon\xR_+^*\rightarrow \xR_+^*$ ($1\le j\le3$), 
with $\lim_{T\rightarrow 0}\mathcal{F}_j(T)=0$, 
such that, for any $T\in (0,1]$, 
the following properties hold.

$i)$ If 
\be\label{pn200-b}
\lA (c-1,V,R)\rA_{X^{s_0,s}(T)}\le \mathcal{F}_1(T),
\ee
then there exists a bounded operator
$$
\Theta_{s,T}[(V,c,R)]\colon H^{s+\tdm}(\xT)\rightarrow C^0([0,T];H^{s+\tdm}(\xT))
$$
such that, for any $v_{in}\in H^{s+\tdm}(\xT)$ satisfying
$$
\IM \int_\xT v_{in}(x)\, dx=0,
$$
setting $f\defn \Theta_{s,T}[(V,c,R)](v_{in})$ one has
\be\label{n532}
\lA f\rA_{\bo([0,T];H^{s+\tdm})}
\le \frac{\lA v_{in}\rA_{H^{s+\tdm}}}{\mathcal{F}_2(T)}
\ee
and the unique solution $v$ to $Pv= T_{p}\chi_\omega \RE f,~v_{\arrowvert t=0}=v_{in}$ satisfies
$$
v(T)=0.
$$

$ii)$ Assume that the triple $(c,V,R)$ satisfies \e{pn200-b} and 
\be \label{fur con R}
\| \pa_t^2 c \|_\Norm{T}{H^{s_0}} + \| \pa_t V \|_\Norm{T}{H^{s_0}} + \| \pa_t R \|_\Norm{T}{\mL(H^s)} \leq 1. 
\ee 
Let $(c',V',R')$ be another triple also satisfying the same (corresponding) bounds 
\e{pn200-b} and \e{fur con R}. Then 
\begin{align} \label{n534}
& \| \Theta_{s,T}[(V,c,R)](v_{in})-\Theta_{s,T}[(V',c',R')](v_{in}) \|_\Norm{T}{H^{s}} \\
& \le \frac{\lA v_{in}\rA_{H^{s+\tdm}}}{\mathcal{F}_3(T)} 
\Big\{ \| (c-c', \pa_t(c-c'), V-V') \|_\Norm{T}{H^{s_0}} + \| R-R' \|_\Norm{T}{\Lr(H^s)} \Big\}. \notag 
\end{align}
\end{prop}

\begin{proof}
Let $P$ be given by \e{def P in sec 9}, with $V,c,R$ satisfying \e{pn200-b}. 
We begin by recalling how the various linear operators 
have been defined in the previous sections starting from $P$:
\begin{align*}
& \widetilde{P} 
:= \Lambda_{h,s} P \Lambda_{h,s}^{-1} 
= \partial_t  +T_V\partial_x +iL^\mez ( T_c L^\mez \cdot )+ \Run 
= \partial_t  + V \partial_x +i L^\mez ( c L^\mez \cdot ) + R_2, 
\\
& \wide P_3 := \Phi m^{-1} \widetilde P \Phi^{-1} 
= \pa_t + W \pa_x + i L + \Rtrois,
\\
& \mathcal{P} := - (\widetilde P_3)^* 
= \partial_t +W\px +iL +\Rquatre,
\end{align*}
where $\Phi, m, W$ are given in Proposition~\ref{P:38}, 
\begin{align}
\Run & \defn \Lambda_{h,s} R \Lambda_{h,s}^{-1}
+ [ \Lambda_{h,s},\partial_t ] \Lambda_{h,s}^{-1}
+ [ \Lambda_{h,s}, T_V \px ] \Lambda_{h,s}^{-1} 
+ i [ \Lambda_{h,s}, L^\mez ( T_c L^\mez \cdot ) ] \Lambda_{h,s}^{-1},\notag
\\
\Rdeux u & := \Run u + T_V \px u - V \px u + i (L^\mez T_c L^\mez u - L^\mez (c L^\mez u )),\notag
\\
\Rquatre w & \defn -\Rtrois^*w + (\px \V) w,\label{n534-z}
\end{align}
and $R_3$ has a more involved expression, obtained in Appendix \ref{S:AC1}. 
Moreover $ \wide P = m \Phi^{-1} \wide P_3 \Phi$. 
As a first step in the proof of Proposition \ref{PP10-b}, 
we study the control problem for $\wide P$.

\begin{lemm}\label{P:51-w}
There exist $s_0$ large enough and 
increasing functions $\mathcal{F}_j\colon\xR_+^*\rightarrow \xR_+^*$ ($j=1,2,3$), satisfying $\lim_{T\rightarrow 0}\mathcal{F}_j(T)=0$, 
such that for any $T > 0$ the following result holds. 

$i)$ If 
\begin{multline} \label{n510z} 
\| (c-1,\pa_t c,V) \|_\Norm{T}{H^{s_0}} 
+ \sum_{k=2,3,4} \| \pa_t^k c \|_\Norm{T}{H^1} 
\\  
+ \sum_{k=1,2,3} \| \pa_t^k V \|_\Norm{T}{H^1} + \| \Rdeux \|_\Norm{T}{\mL(L^2)} 
\le \mathcal{F}_1(T), 
\end{multline}
then there exists an operator ${\wide \Theta}_T\colon L^2\rightarrow C^0([0,T];L^2)$ 
such that for any $u_{in}\in L^2$, setting 
$f\defn{\wide \Theta}_T(u_{in})$ one has
\be\label{n510az}
\lA f\rA_{\bo([0,T];L^2)}\le \frac{\lA u_{in}\rA_{L^2}}{\mathcal{F}_2(T)}
\ee
and the unique solution $u$ of
$$
\wide P u=\chi_\omega \RE f, \quad u(0)=u_{in}
$$
satisfies $u(T,x)=ib$ for some $b\in\xR$ and all $x\in \xT$.  
If, in addition, 
\be \label{R2 32}
\| R_2 \|_\Norm{T}{\mL(H^{\tdm})} \le \mathcal{F}_1(T),
\ee
then 
\be\label{n510az-tdm}
\lA f\rA_{\bo([0,T];H^\tdm)}\le \frac{\lA u_{in}\rA_{H^\tdm}}{\mathcal{F}_2(T)}\,.
\ee

$ii)$ Assume that $(V,c,\Rdeux)$ satisfies
\e{n510z}, \e{R2 32} and
\be \label{fur con}
\| \pa_t V \|_\Norm{T}{H^2} + \| R_2 \|_\Norm{T}{\mL(H^1)} + \| \pa_t R_2 \|_\Norm{T}{\mL(L^2)} \leq 1,
\ee
and consider another triple $(V',c',\Rdeux')$ also satisfying the same (corresponding) bounds 
\e{n510z}, \e{R2 32} and \e{fur con}. Then  
\be\label{z534a sec 9} 
\begin{aligned}
&\| (\wide \Theta_T-\wide \Theta'_T)(u_{in})\|_\Norm{T}{L^2}  \\
&\qquad\le \frac{\lA u_{in}\rA_{H^{\tdm}}}{\mathcal{F}_3(T)} 
\Big\{ \| c - c' \|_\Norm{T}{H^{r+1}} + \| \pa_t c - \pa_t c' \|_\Norm{T}{H^1} \\
&\qquad\qquad \qquad \qquad 
+ \| V - V' \|_\Norm{T}{H^2} + \| R_2 - R'_2 \|_\Norm{T}{\mL(L^2)}\Big\}. 
\end{aligned}
\ee 
\end{lemm}

\begin{proof}
Recall that the cut-off function $\chi_\omega(x)$ is supported on $\omega$ and $\chi_\omega = 1$ on the open interval $\omega_1 \subset \omega$. 
Consider another open interval $\omega_2$ and a cut-off function $\chi_2(x)$ such that
\be \label{supp 1}
\begin{cases} (i) \ \text{supp}(\chi_2) \subseteq \omega_2 \, ; \\
(ii) \  \text{if supp}(h) \subseteq \omega_2 \ \Rightarrow \ \text{supp}(\Phi^{-1} h) \subseteq \omega_1 
\quad \forall t \in [0,T], \ \forall h \in L^2(\T).
\end{cases}
\ee
We want to apply Proposition \ref{P:51} for $Q = \wide P_3$. 
The hypothesis \e{n510} of Proposition \ref{P:51}, namely the inequality  $\| (W, R_3) \|_{r,T} \leq \mathcal{F}_1(T)$, follows from the assumption \e{n510z}, 
by using \e{p215} and \e{p214 high} with $\s = 3/2$. 
Hence, by the statement $i)$ of Proposition \ref{P:51}
(applied with $T_1$ instead of $T$ and $\chi_2$ instead of $\chi_\om$), given $w_{in} \in L^2(\T)$, the unique solution $w$ of the Cauchy problem 
\be \label{1912 1}
\widetilde P_3 w = \chi_2 \RE (f_2) \quad \forall t \in [0,T_1], \qquad 
w(0) = w_{in} 
\ee
satisfies $w(T_1) = i b M$ for some real constant $b$ 
if we choose $f_2 = \Theta_{M,T_1}(w_{in})$, 
where $\Theta_{M,T_1}$ is the operator given by Proposition~\ref{P:51},  
and the function $M$ will be fixed below in this proof. 
Also, by \e{n510a}, $f_2$ satisfies 
\be \label{jan 1}
\| f_2 \|_{C^0([0,T_1];L^2)} \le \frac{ \| w_{in} \|_{L^2}}{\mathcal{F}_2(T_1)}\,.
\ee
Moreover, if \e{R2 32} also holds, then, using \e{p215 high} with $\s = 3/2$, 
we deduce the bound \e{n510-mu} for $W,R_3$ with $\mu = 3/2$. 
Therefore, by the statement $iii)$ of Proposition \ref{P:51}, 
\be\label{jan 2} 
\| f_2 \|_{\bo([0,T_1];H^\tdm)} \le \frac{\| w_{in} \|_{H^\tdm}}{\mathcal{F}_3(T_1)}\,.
\ee 
Now let $u_{in} \in L^2(\xT)$ be given 
and define $w_{in}\in L^2(\xT)$ by 
$w_{in}:=\Phi\arrowvert_{t=0} u_{in}$. 
We apply the previous argument and obtain 
a function $w$ satisfying \e{1912 1} and  $w(T_1) = i b M$. Set 
$u := \Phi^{-1} w$. 
Since $\wide P u = m \Phi^{-1} \wide P_3 \Phi u$, it follows from \e{1912 1} that
\be \label{1912 2}
\wide P u = m \Phi^{-1} (\chi_2 \RE (f_2)) \quad \forall t \in [0,T], \qquad 
u(0) = u_{in},
\ee
and $u(T) = \Phi^{-1}_{|t = T} (i b M)$. 
Then we set 
$f := m \Phi^{-1}(\chi_2  f_2)$, namely we define 
\be\label{p511z}
f = {\wide \Theta}_T (u_{in}) \defn 
 m \Phi^{-1} (\chi_2  \Theta_{M,T_1}(\Phi u_{in})),
\ee
where $\Phi u_{in} = \Phi_{|t=0} u_{in}$. 
By the assumption $(ii)$ in \e{supp 1}, $f$ is supported in $\om_1$, and therefore $f = \chi_\omega f$. Then, since $m \Phi^{-1} (\chi_2 \RE (f_2)) 
= \RE( m \Phi^{-1} (\chi_2 f_2)) = \RE(f)$, 
$$
\wide P u = \chi_\omega \RE f,
\quad u(0)=u_{in},
$$
and we have to choose $M$ so that $u(T)=ib$. 
By definition of $\Phi$, recall that $w=\Phi u$ means that
\[ 
w(t,x) = \big\{ 1+\px\tilde\beta_1 \big( \psi^{-1}(t),x - p(t) \big) \big\}^\mez \, 
u \Big( \psi^{-1}(t), x - p(t)  + \tilde\beta_1 \big( \psi^{-1}(t), x - p(t) \big) \Big)
\]
for $t \in [0,T_1]$, $x \in \T$.
Since $\psi^{-1}(T_1)=T$, we see that $u(T)=ib$ provided that $w(T_1,x)=ibM(x)$ with
\be\label{p515z}
M(x)=\big\{ 1 + \px\tilde\beta_1 (T, x -  p(T_1)) \big\}^\mez,
\ee
and $p(T_1)$ is given in \e{pT}. 
Now the estimates \e{n510az} and \e{n510az-tdm} 
follow from \e{jan 1}, \e{jan 2} 
and Proposition~\ref{P:38}. This completes the proof of statement $i)$.

$ii)$ 
In what follows, we add the exponent $'$ to denote the objects associated to $(V',c',\Rdeux')$. 
Let $f = \wide \Theta_T(u_{in})$ be defined by \e{p511z}, 
and let $f' = \wide \Theta'_T(u_{in})$ be the corresponding function obtained by taking 
$(V',c', R_2')$ instead of $(V,c, R_2)$. We have to estimate the difference $f - f'$. 
If the constant $\mathcal{F}_1(T)$ in \e{n510z} is sufficiently small, 
then $\om_2, \chi_2$ can be chosen so that \eqref{supp 1} holds both for $\Phi$ and for $\Phi'$. Hence
\[
f - f' = m \Phi^{-1} (\chi_2 \Theta_{M,T_1} (\Phi u_{in}))
- m' \Phi'^{-1} (\chi_2 \Theta'_{M',T_1'} (\Phi' u_{in})).
\]
We split this difference into the sum $f - f' = A_1 + \ldots + A_6$, where 
\begin{align*}
A_1 & := (m - m') \Phi^{-1} (\chi_2 \Theta_{M,T_1} (\Phi u_{in})) 
\\
A_2 & := m' \Phi^{-1} [\chi_2 \Theta_{M,T_1} (\Phi u_{in} - \Phi' u_{in})]
\\
A_3 & := m' \Phi^{-1} [\chi_2 (\Theta_{M,T_1} - \Theta_{M',T_1})(\Phi' u_{in})]
\\
A_4 & := m' (\Psi_1^{-1} - \Psi_1'^{-1}) \psi_* \ph_* [\chi_2 \Theta_{M',T_1}(\Phi' u_{in})]
\\
A_5 & := m' \Psi_1'^{-1} (\psi_* \ph_* - \psi'_* \ph'_* \mathcal{T}) 
[\chi_2 \Theta_{M',T_1}(\Phi' u_{in})]
\\
A_6 & := m' \Phi'^{-1} [\chi_2 \{ \mathcal{T} \Theta_{M',T_1}(\Phi' u_{in}) - \Theta'_{M',T_1'}(\Phi' u_{in}) \}]
\end{align*}
and $\mathcal{T}$ is the time-rescaling operator defined above, namely $(\mathcal{T} h)(t,x) := h(\lm t, x)$, with $\lm := T_1 / T_1'$. Let us estimate each $A_i$.  

\emph{Estimate for $A_1$}. 
Apply \e{wwn 3}. 
\emph{Estimate for $A_2$}.
By construction (see Appendix \ref{S:AC1}), $\psi^{-1}(0) = 0$, $p(0) = 0$, and therefore $\Phi u_{in} = \Phi_{|t = 0} u_{in} = (\Psi_1^{-1})_{|t=0} (u_{in})$. Hence the estimate for $A_2$ follows by \e{wwn 1} 
and \e{n510a}. 
\emph{Estimate for $A_3$}. Apply \e{wwn 8}. 
\emph{Estimate for $A_4$}. Apply \e{wwn 1} and \e{p51a} with $\mu = 1$.
\emph{Estimate for $A_5$}. Apply \e{wwn 2}. To estimate $\pa_t f_2$, use that $f_2$ solves $\wide P_3^* f_2 = 0$ (statement $ii)$ of Proposition \ref{P:51}), and similarly for $f_2'$. 
\emph{Estimate for $A_6$.} The assumptions \e{n510z} and \e{R2 32} imply that $W,R_3$ and $W',R_3'$ satisfy 
\e{n510-mu} with $\mu = 3/2$, which is the hypothesis of statement $iv)$ of Proposition \ref{P:51}. 
Then \e{z534a} holds, namely 
\begin{multline*}
\| \mathcal{T} \Theta_{M',T_1}(\Phi' u_{in}) - \Theta'_{M',T_1'}(\Phi' u_{in}) 
\|_{C^0([0,T_1'];L^2)} 
\\ 
\les \| \Phi' u_{in} \|_{H^{3/2}} \big( |1-\lm| + \| W' - (\mT W) \|_{C^0([0,T'_1];H^2)} 
+ \| R_3' - (\mT R_3) \|_{C^0([0,T'_1]; \mL(L^2))} \big).
\end{multline*} 
Now $\| \Phi' u_{in} \|_{H^{3/2}} \les  \| u_{in} \|_{H^{3/2}}$, 
and the bounds for the last three differences are given in \e{wwn 3}, \e{wwn 5} (with $\s = 2$) and \e{wwn 7}.  
Note that assumptions \e{n510z}, \e{R2 32} and \e{fur con} imply \e{wwn 4}, \e{wwn 6}, 
which imply \e{wwn 5} and \e{wwn 7}. 
\end{proof}

\begin{rema} \label{rem:34pat} 
The function $W$ contains the terms $\pa_t c$ and $V$: see Appendix \ref{S:AC1} 
(see also the bound \e{p214}).
For this reason we assume $\pa_t^4 c$ and $\pa_t^3 V$ to be bounded in \e{n510z} 
in order to get a bound for $\pa_t^3 W$, as required by Proposition \ref{P:51}. 
\end{rema}

\begin{lemm}\label{PPLL'}
If the $\eC{\tdm}$-norms of $c-1$ and $c'-1$ are small enough, then
\[ 
\| \Lambda_{h,s} - \Lambda_{h,s}' \|_{\mathcal{L}(H^s,L^2)} 
+ \| (\Lambda_{h,s})^{-1}-(\Lambda_{h,s}')^{-1} \|_{\mathcal{L}(L^2,H^s)}
\les \| c-c'\|_{H^1}.
\]
\end{lemm}

\begin{proof}
By definition~\e{pf1} of $\Lambda_{h,s}$ one has 
$$
\Lambda_{h,s}-\Lambda_{h,s}'
=h^s T_{c^{(2s)/3}-c'^{(2s)/3}}L^{(2s)/3}.
$$
So the bound for $\Lambda_{h,s} - \Lambda_{h,s}'$ follows from the paradifferential rule \e{esti:quant0} and the Sobolev embedding $H^1(\xT)\subset L^\infty(\xT)$. 
To prove the other bound, 
we use the identity \e{pna19} to obtain that
$$
\Lambda_{h,s}^{-1}-(\Lambda_{h,s}')^{-1}=
(I+h^{s}L^{\frac{2s}{3}})^{-1} [(I+B)^{-1}-(I+B')^{-1}].
$$
Recall that 
$\lA B\rA_{\mathcal{L}(L^2)}\le \frac{1}{2}$ and 
$\lA B'\rA_{\mathcal{L}(L^2)}\le \frac{1}{2}$ so the identity
$$
(I+B)^{-1}-(I+B')^{-1}
=(I+B)^{-1}(B'-B)(I+B')^{-1}
$$
implies that
\be\label{g10}
\| (I+B)^{-1}-(I+B')^{-1} \|_{\mathcal{L}(L^2)} 
\le 4 \| B-B' \|_{\mathcal{L}(L^2)},
\ee
and the bound follows from the definition of $B,B'$ and the paradifferential rule \e{esti:quant0} as above.
\end{proof}

\emph{End of the proof of Proposition \ref{PP10-b}.}
We recall that $\widecontrol{T}$ is the control operator as given by Lemma~\ref{P:51-w}, 
and the operator $\opk$ is introduced in~\e{n211}, with $\| (I+\opk)^{-1} \|_{\mathcal{L}(L^2)}\le 2$ 
if $(c,V,R_2)$ satisfy \e{n510z}, \e{R2 32} and $\| c-1 \|_\Norm{T}{H^2}$ is small enough. 
Set
\be\label{p519z}
\Theta_{s,T}[(V,c,R)] \defn \Lambda_{h,s}^{-1}\widecontrol{T} (I+\opk)^{-1}\Lambda_{h,s},
\ee
and let $f := \Theta_{s,T}[(V,c,R)](v_{in})$. 
Then it follows from the previous construction 
(see Section~\ref{P:S3}, 
in particular the \emph{Proof of Proposition~\ref{PP10} given Proposition~\ref{PP11}}) 
that the unique solution $v$ to $Pv=T_p\chi_\omega f,~v_{\arrowvert t=0}=v_{in}$ 
satisfies $v(T)=ib$ for some constant $b\in\xR$. 
Since $\IM \int_\xT v_{in}(x)\, dx=0$ by assumption, by \e{p517z} we deduce that
$$
\IM \int_\xT v(T,x)\, dx=0.
$$
Therefore $b=0$ and $v(T)=0$. Thus it remains to prove \e{n532}. 
Following the same argument used in Section \ref{P:S3} to prove \e{n217}, one proves that  
$\| \opk \|_{\mL(H^{3/2})} \leq 1/2$, whence $\| (I + \opk)^{-1} \|_{\mL(H^{3/2})} \leq 2$. 
By combining this estimate with \e{n510az-tdm}, we have 
\begin{align*}
\| f \|_\Norm{T}{H^{s+\tdm}} 
& \les \| \widecontrol{T} (I+\opk)^{-1}\Lambda_{h,s} v_{in} \|_\Norm{T}{H^\tdm} 
\les \| \Lambda_{h,s} v_{in} \|_{H^\tdm} 
\les \| v_{in} \|_{H^{s+\tdm}},
\end{align*}
which is \e{n532}. Finally, we observe that
\begin{multline} \label{1201 1}
\| R_2 \|_\Norm{T}{\mL(L^2)} + \| R_2 \|_\Norm{T}{\mL(H^{\frac32})} \\
\les \| c-1, \pa_t c, V \|_\Norm{T}{H^{s_0}} 
+ \| R \|_\Norm{T}{\mL(H^s)} + \| R \|_\Norm{T}{\mL(H^{s+\frac32})}.
\end{multline}
This bound for $R_2$ follows easily from the arguments used in the proof of Lemma~\ref{PP10.5deux}Ê 
and Lemma~\ref{PPLL'}. 
Hence, if $(c,V,R)$ satisfy \e{pn200-b}, then $(c,V,R_2)$ satisfy \e{n510z}, \e{R2 32}.  
This completes the proof of statement $i)$. 

$ii)$ 
Given $y \in H^{3/2}(\T)$, we have to estimate the difference $\Theta_{s,T}[(V,c,R)](v_{in}) -\Theta_{s,T}[(V',c',R')](v_{in})$, which is, by definition, 
\[ 
\Lambda_{h,s}^{-1}\wide \Theta_{T} (I+\opk)^{-1}\Lambda_{h,s} v_{in}
-(\Lambda_{h,s}')^{-1}\wide \Theta'_{T} (I+\opk')^{-1}\Lambda_{h,s}' v_{in}.
\]

We write it as the sum $B_1 + \ldots + B_4$, with 
\begin{align*}
B_1 & := \{ \Lambda_{h,s}^{-1} - (\Lambda_{h,s}')^{-1} \} \wide \Theta_{T} (I+\opk)^{-1}\Lambda_{h,s} v_{in}, \\
B_2 & := (\Lambda_{h,s}')^{-1} (\wide \Theta_{T} - \wide \Theta'_{T}) (I+\opk)^{-1} \Lambda_{h,s} v_{in}, \\
B_3 & := (\Lambda_{h,s}')^{-1} \wide \Theta'_{T} \{ (I+\opk)^{-1} - (I+\opk')^{-1} \} \Lambda_{h,s} v_{in}, \\
B_4 & := (\Lambda_{h,s}')^{-1} \wide \Theta'_{T} (I+\opk')^{-1} (\Lambda_{h,s} - \Lambda_{h,s}') v_{in}.
\end{align*}
If $(c,V,R)$ satisfy \e{pn200-b}, then $(c,V,R_2)$ satisfy \e{n510z} and \e{R2 32}, and $\| \opk \|_{\mL(L^2)} \leq \frac12$, see \e{n217}. Then, using Lemma \ref{PPLL'} and \e{n510az}, we bound the 
$C^0([0,T];H^s)$-norm of $B_1$ and $B_4$ by $\| c - c' \|_{H^1} \| v_{in} \|_{H^s}$. 
To estimate $B_2$, we want to use \e{z534a sec 9}, which holds provided that $(c,V,R_2)$ and $(c',V',R_2')$ satisfy \e{fur con}. One proves that, if $\| c-1 \|_\Norm{T}{H^3}$ is small enough, then 
\begin{align} \label{pna15zb}
\| \Rdeux - \Rdeux' \|_\Norm{T}{\mL(L^2)} 
& \les \| (c-c', \partial_t (c-c'), V-V') \|_\Norm{T}{H^{s_0}} \notag\\
&\quad+ \| R-R' \|_\Norm{T}{\Lr(H^s)}
\\
\label{1201 2}
\| \pa_t R_2 \|_\Norm{T}{\mL(L^2)} 
& \les \| (c-1, \pa_t c, \pa_t^2 c, V, \pa_t V) \|_\Norm{T}{H^{s_0}} \notag\\
&\quad + \| (R, \pa_t R) \|_\Norm{T}{\mL(H^s)}.
\end{align}
These bounds follow from the arguments used in the proof of Lemma~\ref{PP10.5deux}Ê and Lemma~\ref{PPLL'}. 
Hence assumptions \e{norm Xs0s} and \e{fur con R} imply \e{fur con}, which implies \e{z534a sec 9}.
We have $\| (I+\opk)^{-1} \Lambda_{h,s} v_{in} \|_{H^{3/2}} 
\leq 2 \| \Lambda_{h,s} v_{in} \|_{H^{3/2}} \les \| v_{in} \|_{H^{s + 3/2}}$. 
Using \e{pna15zb} to estimate the last term in \e{z534a sec 9}, we deduce that 
\begin{align*}
\| B_2 \|_\Norm{T}{H^s} \les \| v_{in} \|_{H^{s + \frac32}} 
\{ & \| (c-c', \partial_t (c-c'), V-V') \|_\Norm{T}{H^{s_0}} \\
&\quad+ \| R-R' \|_\Norm{T}{\Lr(H^s)} \} .
\end{align*}
It remains to estimate $B_3$. The difference $\opk - \opk'$ satisfies 
\[ 
\| (\opk - \opk') y \|_{L^2} \les \| y \|_{H^{\frac32}} 
\{ \| (c - c', V - V') \|_\Norm{T}{H^1} + \| R_2 - R_2' \|_\Norm{T}{\mL(L^2)} \}.
\]
To prove this bound, recall that $\mathcal{K}$ is defined by solving an evolution equation, 
and then, as above, use the energy estimates proved in the appendix to bound the difference 
of two solutions satisfying evolutions equations. 
Since $\| (I+\opk)^{-1} \|_{\mL(L^2)} \leq 2$, 
$\| (I + \opk')^{-1} \|_{\mL(H^{3/2})} \leq 2$, and 
\[
(I+\opk)^{-1} - (I+\opk')^{-1} = (I+\opk)^{-1} (\opk'-\opk) (I + \opk')^{-1},
\]
we deduce that $B_3$ satisfies the same bound as $B_2$. 
The proof of Proposition \ref{PP10-b} is complete.
\end{proof}

\section{Iterative scheme}\label{P:S8}

\newcommand{\bcr}{\begin{color}{red}}
\newcommand{\bcb}{\begin{color}{blue}}
\newcommand{\ec}{\end{color}}

In this section we 
conclude the proof of Theorem~\ref{T1}. 
As explained in Remark~\ref{R:1}, it is sufficient to prove this result with $(\eta_{final},\psi_{final})=(0,0)$. Also, as explained in the introduction, 
we seek $P_{ext}$ as the real part of the limit of solutions to approximate control problems with variable coefficients. 

Consider the unknown $u=T_p\omega-iT_q \eta$ as introduced by Proposition~\ref{T24}. As proved in \S\ref{P:S23} (see also Section \ref{SN:3}), this 
new unknown $u$ solves an equation of the form
\be\label{n60a}
\partial_t u +T_{V(u)}\px u +iL^\mez\big(T_{c(u)}L^\mez u\big) u +R(u)u = T_{p(u)} P_{ext},
\ee
where, with a little abuse of notation, 
we write $V(u)$, $c(u)$, ... as shorthand notations for 
$V(\eta)\psi$ (see \e{p24z}), $c=(1+(\px\eta)^2)^{-\tq}$, ... 
where $(\eta,\psi)$ is expressed in terms of $u$ by means of Lemma~\ref{T27}.

Fix $T>0$. We claim that there is $\eps>0$ such that, for all initial data whose $H^s(\xT)$-norm 
(with $s$ large enough) is smaller than $\eps$, and all 
source term $P_{ext}$ whose $L^1([0,T];H^s(\xT))$-norm is smaller than $\eps$, 
the Cauchy problem for \e{n60a}Ê
has a unique solution in $C^0([0,T];H^s(\xT))$. 
The existence of a solution follows from the analysis given below. The uniqueness is obtained 
by estimating the difference of two solutions (as in \cite{ABZ1}) and we omit its proof.

Recall that $\tilde H^\mu(\xT;\xC)$ denotes the space of $H^\mu$-functions whose imaginary part have zero mean (see Notation \ref{N:26}). 
\begin{prop}\label{Tfinal}
Let $T>0$. For all $u_{in}\in \tilde H^{\sigma}(\xT;\xC)$ 
for some $\sigma$ large enough 
such that $\lA u_{in}\rA_{H^{\sigma}}$ is small enough, there exists
a real-valued function
$$
P_{ext}\in C^0([0,T];H^{\sigma}(\xT)),
$$ 
such that the unique solution 
$u\in C^0([0,T];H^\sigma(\xT))$ to \e{n60a} with initial data $u_{in}$ satisfies $u(T)=0$.
\end{prop}

Before proving this proposition, let us explain how to deduce Theorem~\ref{T1} from it. 
Recall that it is sufficient to consider the case where $(\eta_{final},\psi_{final})=(0,0)$. 
Once $P_{ext}$ is defined by means of Proposition \ref{Tfinal} applied with 
$u_{in} = T_{p_{in}}\omega_{in}-iT_{q_{in}} \eta_{in}$, we solve the water waves system~\e{system} for $(\eta,\psi)$ with data $(\eta_{in},\psi_{in})$ with this pressure seen as a source term. Then $u=T_p\omega-iT_q \eta$ solves \e{n60a}, so $u(T)=0$ which in turn implies that $(\eta,\psi)(T)=0$ in view of Lemma~\ref{T27}. 

\begin{proof}[Proof of Proposition~\ref{Tfinal}]
%
Set $s=\sigma-3/2$. 
Given $u_{in}\in \tilde H^{s+\tdm}(\xT;\xC)$ 
and $T>0$, introduce the following 
scheme: 
define $(u_0, f_0) := (0,0)$, and then, for $n \geq 0$, $(u_{n+1},f_{n+1})$ 
are defined by induction in this way: 
$f_{n+1}$ is determined by asking that the unique solution $u_{n+1}$ to the Cauchy problem 
\begin{align}
&\partial_t u_{n+1} +T_{V(u_{n})}\partial_x u_{n+1}
+iL^\mez T_{c(u_n)}L^\mez u_{n+1} 
+R(u_n)u_{n+1}
=T_{p(u_n)}\chi_\omega \RE f_{n+1}\notag\\[1ex]
&u_{n+1}\arrowvert_{t=0}=u_{in},\label{n63}
\end{align}
satisfies $u_{n+1}(T)=0$. 

Our goal is to prove that this scheme converges. Then we define $P_{ext}$ as the limit 
of $(\RE f_n)$ when $n$ goes to $+\infty$. 
Using the operator $\Theta_{s,T}$ defined by Proposition~\ref{PP10-b}, the scheme corresponds to 
define $(u_n)$ and $(f_n)$ as follows: 
\be\label{n64}
f_{n+1}\defn \Theta_{s,T}[X_n]\big(u_{in}\big)\quad\text{where}\quad 
X_n\defn (V(u_n),c(u_n),R(u_n))
\ee
and $u_{n+1}$ is defined as the unique solution to the Cauchy problem \e{n63}; by definition of 
$f_{n+1}$ we then have $u_{n+1}(T)=0$. 
Our goal is to prove that, for any $T > 0$, 
if $u_{in}$ is small enough, 
then this scheme is well-defined and 
$(u_n, f_n)$ converges to a solution $(u,f)$ of the desired nonlinear control problem.
This will be a consequence of the following result.

\begin{lemm}\label{P:71}
Consider $T > 0$. There exists $s_0$ large enough and for any 
$s\ge s_0+6$ there exist $\eps_0>0$ and positive constants $K_1,\ldots,K_7$ such that, 
for any $\eps\in (0,\eps_0]$, if 
$$
\lA u_{in}\rA_{H^{s+\tdm}(\xT)}\le \eps
$$ 
then, for any $n \geq 0$, there holds
\begin{align}
&\lA u_{n}\rA_{\bo([0,T];H^{s+\tdm})}\le K_1\eps, \label{n65} \\
&\blA \partial_t^k u_{n}\brA_{\bo([0,T];H^{s_0})}\le K_2\eps \qquad\text{for }1\le k\le 4,\label{n66}
\end{align}
Moreover, for any $n \geq 0$, 
\begin{align}
&\lA u_{n+1}-u_n\rA_{C^0([0,T];H^{s})}\le K_3 \eps 2^{-n},\label{n67}\\
&\lA \partial_t(u_{n+1}-u_n)\rA_{C^0([0,T];H^{s-\tdm})}\le K_4 \eps 2^{-n} ; \label{n67dt}
\end{align}
and for any $n \geq 1$,
\begin{align}
&\lA f_{n}\rA_{\bo([0,T];H^{s+\tdm})}\le K_5\eps,\label{n68a}\\[1ex]
&
\lA \partial_t^k f_{n}\rA_{\bo([0,T];H^{s_0})}\le K_6\eps\label{n68}\qquad 
\text{for }1\le k\le 3,\\
&\lA f_{n+1}-f_n\rA_{C^0([0,T];H^{s})}\le K_7\eps^2 2^{-n}.  \label{n69}
\end{align}
\end{lemm}
\begin{proof} For this proof we denote by $C$ various constants 
depending only on $T$, $s,s_0$ or $\omega$. Also we denote by $\mathcal{F}$ 
various increasing functions $\mathcal{F}\colon \xR_+\rightarrow \xR_+$ depending on parameters that are considered fixed. 

\begin{center}
{\em Step 1}\,: proof of \eqref{n65}, \eqref{n66}, \eqref{n68a} and \eqref{n68}. 
\end{center}

We prove these estimates by induction. They hold for $n=0$ 
since $(u_0, f_0) = (0,0)$. We now assume that they hold at rank $n$ and prove that they hold at rank $n+1$.

We begin by checking that the fact that 
the properties \e{n65}--\e{n66} hold at rank $n$ implies that 
one can apply Proposition~\ref{PP10-b} to prove that the scheme is well-defined. 
This means that we have to prove that the smallness assumption \e{pn200-b} is satisfied. 
To do so, we first recall that (see \e{p25})
$\lA V(u_n)\rA_{H^{s_0}}\le 
\mathcal{F}\big(\lA \eta_n\rA_{H^{s_0+1}}\big)\lA \psi_n\rA_{H^{s_0+1}}$. 
Then the estimate \e{p44a}Ê (applied with $s$ replaced by $s_0+1$)
implies that 
$\lA V(u_n)\rA_{H^{s_0}}\les \lA u_n\rA_{H^{s_0+1}}$. 
Similarly, the estimates \e{p35b} and \e{p44a} yield
$$
\lA R(u_n)\rA_{\mathcal{L}(H^{s+\tdm})}\le \mathcal{F}
\big(\lA \eta_n\rA_{H^{s+\tdm}}
\big)\lA \eta_n\rA_{H^{s+\tdm}}\le \mathcal{F}\left(\lA u_n\rA_{H^{s+\tdm}}\right)\lA u_n\rA_{H^{s+\tdm}},
$$
and, directly from the definition $c=(1+(\px\eta)^2)^{-3/4}$, one has
$$
\lA c(u_n)-1\rA_{H^{s_0}}\le \mathcal{F}\big(\lA \eta_n\rA_{H^{s_0+1}}\big)\lA \eta_n\rA_{H^{s_0+1}}
\les 
\lA u_n\rA_{H^{s_0+\mez}}.
$$
Gathering these estimates and recalling that $s_0+1\le s$, we conclude that
\be \label{gathering VcR}
\lA V(u_n)\rA_{H^{s_0}}+\lA c(u_n)-1\rA_{H^{s_0}}+\lA R(u_n)\rA_{\mathcal{L}(H^{s+\tdm})}\les \lA u_n\rA_{H^{s+\tdm}}.
\ee
Consequently, the property \e{n65} at rank $n$ implies that the part of the smallness condition 
\e{pn200-b} concerning $V,c,R$ is satisfied. Concerning the estimates of the time derivatives $\partial_t^k V$ and $\partial_t^k c$, we use 
the equations \eqref{system} and the rule (see \cite{LannesJAMS})
$$
\partial_t G(\eta)\psi=G(\eta)\big\{ 
\partial_t \psi-(B(\eta)\psi)\partial_t \eta\big\}
-\px \big((V(\eta)\psi)\eta\big)
$$
(where $B(\eta)\psi$ and $V(\eta)\psi$ are given by \e{p24z}) 
to express time derivatives $\partial_t^k V$ and $\partial_t^k c$ in terms of spatial derivatives and in terms of the operators $B(\eta),V(\eta)$ 
(see Appendix A.3 in \cite{AlDe} or \cite{LannesLivre,Mesognon}). 
Then, as above, the desired estimates then follow from \e{p25} and the usual nonlinear estimates in Sobolev spaces.

We now prove \eqref{n65} and \e{n68a} at rank $n+1$. 
By \e{pn26} we obtain that
\be \label{1401 1}
\lA u_{n+1}\rA_{C^0([0,T];H^{s+\tdm})}
\le C \big( \lA u_{in}\rA_{H^{s+\tdm}} + 
\lA T_{p_n} \chi_\omega \RE  f_{n+1}\rA_{C^0([0,T];H^{s+\tdm})} \big),
\ee
where the constant $C$ depends on $s,T$ 
(by \e{gathering VcR} and \e{n65} at rank $n$, 
the constant $M$ in Proposition \ref{P:10} is bounded by 1 if $K_1$ is large enough and $\eps_0$ is small enough).
Now observe that, since $T_{p_n}$ acts on any Sobolev space with operator norm 
bounded by $M^0_0(p_n) \le \mathcal{F}(\lA u_n\rA_{H^{s}}) 
\le \mathcal{F}(1)$, one has
$$
\lA T_{p_n} \chi_\omega \RE  f_{n+1}\rA_{C^0([0,T];H^{s+\tdm})}
\le C \lA f_{n+1}\rA_{C^0([0,T];H^{s+\tdm})}.
$$
Moreover, by \e{n532}, 
$\lA f_{n+1}\rA_{C^0([0,T];H^{s+\tdm})} \leq K_0 \lA u_{in}\rA_{H^{s+\tdm}}$ 
for some $K_0$ depending only on $T$. 
We conclude that, choosing $K_1$ large enough and $\eps_0$ small enough, 
\e{n65} holds at rank $n+1$.
Also \e{n68a} at rank $n+1$ follows by the same argument. 

%
%
%
%
%

It remains to prove \eqref{n66} and \e{n68}. Directly from the equation \e{n63}, 
expressing $\partial_t u_{n+1}$ in terms of $u_n,u_{n+1}$ and $f_{n+1}$  and using 
the operator norm estimate \e{esti:quant0} for paradifferential operators, 
one deduces \e{n66} for $k=1$ from the bounds \e{n65} and \e{n68a}. 
We next prove \e{n68} for $k=1$. 
To do so, the key point is to make explicit the equation satisfied by $f_{n+1}$. 
We recall from \e{n64}, \e{p519z} and \e{p511z} that 
\[
f_{n+1} \defn ( \Lambda_{h,s}^n )^{-1} 
\big( m^n ( \Phi^n )^{-1} ( \chi_{2}\wide f_{n+1} ) \big), \quad 
\wide f_{n+1} 
\defn \Theta_{M_n,T_1^n} \big(\Phi^n (I+\mathcal{K}_n)^{-1}\Lambda_{h,s}^n u_{in}\big),
\]
where $\Lambda_{h,s}^n,\Phi^n,m^n,M_n,\mathcal{K}_n,T_1^n$ are given 
by replacing $(V,c)$ with $(V(u_n),c(u_n))$ in the definition of $\Lambda_{h,s},\Phi,m,M,\mathcal{K},T_1$. By definition of $\Theta_{M,T}$ 
(Lemma~\ref{L3}) one has
\begin{equation}\label{n618}
\left\{
\begin{aligned}
&\partial_t \widetilde{f}_{n+1}+ W(u_n) \px \widetilde{f}_{n+1}+ i L\widetilde{f}_{n+1} + 
\Rquatre(u_n)\widetilde{f}_{n+1}=0,  \\ 
&\widetilde{f}_{n+1}
\arrowvert_{t=T_1^n}=\widetilde{f}_{n+1}^1,
\end{aligned}
\right.
\end{equation}
where $\Rquatre(u_n)$ is given by \e{n534-z} and  
the initial data $\widetilde{f}_{n+1}^1$ is given by Lemma~\ref{L3}. It follows from \e{p51a}Ê
that 
$$
\blA \widetilde{f}_{n+1}\brA_{C^0([0,T];H^\tdm)}\le K 
\blA \Phi^n (I+\mathcal{K}_n)^{-1}\Lambda_{h,s}^n u_{in}\brA_{H^\tdm}
\le K \lA u_{in}\rA_{H^{s+\tdm}}.
$$
Using the equation \e{n618} we thus estimate the 
$C^0([0,T];L^2)$-norm of $\partial_t\widetilde{f}_{n+1}$ from 
which we estimate $\partial_t f_{n+1}$ in $C^0([0,T];H^{s})$. This gives 
\e{n68}Ê
for $k=1$ since $s\ge s_0$. Now we obtain \e{n66} for $k=2,3,4$ as well as \e{n68} for $k=2,3$ by differentiating in time 
the equations satisfied by $u_{n+1}$ and $f_{n+1}$.

\begin{center}
{\em Step 2}\,: proof of \eqref{n67}, \eqref{n67dt}, \eqref{n69}.
\end{center}


The estimate \eqref{n69} will be deduced from 
\e{n67} and \e{n67dt}. To prove \e{n67} and \e{n67dt} 
we proceed by induction.  We assume that 
they hold at rank $n-1$ and prove that they hold at rank $n$. 

The key point is to estimate $\delta_{n} := u_{n+1}-u_n$. Write
\be \label{eq:deltan}
\partial_t \delta_{n} +T_{V(u_{n})}\partial_x \delta_{n}
+iL^\mez T_{c(u_n)}L^\mez \delta_{n}
+R(u_n)\delta_{n}
=G_n
\ee
with 
\begin{align} \label{n628}
G_n & := (T_{V(u_{n-1})}-T_{V(u_n)})\px u_n 
+ i L^\mez (T_{c(u_{n-1})} - T_{c(u_n)} ) L^\mez u_n 
\\ & \quad 
+ \big(R(u_{n-1})-R(u_n)\big)u_n
+ T_{p(u_n)}\chi_\omega(f_{n+1}-f_n)
+ (T_{p(u_n)}-T_{p(u_{n-1})})\chi_\omega f_n. \notag 
\end{align}
As in the previous step, it follows from Proposition~\ref{P:10} (noticing that $\delta_{n+1}(0)=0$) 
that $\| \delta_n \|_{C^0([0,T];H^s)} \leq C_0 \| G_n \|_{C^0([0,T];H^s)}$ for some $C_0$ depending on $s,T$.

\smallbreak

\noindent{\em Estimate for} $G_n$.  We claim that 
\be\label{n629}
\lA G_n\rA_{C^0([0,T];H^s)} \le \eps K(T) 
\lA \delta_{n-1}\rA_{C^0([0,T];H^s)}
+\eps K(T) \lA \partial_t\delta_{n-1}\rA_{C^0([0,T];H^{s-\tdm})}.
\ee
Let us prove this claim. At each $t \in [0,T]$, 
using \e{esti:quant0} one has
$$
\lA (T_{V(u_{n-1})}-T_{V(u_n)})\px u_n\rA_{H^s}\les \lA V(u_{n-1})-V(u_n)\rA_{L^\infty}\lA \px u_n\rA_{H^{s}}.
$$
It follows from \e{n65}Ê
that $\lA \px u_n\rA_{H^{s}}\le K_1\eps $. 
To estimate $V(u_{n-1})-V(u_n)$ we use the following consequence of Lemma~5.3 in \cite{ABZ-ul}: 
Assume $s>3/2$ and consider $(\eta_1,\eta_2)$ such that 
$\lA \eta_1\rA_{H^s}+\lA \eta_2\rA_{H^s}\le 1$. Then 
$$
\lA G(\eta_1)f_1-G(\eta_2)f_2\rA_{H^{s-\tdm}}\le K 
\lA \eta_1-\eta_2\rA_{H^{s-\mez}}\lA f_1\rA_{H^s}+K\lA f_1-f_2\rA_{H^{s-\mez}}.
$$
Then, directly from the definition of $V(\eta)\psi$ one deduces that
$$
\lA V(\eta_1)\psi_1-V(\eta_2)\psi_2\rA_{H^1}\le K 
\lA \eta_1-\eta_2\rA_{H^{2}}\lA \psi_1\rA_{H^{5/2}}+K\lA \psi_1-\psi_2\rA_{H^{2}}.
$$
Since $H^1(\xT)\subset L^\infty(\xT)$, we then conclude that 
$$
\lA V(u_{n-1})-V(u_n)\rA_{L^\infty}\les \lA \eta_n-\eta_{n-1}\rA_{H^s}+\lA \psi_n-\psi_{n-1}\rA_{H^s}\les
\lA u_n-u_{n-1}\rA_{H^{s}}.
$$
The estimate of the 
$H^s$ norm of $L^\mez (T_{c(u_{n-1})} - T_{c(u_n)} ) L^\mez u_n$ is similar. 
To estimate $\big(R(u_{n-1})-R(u_n)\big)u_n$ recall that $R(\underline{u})u$ is as given by 
Proposition~\ref{T24}. This operator is defined by means of the remainder $F(\eta)\psi$ in \e{p38} and 
also in terms of explicit expressions involving symbolic calculus or the paralinearization of products. The only delicate point is to estimate $F(\eta_n)\psi_n-F(\eta_{n-1})\psi_{n-1}$. To do so one uses Lemma~6.8 in \cite{ABZ1}.

It remains to estimate the last two terms in the right-hand side of \e{n628}. Directly from \e{esti:quant1} we find that
$$
\lA (T_{p(u_n)}-T_{p(u_{n-1})})\chi_\omega f_n\rA_{H^s}
\les M^0_0(p(u_n)-p(u_{n-1})) \lA \chi_\omega f_n\rA_{H^s}.
$$
Now $\lA \chi_\omega f_n\rA_{H^s}\les \lA \chi_\omega \rA_{H^s}\lA f_n\rA_{H^s}\les \eps$ 
\bcr by \e{n68a}, \ec and 
$M^0_0(p(u_n)-p(u_{n-1}))$ is bounded by $K\lA u_n-u_{n-1}\rA_{H^{s}}$. 
Eventually, to estimate the $H^s$-norm of $T_{p(u_n)}\chi_\omega(f_{n+1}-f_n)$ 
we use again \e{esti:quant1} to bound this expression in terms of $\lA f_{n+1}-f_n\rA_{H^s}$. 
We use \e{n534} to obtain 
\begin{align}
&\| f_{n+1}-f_n \|_{C^0([0,T];H^{s})} \notag\\
&\qquad\qquad\les \lA u_{in}\rA_{H^{s+\tdm}}
\big\{ \| (c_{n}-c_{n-1}, \pa_t(c_{n}-c_{n-1}), V_{n}-V_{n-1}) \|_{C^0([0,T];H^{s_0})} \notag\\
&\qquad\qquad\qquad\qquad\qquad\quad + \| R_{n}-R_{n-1}\|_{C^0([0,T];\Lr(H^s))} \big\}\notag\\
&\qquad\qquad\les \lA u_{in}\rA_{H^{s+\tdm}}\big\{ \lA u_{n}-u_{n-1}\rA_{H^{s_0}}
+\lA \partial_t(u_{n}-u_{n-1})\rA_{H^{s_0}}\big\},\label{grande labello faccio bella vita}
\end{align}
and then we use \e{n67} and \e{n67dt} at rank $n-1$.

\smallbreak

\noindent{\em Estimate for} $u_{n+1}-u_n$. 
For $\eps_0 K(T) C_0 \le 1/2$, it follows from \e{n67} at rank $n-1$ and \e{n629} that the desired result \e{n67} at rank $n$ holds. 

\smallbreak

\noindent{\em Estimate for} $f_{n+1}-f_n$. 
The estimate \e{n69} follows from \e{grande labello faccio bella vita} and the assumptions \e{n67}--\e{n67dt} at rank $n-1$. 

\smallbreak

\noindent{\em Estimate for} $\partial_t(u_{n+1}-u_n)$. 
By \e{eq:deltan}, 
\be \label{eq:pat deltan}
\partial_t \delta_{n}=-T_{V(u_{n})}\partial_x \delta_{n}
-iL^\mez T_{c(u_n)}L^\mez \delta_{n}
-R(u_n)\delta_{n}
+G_n.
\ee
As above, one has
$$
\| T_{V(u_n)} \|_{\Lr(H^{s},H^{s-1})}
+ \| L^\mez T_{c(u_n)}L^\mez\|_{\Lr(H^{s},H^{s-\tdm})}
+ \| R(u_n)\|_{\Lr(H^s,H^s)} \le C \| u_n \|_{H^{s}}.
$$
Therefore one can use \e{n67} and \e{n65} to estimate the first three terms 
in the right-hand side \bcr of \e{eq:pat deltan}. \ec 
The last term $G_n$ is estimated by means of \e{n629} and the induction assumptions. 
Consequently, we get $\| \partial_t \delta_n \|_{C^0([0,T];H^{s-\tdm})} \le C \eps^2 \, 2^{-n}$, and for $\eps \le \eps_0$, with $\eps_0$ small enough, we deduce \e{n67dt}.
\end{proof}

We can now conclude the proof of Proposition~\ref{Tfinal}. 

Recall that $s=\sigma-3/2$ by notation. 
By~\eqref{n67} and~\eqref{n69}, we deduce that $(u_n)_{n \in \mathbb{N}}$ and $(f_n)_{n \in \mathbb{N}}$ are Cauchy sequences in $C^0([0,T]; H^{s})$ 
and therefore converge to some limits $u$ and $f$ in $C^0([0,T];H^{s})$. 
Using the uniform bounds~\eqref{n65} and~\eqref{n68a} and the interpolation inequality in Sobolev spaces, 
we infer that $(u_n)_{n \in \mathbb{N}}$ and $(f_n)_{n \in \mathbb{N}}$ converge in $C^0([0,T]; H^{s'+\tdm})$ for all $s'<s$. Furthermore, we get that $u$ and $f$ belong to $C^0([0,T]; H^{s'+\tdm}) \cap L^\infty([0,T]; H^{s+\tdm})$ for all $s'<s$. 
Passing to the limit in~\eqref{n63}, we conclude that $u$ and $f$ satisfy~\eqref{n60a} 
and $u(T)=0$. 
Eventually, using Lemma~\ref{P:10} (seeing~\eqref{n60a} as a linear equation of the type~\eqref{pn25} 
with unknown $u$ and coefficients in $L^\infty([0,T]; H^s)$), we deduce $u \in C^0([0,T]; H^{s+\tdm})$. 

It remains to prove that $f \in C^0([0,T];H^{s + \tdm})$.  
We know that $u_n \to u$ in $C^0([0,T];H^s)$ $\subset C^0([0,T];H^{s_0 + 6})$. 
As a consequence, $V(u_n) \to V(u)$, $c(u_n) \to c(u)$, $\pa_t c(u_n) \to \pa_t c(u)$, 
$p(u_n) \to p(u)$ in $C^0([0,T];H^{s_0})$, and 
$R(u_n) \to R(u)$ in $C^0([0,T];\mL(H^s))$. 
Now consider $f_\infty := \Theta_{s,T}[V(u), c(u), R(u)](u_{in})$, 
and recall the definition \e{n64}. 
By \e{n534}, $\| f_n - f_\infty \|_{C^0([0,T];H^s)} \to 0$ as $n \to \infty$. 
On the other hand, $f = \lim f_n$ in $C^0([0,T];H^s)$, and therefore $f = f_\infty$. 
By statement $(i)$ of Proposition \ref{PP10-b}, $f_\infty \in C^0([0,T];H^{s+\tdm})$, 
with estimate \e{n532}.

This concludes the proof of Proposition~\ref{Tfinal} and hence the proof of Theorem~\ref{T1}.
\end{proof}

\appendix

\section{Paradifferential operators}

\begin{nota}
For $\rho\in\xN$, we denote 
by $W^{\rho,\infty}(\xT)$ the Sobolev spaces of $L^\infty$ functions 
whose derivatives of order $\rho$ are in $L^\infty$. 
For $\rho\in ]0,+\infty[\setminus \xN$, we denote 
by $W^{\rho,\infty}(\xT)$ the 
space functions in $W^{[\rho],\infty}(\xT)$ whose derivatives of order $[\rho]$  are uniformly H\"older continuous with 
exponent $\rho- [\rho]$.
\end{nota}
\begin{defi}\label{defi Gamma}
Given real numbers $\rho\ge 0$ and $m\in\xR$, 
$\Gamma_{\rho}^{m}$ denotes the space of functions $a(x,\xi)$
on $\xT\times \xR$ which are $C^\infty$ with respect to $\xi$, and
such that, for all $\alpha\in\xN$ and all $\xi$, the function
$x\mapsto \partial_\xi^\alpha a(x,\xi)$ belongs to $W^{\rho,\infty}(\xT)$ and
\begin{equation*}
\lA \partial_\xi^\alpha a(\cdot,\xi)\rA_{W^{\rho,\infty}}\le C_\alpha 
(1+\la\xi\ra)^{m-\la\alpha\ra}.
\end{equation*}
\end{defi}
\begin{defi}\label{defiGmrho}
For~$m\in\xR$,~$\rho\in [0,1]$ and~$a\in \Gamma^m_{\rho}({\mathbf{R}}^d)$, we set
\begin{equation}\label{defi:norms}
M_{\rho}^{m}(a)= 
\sup_{\la\alpha\ra\le 6 +\rho ~}\sup_{\xi\in\xR}
\lA (1+\la\xi\ra)^{\la\alpha\ra-m}\partial_\xi^\alpha a(\cdot,\xi)\rA_{W^{\rho,\infty}(\xT)}.
\end{equation}
\end{defi}

Now consider 
a $C^\infty$ function $\chi$ homogeneous of degree $0$ and satisfying, 
for $0<\eps_1<\eps_2$ small enough,
$$
\chi(\theta,\eta)=1 \quad \text{if}\quad \la\theta\ra\le \eps_1\la \eta\ra,\qquad
\chi(\theta,\eta)=0 \quad \text{if}\quad \la\theta\ra\geq \eps_2\la\eta\ra.
$$ 
Given a symbol $a$, we define
the paradifferential operator $T_a$ by
\begin{equation}\label{eq.para}
\widehat{T_a u}(\xi)=(2\pi)^{-1}\sum_{\eta\in\xZ}
\chi(\xi-\eta,\eta)\widehat{a}(\xi-\eta,\eta)\widehat{u}(\eta),
\end{equation}
where
$\widehat{a}(\theta,\xi)=\int e^{-ix\cdot\theta}a(x,\xi)\, dx$
is the Fourier transform of $a$ with respect to the first variable.

The main features of symbolic calculus for paradifferential operators 
are given by the following theorem.
\begin{defi}\label{defi:order}
Let~$m$ in $\xR$.
An operator~$T$ is said of order~$m$ if, for any~$\mu\in\xR$,
it is bounded from~$H^{\mu}(\xT)$ to~$H^{\mu-m}(\xT)$.
\end{defi}
\begin{theo}\label{theo:sc0}
Let~$m\in\xR$. 

$(i)$ If~$a \in \Gamma^m_0$, 
then~$T_a$ is of order~$m$. 
Moreover, for any~$\mu\in\xR$ there exists~$K>0$ such that
\begin{equation}\label{esti:quant1}
\lA T_a \rA_{\Fl{H^{\mu}}{H^{\mu-m}}}\le K M_{0}^{m}(a).
\end{equation}
$(ii)$ Let~$(m,m')\in\xR^2$ and~$\rho\in (0,+\infty)$. 
If~$a\in \Gamma^{m}_{\rho}, b\in \Gamma^{m'}_{\rho}$ then 
$T_a T_b -T_{a\sharp b}$ is of order~$  m+m'-\rho$ where
\be\label{defi:sharp}
a\sharp b=
\sum_{\la \alpha\ra < \rho} \frac{1}{i^{\la\alpha\ra} \alpha !} \partial_\xi^{\alpha} a \partial_{x}^\alpha b.
\ee
Furthermore, for any~$\mu\in\xR$ there exists $K>0$ such that
\begin{equation}\label{esti:quant2sharp}
\lA T_a T_b  - T_{a\sharp b}   \rA_{\Fl{H^{\mu}}{H^{\mu-m-m'+\rho}}}\le 
K M_{\rho}^{m}(a)M_{\rho}^{m'}(b).
\end{equation}

In particular, if~$\rho\in ( 0,1]$, 
$a\in \Gamma^{m}_{\rho}, b\in \Gamma^{m'}_{\rho}$ then
\begin{equation}\label{esti:quant2}
\lA T_a T_b  - T_{a b}   \rA_{\Fl{H^{\mu}}{H^{\mu-m-m'+\rho}}}\le 
K M_{\rho}^{m}(a)M_{\rho}^{m'}(b).
\end{equation}
$(iii)$ Let $m\in\xR$, $\rho>0$ and $a\in \Gamma^{m}_{\rho}(\xR^d)$. Denote by 
$(T_a)^*$ the adjoint operator of $T_a$ and by $\overline{a}$ the complex-conjugated of $a$. Then 
$(T_a)^* -T_{a^*}$ is of order $m-\rho$ where
$$
a^*=
\sum_{\la \alpha\ra < \rho} \frac{1}{i^{\la\alpha\ra} \alpha !} \partial_\xi^\alpha \partial_x^{\alpha} \overline{a} .
$$
Moreover, for all $\mu$ there exists a constant $K$ such that
\begin{equation}\label{esti:quant3}
\lA (T_a)^*   - T_{a^*}   \rA_{H^{\mu}\rightarrow H^{\mu-m+\rho}}\le 
K M_{\rho}^{m}(a).
\end{equation}
\end{theo}
\begin{rema}
These properties are well-known when 
Sobolev spaces of periodic functions are replaced by 
Sobolev spaces 
on the real line. To prove these results for periodic functions, one can 
use the results proved in \cite{ABZ-ul} about the general case of uniformly local Sobolev spaces $H^s_{ul}(\xR)$. Namely, 
in \cite{ABZ-ul}, the above results are proved to hold when $H^s(\xT)$ is replaced by $H^s_{ul}(\xR)$. In particular it is proved that
$$
\lA (T_a T_b  - T_{a b})u\rA_{H^{\mu-m-m'+\rho}_{ul}}\le 
K M_{\rho}^{m}(a)M_{\rho}^{m'}(b)\lA u\rA_{H^\mu_{ul}}.
$$
Since $\lA u\rA_{H^s_{ul}}\les \lA u\rA_{H^s(\xT)}$, it follows that
$$
\lA (T_a T_b  - T_{a b})u\rA_{H^{\mu-m-m'+\rho}_{ul}(\xR)}\le 
K M_{\rho}^{m}(a)M_{\rho}^{m'}(b)\lA u\rA_{H^\mu(\xT)}.
$$
Now, if $u$ is a periodic function and $a$ and $b$ are periodic in $x$, 
so is  $(T_a T_b  - T_{a b})u$ and we deduce that
$$
\lA (T_a T_b  - T_{a b})u\rA_{H^{\mu-m-m'+\rho}(\xT)}\les \lA (T_a T_b  - T_{a b})u\rA_{H^{\mu-m-m'+\rho}_{ul}(\xR)}.
$$
By combining the previous estimates we obtain \e{esti:quant2}. 
The other estimates are proved in a similar way. 
\end{rema}

It follows from \e{esti:quant2} applied with $\rho=1$ that, 
if $a\in \Gamma^{m}_{1}, b\in \Gamma^{m'}_{1}$ then
\begin{equation}\label{esti:quant4}
\lA [ T_a, T_b]  \rA_{\Fl{H^{\mu}}{H^{\mu-m-m'+1}}}\le 
K M_{1}^{m}(a)M_{1}^{m'}(b).
\end{equation}

If~$a=a(x)$ is a function of~$x$ only, then $T_a$ is called a paraproduct. 
We often use that the following consequence of \eqref{esti:quant1}: 
if~$a\in L^\infty(\xT)$ then~$T_a$ is an operator of order 
$0$, together with the estimate
\begin{equation}\label{esti:quant0}
\forall\sigma\in \xR,\quad \lA T_a u\rA_{H^\sigma}\les \lA a\rA_{L^\infty}\lA u\rA_{H^\sigma}.
\end{equation}
If $a=a(x)$ and $b=b(x)$ then \e{defi:sharp} simplifies to 
$a\sharp b=ab$ and hence \e{esti:quant2sharp} implies that, for any $\rho>0$,
\begin{equation}\label{esti:quant2-func}
\lA T_a T_b  - T_{a b}   \rA_{\Fl{H^{\mu}}{H^{\mu-m-m'+\rho}}}\le 
K \lA a \rA_{\eC{\rho}}\lA b\rA_{\eC{\rho}},
\end{equation}
provided that $a$ and $b$ are in to $\eC{\rho}(\xT)$.

\begin{theo}
$i)$
Given two functions~$a,b$ defined on~$\xR$ we define the remainder
\begin{equation}\label{defi:RBony}
\RBony(a,u)=au-T_a u-T_u a.
\end{equation}
Let~$\alpha\in \xR_+$ and $\beta\in \xR$ be such that~$\alpha+\beta>0$. Then
\be
\lA \RBony(a,u) \rA _{H^{\alpha + \beta-\frac{1}{2}}} 
\leq K \lA a \rA _{H^{\alpha}}\lA u\rA _{H^{\beta}}.\label{Bony}
\ee
$ii)$ Let $\alpha>1/2$. For all $C^\infty$ function $F$ with $F(0)=0$, if $a \in H^{\alpha}(\xT)$ then 
\begin{equation}\label{FBony}
\lA F( a)-T_{F'(a)}a\rA_{H^{2\alpha-\frac{1}{2}}}\le C\left( \lA a\rA_{H^s}\right) \lA a\rA_{H^s}.
\end{equation}
\end{theo}
\begin{prop}\label{lemPa}
Let~$r,\mu\in \xR$ be such that~$r+\mu>0$. If~$\gamma\in\xR$ satisfies 
$$
\gamma\le r  \quad\text{and}\quad \gamma < r+\mu-\frac{1}{2},
$$
then there exists a constant~$K$ such that, for all 
$a\in H^{r}(\xT)$ and all~$u\in H^{\mu}(\xT)$, 
\be\label{pAz1}
\lA au - T_a u\rA_{H^{\gamma}}\le K \lA a\rA_{H^{r}}\lA u\rA_{H^\mu}.
\ee
\end{prop}

We also recall two well-known nonlinear properties. Firstly, 
If~$u_1,u_2 \in H^{s}(\xT)\cap L^{\infty}(\xT)$ and $s\ge 0$ then
\begin{equation}\label{prtame}
\lA u_1 u_2 \rA_{H^{s}}\le K \lA u_1\rA_{L^\infty}\lA u_2\rA_{H^{s}}+K \lA u_2\rA_{L^\infty}\lA u_1\rA_{H^s},
\end{equation}
and hence, for $s>1/2$, 
\begin{equation}\label{prtame2}
\lA u_1 u_2 \rA_{H^{s}}\le K \lA u_1\rA_{H^s}\lA u_2\rA_{H^{s}}.
\end{equation}
Similarly, for~$s>0$ and $F\in C^\infty(\xC^N)$ such that~$F(0)=0$, 
there exists a non-decreasing function~$C\colon\xR_+\rightarrow\xR_+$ 
such that
\begin{equation}\label{esti:F(u)}
\lA F(U)\rA_{H^s}\le C\bigl(\lA U\rA_{L^\infty}\bigr)\lA U\rA_{H^s},
\end{equation}
for any~$U\in (H^s(\xT)\cap L^\infty(\xT))^N$.

\section{Energy estimates and well-posedness of some linear equations}

Recall the linearized equation 
$\partial_t u+i\Lk u=0$, 
where $\Lk\defn \big( (g-\kappamuet\px^2) |D_x|\big)^{\frac{1}{2}}$. 

We gather in this section Sobolev energy estimates for linear equations of the form 
$$
\partial_t \varphi +V\px \varphi+ iL^\mez \big( c L^\mez \varphi\big)+R\varphi=F,
$$
where $V=V(t,x)$ is a real-valued coefficient, $c=c(t,x)$ is a real-valued coefficient bounded from below by $1/2$, $F=F(t,x)$ 
is a given complex-valued source term and $R$ is a time dependent operator of 
order $0$ which means that $R\varphi$ is defined by 
$(R\varphi)(t)=R(t)\varphi(t)$ and $R$ belongs to 
$C^0(\xR_+; \mathcal{L}(H^\mu))$ (for some $\mu$) 
where $\mathcal{L}(H^\mu)$ 
denotes the set of bounded 
operator on $H^\mu(\xT)$. Below we consider various equations 
of this form where, for instance, $R$ is either a multiplication operator 
by some function or the commutator between $V\px$ and a Fourier multiplier.

We also consider paradifferential equations of the form
$$
\partial_t \varphi +T_V\px \varphi+i\mathcal{L} 
\varphi+ R\varphi =F,
$$
where $\mathcal{L}=L^\mez \big(T_c L^\mez \cdot\big)$ 
and $V,c,R$ are as above.

\begin{prop}\label{P:10}Let $T>0$ and $\mu\in [0,+\infty)$. 
Consider $R\in C^0([0,T]; \mathcal{L}(H^\mu))$ and 
real-valued coefficients $V,c$ 
satisfying
$$
V\in C^0([0,T];\eC{1}(\xT)), \quad 
c\in C^0([0,T];\eC{\tdm}(\xT)),
$$
with the $L^\infty_{t,x}$-norm of $c-1$ small enough. 

For any $\varphi_{in}\in H^\mu(\xT)$ and any 
$F\in L^1([0,T];H^\mu(\xT))$, 
there exists a unique $\varphi\in C^0([0,T];H^\mu(\xT))$ such that
\be\label{pn25}
\partial_t \varphi +T_V\px \varphi+ R \varphi+i\mathcal{L}\varphi =F, 
\qquad \varphi_{\arrowvert t=0}=\varphi_{in}.
\ee
Moreover, for any $t\ge 0$,
\be\label{pn26}
\lA \varphi(t)\rA_{H^\mu}\le e^{Ct}\left( \lA \varphi_{in}\rA_{H^\mu}+\lA F\rA_{L^1([0,t];H^\mu)}\right),
\ee
for some constant $C=C(\mu,M)$ depending only on $\mu$ and 
$$
M=\sup_{t\in [0,T]}\Big\{\lA \px V(t)\rA_{L^\infty}+\lA c(t)\rA_{\eC{\tdm}}+\lA R(t)\rA_{\mathcal{L}(H^\mu)}\Big\}.
$$
\end{prop}
\begin{rema}
We often use energy estimates for backward Cauchy problems, 
that is for Cauchy problems on time intervals $[0,T]$ with a data prescribed at time $T$. 
Then the energy estimates read
\be\label{pn27b}
\lA \varphi(t)\rA_{H^\mu}\le e^{CT}\left( \lA \varphi(T)\rA_{H^\mu}+\lA F\rA_{L^1([0,T];H^\mu)}\right).
\ee
\end{rema}
\begin{proof}
As already seen in \e{pb41}, 
$\Lr=L^\mez T_c L^\mez =T_\gamma+R'$ where $R'$ is of order $0$ and 
$$
\gamma=c\ell +\frac{1}{i}\big(\partial_\xi \sqrt{\ell}\big)\sqrt{\ell} \px c.
$$
Up to replacing in \e{pn25}Ê
the remainder $R$ by $R+iR'$, we prove the existence of the solution as limits of approximate problems of the form
 \be\label{pn27}
\partial_t \varphi +T_V\px J_\eps \varphi
+iT_\gamma J_\eps \varphi+ R \varphi=F, 
\qquad \varphi_{\arrowvert t=0}=J_\eps \varphi_{in},
\ee
where $J_\eps$ are smoothing operators. Then \e{pn27} is an ODE 
in Banach spaces and admits a global in time solution denoted by $\varphi_\eps$. 

Set $\gamma^{(3/2)}(t,x,\xi)=c(t,x)\lk(\xi)$, 
which is the principal symbol of $\gamma$. 
As in \cite{ABZ1}, consider the paradifferential 
operator 
$\Lambda_{\mu}$ with symbol $1+(c(t,x)\ell(\xi))^{2\mu/3}$ and, 
given $\eps\in [0,1]$, define $J_\eps$ as the 
paradifferential operator with symbol $\jmath_\eps=\jmath_\eps(t,x,\xi)$ given by
$$
\jmath_\eps=\jmath_\eps^{(0)}+\jmath_\eps^{(-1)}= \exp \big( -\eps \gamma^{(3/2)}\big) -\frac{i}{2}(\partial_x\partial_\xi)\exp \big( -\eps \gamma^{(3/2)}\big).
$$
Recall that the Poisson bracket of two symbols is 
$\{a,b\}=(\pa_x a)(\pa_\xi b)-(\pa_\xi a)(\pa_x b)$. 
Then
\be\label{pn30}
\{\jmath_\eps^{(0)}, \gamma^{(3/2)}\}=0, \quad 
\left\{\jmath_\eps^{(0)},(c\ell)^{2\mu/3}\right\}=0,
\left\{\gamma^{(3/2)},(c\ell)^{2\mu/3}\right\}=0,
\ee
and
$$
\IM \jmath_\eps^{(-1)} = -\frac{1}{2}(\partial_x\partial_\xi) \jmath_\eps^{(0)}.
$$
Of course, for any $\eps >0$, 
$\jmath_\eps \in C^{0}([0,T];\Gamma^m_{3/2}(\xR^d))$ 
for all $m\le 0$, so that $T_{\jmath_\eps}u\in C^{0}([0,T];H^\infty(\xT))$ 
for any 
$u\in C^{0}([0,T];H^{-\infty}(\xT))$. 
Also 
$\jmath_\eps$ is uniformly bounded 
in $C^{0}([0,T];\Gamma^0_{3/2}(\xR^d))$ for all $\eps\in [0,1]$. 
Hence, using \e{esti:quant2sharp} with 
$\rho=3/2$ or \e{esti:quant2} with $\rho=1$, 
we have the following estimates (uniformly in $\eps$): 
\be\label{B6-B8}
\begin{aligned}
&\lA [J_\eps, T_\gamma]u\rA_{H^\mu} \le 
C\lA u\rA_{H^\mu}, \qquad &
&\lA (J_\eps)^*u -J_\eps u\rA_{H^{\mu+\tdm}} \le 
C\lA u\rA_{H^\mu},\\
& \lA \left[\Lambda_\mu,L^\mez(T_cL^\mez \cdot)\right]u\rA_{L^{2}}
\le 
C\lA u\rA_{H^\mu},
&&\lA \left[\Lambda_\mu,J_\eps \right]u\rA_{H^{\tdm}}\le 
C\lA u\rA_{H^\mu},\\
&\lA\left[\Lambda_\mu,T_{V}\partial_x J_\eps \right]u\rA_{L^{2}}\le 
C\lA V\rA_{\eC{1}}\lA u\rA_{H^\mu},
&&\lA [J_\eps, T_V\px]u\rA_{H^\mu} \le 
C\lA V\rA_{\eC{1}}\lA u\rA_{H^\mu},
\end{aligned}
\ee
for some constant $C$ depending only on $\lA c\rA_{W^{3/2,\infty}}$ and 
uniform in $\eps\in [0,1]$. 
%


Recall that, by notation, $\varphi_\eps$ is the unique solution to \e{pn27} and introduce 
$\dot{\varphi}_\eps \defn \Lambda_\mu\varphi_\eps$. 
Using the fact that $\Lambda_\mu$ is invertible (for $c-1$ small enough) 
and the preceding estimates, we deduce that
\be\label{pn31}
\partial_t \dot{\varphi}_\eps +T_V\px J_\eps \dot{\varphi}_\eps
+ \Lambda_\mu R \Lambda_\mu^{-1}\dot{\varphi}_\eps+iT_\gamma J_\eps \dot{\varphi}_\eps=F_\eps, 
\qquad \dot{\varphi}_\eps {\arrowvert}_{t=0}=\Lambda_\mu J_\eps \varphi_{in},
\ee
where 
\be\label{pn31b}
\lA F_\eps\rA_{L^1([0,T];L^2)}
\le C(M)\Big\{\lA \varphi_\eps\rA_{L^1([0,T];H^\mu)}+\lA F\rA_{L^1([0,T];H^\mu)}\Big\}.
\ee
Write 
$\frac{d}{dt}\lA \dot{\varphi}_\eps\rA_{L^2}^2 
= 2 \RE \left\langle \partial_{t}\dot{\varphi}_\eps,\dot{\varphi}_\eps\right\rangle$,  
where $\langle\cdot,\cdot\rangle$ denotes the scalar product in $L^{2}(\xT)$, and hence
\begin{align*}
&\frac{d}{dt}\lA \dot{\varphi}_\eps\rA_{L^2}^2 
= - \left\langle (P+P^*)\dot{\varphi}_\eps,\dot{\varphi}_\eps\right\rangle \quad\text{with}\\
&P=T_V\px J_\eps + \Lambda_\mu R \Lambda_\mu^{-1}+iT_\gamma J_\eps.
\end{align*}
To estimate the operator norm of 
$P+P^*$, there are two ingredients. Firstly, we replace $J_\eps^*$ by $J_\eps+(J_\eps^*-J_\eps)$ and commute $J_\eps$ with $T_V\px$ and $T_\gamma$. This produces remainder terms that are estimated by means of \e{B6-B8}. 
The proof is then reduced to the case without $J_\eps$ and it suffices to estimate 
the operator norm of $\tilde{P}+\tilde{P}^*$ where 
$\wide P=T_V\px + \Lambda_\mu R \Lambda_s^{-1}+iT_\gamma$. Since $\Lambda_\mu R \Lambda_\mu^{-1}$ is bounded from $L^\infty([0,T];L^2(\xT))$ into itself with an operator norm estimated by $M$, it remains only to estimate 
$T_V\px +iT_\gamma+\big(T_V\px +iT_\gamma)^*$, which can be done directly by means of the paradifferential rule \e{esti:quant3}. 
We conclude that 
\begin{equation}\label{pn612}
\frac{d}{dt}\lA \dot\varphi_\eps\rA_{L^2}^2 
\le C(M) \lA \dot\varphi_\eps\rA_{L^2}^2 + \la \left\langle  2F_\eps,
\dot\varphi_\eps\right\rangle\ra.
\end{equation}

We thus 
obtain a uniform estimate for the $L^\infty([0,T];L^2)$-norm of 
$\dot\varphi_\eps$ (from Gronwall's inequality and \e{pn31b}) which gives a uniform estimate for the $L^\infty([0,T];H^\mu)$-norm of $\varphi_\eps$. From this uniform estimate and classical arguments (see \cite{MePise}), one deduce the existence of a solution in $L^\infty([0,T];H^\mu(\xT))$. 
The uniqueness is obtained by considering the equation satisfied by the difference of two solutions and performing an $L^2$-energy inequality (using similar arguments to those used above). The continuity in time of the solution is proved as in \cite[\S 6.4]{ABZ1}. 
\end{proof}

\begin{lemm}\label{L2}
Consider real-valued coefficients $V,c$ 
satisfying
$$
V\in C^0([0,T];\eC{1}(\xT)), \quad 
c\in C^0([0,T];\eC{\tdm}(\xT)),
$$
with the $L^\infty_{t,x}$-norm of $c-1$ small enough. 
Consider also $R\in C^0([0,T]; \mathcal{L}(L^2))$. 

$i)$ For any $\varphi_{in}\in L^2(\xT)$ and any 
$F\in L^1([0,T];L^2(\xT))$, 
there exists a unique $\varphi\in C^0([0,T];L^2(\xT))$ such that
\be\label{n25}
\partial_t \varphi +V\px \varphi+ R \varphi+iL^\mez \big( cL^\mez \varphi\big)=F, 
\qquad \varphi_{\arrowvert t=0}=\varphi_{in}.
\ee
Moreover, for any $t\ge 0$,
\be\label{n26}
\lA \varphi(t)\rA_{L^2}\le \exp\left( \int_0^t M(t')\, dt'\right)
\left( \lA \varphi_{in}\rA_{L^2}+\lA F\rA_{L^1([0,t];L^2)}\right),
\ee
with $M(t')=\lA \px V(t')\rA_{L^\infty}+\lA R(t')\rA_{\mathcal{L}(L^2)}$.

$ii)$ Let $\mu\in [0,3/2]$. 
Assume that 
$V\in C^0([0,T];H^2(\xT))$, $c\in C^0([0,T];H^3(\xT))$ and $R\in C^0([0,T]; \mathcal{L}(H^\mu))$. 
If $\varphi_{in}\in H^\mu(\xT)$ 
and $F\in L^1([0,T];H^\mu(\xT))$, then, 
for any $t\ge 0$,
\be\label{n26-H1}
\lA \varphi(t)\rA_{H^\mu}\le \exp\left( \int_0^t M(t')\, dt'\right)
\left( \lA \varphi_{in}\rA_{H^\mu}+\lA F\rA_{L^1([0,t];H^\mu)}\right),
\ee
with $M(t')=\lA V(t')\rA_{H^2}+\lA c(t')\rA_{H^3}+\lA R(t')\rA_{\mathcal{L}(H^\mu)}$.
\end{lemm}
\begin{rema}
Consider a backward Cauchy problem, 
that is a Cauchy problem  
with a data prescribed at time $T$. Then \e{n26} implies that
\be\label{n26b}
\lA \varphi(t)\rA_{L^2}\le \exp\left( \int_0^T M(t')\, dt'\right)
\left( \lA \varphi(T)\rA_{L^2}+\lA F\rA_{L^1([0,T];L^2)}\right),
\ee
with $M(t')=\lA \px V(t')\rA_{L^\infty}+\lA R(t')\rA_{\mathcal{L}(L^2)}$. 
\end{rema}
\begin{proof}
$i)$ The existence of the solution can be deduced from the previous proposition, writing  
$$
\partial_t \varphi +V\px \varphi+ R \varphi+iL^\mez \big( cL^\mez \varphi\big)
$$
under the form
$$
\partial_t \varphi +T_V\px \varphi+ R'\varphi
+iL^\mez \big( T_c L^\mez \varphi\big)
$$
where 
\be\label{pAn10}
R'\varphi=R\varphi+(V\px\varphi-T_V\px \varphi)+
i L^\mez\big( (c-T_c)L^\mez \varphi\big).
\ee
Indeed, $R'$ belongs to  $C^0([0,T]; \mathcal{L}(L^2))$ in view of \e{Bony} and \e{pAz1}. 

In order to see that the energy estimate does not depend on the norm of $c$, start from
$\frac{d}{dt}\lA \varphi\rA_{L^2}^2 
= 2 \RE \left\langle \partial_{t}\varphi,\varphi\right\rangle$. 
Since $\RE \big\langle i L^\mez ( c L^\mez \varphi),\varphi\big\rangle=0$, 
we obtain that
\begin{equation*}
\frac{d}{dt}\lA \varphi\rA_{L^2}^2 
=2\RE \left\langle - V\px \varphi -R\varphi-F,\varphi\right\rangle.
\end{equation*}
Hence, integrating by parts, 
\begin{equation}\label{n612}
\frac{d}{dt}\lA \varphi\rA_{L^2}^2 
= \RE\left\langle  ( (\px V )-2R)\varphi -2F,\varphi \right\rangle,
\end{equation}
and the result easily follows from Gronwall's inequality. 

$ii)$ This follows from \e{pn26} and the fact that the remainder $R'$ in \e{pAn10} belongs to 
$C^0([0,T]; \mathcal{L}(H^\mu))$ in view of \e{Bony} and \e{pAz1}. 
\end{proof}

\section{Changes of variables}\label{S:AC1}

\newcommand{\bb}{b_0} 

Recall the operator 
\begin{equation}  \label{P from Thomas}
\wide P := \pa_t + V \pa_x + i L^{\frac12} \big( c L^{\frac12} \cdot \big)+ R_2, 
\end{equation}
where $L := (g - \kappamuet \pa_{xx})^{\frac12}G(0)^{\frac12}$, 
the operator $R_2$ is of order zero, and $c(t,x)$, $V(t,x)$ are real-valued functions.
Consider a time-depending change of the space variable (namely a diffeomorphism of $\T$) 
and its inverse,
\[
x = y + \tilde \b_1(t,y) \quad 
\Leftrightarrow \quad 
y = x + \b_1(t,x), 
\]
$x,y \in \T$, $t \in \R$, 
with $\| \pa_y \tilde \b_1 \|_{L^\infty}, \| \pa_x \b_1 \|_{L^\infty} \leq 1/2$. 
Introduce a self-adjoint variant of the pull-back operators, defined by
\begin{align} \label{def Psi1}
(\Psi_1 h)(t,y) & := (1 + \pa_y \tilde\b_1(t,y))^{\frac12} h(t, y + \tilde \b_1(t,y)),\quad 
\\ 
\label{def Psi1 inv}
(\Psi_1^{-1}h)(t,x) & := (1 + \pa_x \b_1(t,x))^{\frac12} h(t, x + \b_1(t,x)),
\end{align}
and note that $\Psi_1, \Psi_1^{-1}$ are self-adjoint with respect to the standard $L^2(\T)$ scalar product in space, for any $t$.
We want to compute $\Psi_1 Q_0 \Psi_1^{-1}$ when $Q_0$ is a Fourier multiplier 
(the analysis below applies more generally assuming only that $Q_0$ is a pseudo-differential operator).

\subsection{Change of variable as a flow map} 
Introduce a parameter $\tau\in [0,1]$  
and consider a diffeomorphism of $\T$ (depending on $(\t,t)$) and its inverse,
\[
x = y + \tilde \b(\t,t,y) \quad 
\Leftrightarrow \quad 
y = x + \b(\t,t,x), 
\]
$x,y \in \T$, $\t \in [0,1]$, $t \in \R$, 
where $\b$ and $\tilde \b$ are such that 
$\| \pa_y \tilde \b \|_{L^\infty}, \| \pa_x \b \|_{L^\infty} \leq 1/2$ and
\[
\tilde \b\arrowvert_{\tau=0} = 0, \quad \b\arrowvert_{\tau=0}= 0,\quad 
\tilde \b\arrowvert_{\tau=1} = \tilde \b_1, 
\quad \b\arrowvert_{\tau=1} = \b_1.
\]
We denote 
\begin{align}
(\Psi(\t) h)(t,y) := 
(1 + \pa_y \tilde\b(\t,t,y))^{\frac12} h(t, y + \tilde \b(\t, t,y)); 
\label{Psi} 
\\
(\Psi(\t)^{-1} h)(t,x) := (1 + \pa_x \b(\t,t,x))^{\frac12}  h(t, x + \b(\t, t,x)).
\label{Psi -1}
\end{align}
Then $\Psi_1 = \Psi(1)$. 
The reason to introduce the parameter $\tau$ is that $\Psi(\t)$ satisfies an equation of the form
\begin{equation} \label{flow Psi}
\pa_\t \Psi(\t) = F(\t) \Psi(\t), \quad \Psi(0) = I, 
\end{equation}
namely $\pa_\t (\Psi(\t) h) = F(\t) (\Psi(\t)h)$, $\Psi(0)h = h$ for all $h$, 
where 
\begin{equation} \label{F f pax}
F(\t) = \bb(\t,t,y) \pa_y + \frac12 (\pa_y \bb)(\t,t,y), \quad 
\bb(\t,t,y) := \frac{\pa_\t \tilde \b(\t, t,y)}{1 + \pa_y \tilde \b(\t, t, y)}\,.
\end{equation}
Assume that $Q_0$ is a Fourier multiplier with symbol $q_0(\xi)$ of order $m\le 3/2$. 
We seek a pseudo-differential operator $Q(\tau)$ of order $m$ such that the difference
\begin{equation}  \label{def R dritto}
R(\t)\defn Q(\tau)\Psi(\t)-\Psi(\t)Q_0
\end{equation} 
is an operator of order $0$. 
Commuting $Q(\t)$ with the equation $\pa_\t \Psi(\t) = F(\t) \Psi(\t)$ one obtains
\begin{align*}
\pa_\t (Q(\t)\Psi(\t)) 
& = Q(\t) F(\t) \Psi(\t) + (\pa_\t Q(\t)) \Psi(\t)
\\
& = F(\t) Q(\t) \Psi(\t)
+ \big( [ Q(\t),F(\t) ] \Psi(\t) + (\pa_\t Q(\t))\Psi(\t)\big).
\end{align*}
On the other hand $\pa_\t (\Psi(\t)Q_0) = F(\t) \Psi(\t)Q_0$. By combining both equations we 
obtain that $R$ satisfies 
\begin{equation}  \label{def R storto}
\pa_\t R(\t)=F(\t)R(\t)+\mathcal{R}_1(\t)\Psi(\t),\quad 
\mathcal{R}_1(\t):= [ Q(\t),F(\t) ] + \pa_\t Q(\t).
\end{equation}
The analysis is then in two steps. 
The main step consists in proving that $Q(\t)$ can be so chosen that 
$Q(\tau=0)=Q_0$ (then $R(0)=0$) and $\mathcal{R}_1(\t)$ is of order $0$. 
Then, by using an $L^2$-energy estimate for the hyperbolic equation $\pa_\t u = \bb \pa_y u+f$, 
one deduces an estimate for the operator norm 
of $R(\t)$ uniform in $\t$ (and hence the desired estimate for $\t=1$).
Here we describe in details only the main step, 
as the $L^2$-energy estimate is a standard argument.

\subsection{Expansion of the symbol}

Let $p(\t,t,x,\xi)$ be the symbol of $Q(\t)$. 
To obtain $\mathcal{R}_1$ of order zero amounts to seek $p$ such that 
$\pa_\t p - \s_{[F,Q]}$ has order zero (where $\s_{[F,Q]}$ is the symbol of $[F,Q]$), 
and $p|_{\t = 0} = q_0$.
The asymptotic expansion of $\s_{[F,Q]}$ is
\begin{equation}  \label{asymptotic expansion}
\s_{[F,Q]} \sim \sum_{\a = 1}^\infty \frac{1}{i^\a \, \a!} \, 
\big\{ (\pa_\xi^\a f) (\pa_x^\a p) 
- (\pa_\xi^\a p) (\pa_x^\a f) \big\},
\end{equation}
where $f(\t,t,x,\xi) := i \bb(\t,t,x) \xi + \frac12 (\pa_x \bb)(\t,t,x)$ is the symbol of $F(\t)$ 
(we rename $x$ the space variable).
Since $m\le 3/2\le 2$ by assumption, 
it is enough to determine the principal and the sub-principal symbols of $p$. 
Thus we write $p = p_0 + p_1$, where $p_0$ has order $m$ and $p_1$ has order $m-1$.
The equations for $p_0, p_1$ are 
\begin{alignat}{2}  \label{eq p0}
\pa_\t p_0 & = \bb \pa_x p_0 - \xi (\pa_x \bb) \pa_\xi p_0, \quad 
& p_0|_{\t = 0} & = q_0,
\\
\label{eq p1}
\pa_\t p_1 & = \bb \pa_x p_1 - \xi (\pa_x \bb) \pa_\xi p_1 + z, 
\quad \ 
& p_1|_{\t = 0} & = 0,
\end{alignat}
where 
\begin{equation} \label{def z}
z := \frac{i}{2} (\pa_{xx} \bb) (\pa_\xi p_0 + \xi \pa_{\xi\xi} p_0).
\end{equation}
If $p_0, p_1$ satisfy \eqref{eq p0},\eqref{eq p1}, 
then it follows from standard symbolic calculus for pseudo-differential operators (similar to 
\e{esti:quant2sharp}) that $\mathcal{R}_1(\t)$, defined in \eqref{def R storto}, 
is an operator of order $0$ satisfying
\begin{equation}  \label{remainder estimate appendix}
\lA \mathcal{R}_1(\t)\rA_{\Lr(L^2)} 
+ \lA \mathcal{R}_1(\t)\rA_{\Lr(H^\tdm)}
\les \big(M^{m}_r(p_0)+M^{m-1}_r(p_1)\big)\lA \bb(\t)\rA_{\eC{r}}
\end{equation}
with $r$ large enough 
(here the semi-norms $M^m_\rho$ are as defined by \e{defi:norms}; one has to consider $r$ large enough 
because we are here considering pseudo-differential operators instead of paradifferential ones).

Equation \eqref{eq p0} can be solved by the characteristics method: 
if $x(\t), \xi(\t)$ solve
\begin{equation}  \label{Ham char}
\frac{d}{d\t} x(\t) = - \bb(\t,t,x(\t)), \quad 
\frac{d}{d\t} \xi(\t) = \xi(\t) (\pa_x \bb)(\t,t,x(\t)),
\end{equation}
then
\begin{equation}  \label{along}
p_0(\t,t,x(\t),\xi(\t)) = p_0(0,t,x(0),\xi(0)) \quad \forall \t.
\end{equation}
Now, by \eqref{F f pax}, the first equation in \eqref{Ham char} is 
\[
0 = \{ 1 + (\pa_x \tilde \b)(\t,t,x(\t)) \} \, x'(\t) + (\pa_\t \tilde \b)(\t,t,x(\t)) 
= \frac{d}{d\t} \big\{ x(\t) + \tilde \b(\t,t,x(\t)) \big\},
\]
whence 
\begin{equation}  \label{flow x(tau)}
x(\t) + \tilde \b(\t,t,x(\t)) = x(0) + \tilde \b(0,t,x(0)) 
= x(0).
\end{equation}
Applying the inverse diffeomorphism, we get $x(\t) = x(0) + \b(\t,t,x(0))$. 
This is the solution $x(\t)$ of the first equation in \eqref{Ham char} with initial datum $x(0)$. 
Also, one verifies that 
\begin{equation}  \label{flow xi(tau)}
\xi(\t) = \xi(0) \big( 1 + (\pa_x \tilde\b)(\t,t,x(\t)) \big)
\end{equation}
satisfies the second equation in \eqref{Ham char}, 
because $x(\t)$ satisfies the first equation in \eqref{Ham char}, 
$\bb$ is given by \eqref{F f pax}, and 
\[
\pa_x \bb(\t,t,x) = \frac{\pa_{\t x} \tilde\b(\t,t,x)}{1 + \pa_x \tilde\b (\t,t,x)}
\, - \frac{\pa_\t \tilde\b(\t,t,x) \pa_{xx} \tilde\b(\t,t,x)}{[1 + \pa_x \tilde\b (\t,t,x)]^2} \,.
\]
Hence we deduce a formula for the backward flow of \eqref{Ham char}: 
fixed any $\t_1 \in [0,1]$, and given any $(x_1, \xi_1)$, 
the solution $(x(\t), \xi(\t))$ of \eqref{Ham char} with initial datum $(x(0), \xi(0)) = (x_0, \xi_0)$ satisfies $(x(\t_1), \xi(\t_1)) = (x_1, \xi_1)$ if the initial datum is
\begin{equation}  \label{back flow}
x_0 = x_1 + \tilde\b(\t_1,t,x_1), \quad 
\xi_0 = \frac{\xi_1}{1 + \pa_x \tilde\b (\t_1,t,x_1)} \,.
\end{equation}
As a consequence, using \eqref{along} and the initial datum in \eqref{eq p0}, 
we get 
\begin{align*}
p_0(\t_1,t,x_1,\xi_1) &= p_0(0,t,x_0,\xi_0) = q_0(t,x_0,\xi_0) 
\\ & 
= q_0 \Big( t, x_1 + \tilde\b(\t_1,t,x_1), \frac{\xi_1}{1 + \pa_x \tilde\b (\t_1,t,x_1)} \Big).
\end{align*}
We  have a formula for the solution $p_0(\t,t,x,\xi)$ of \eqref{eq p0}:
\begin{equation}  \label{sol p0}
p_0(\t,t,x,\xi) = q_0 \Big( t, x + \tilde\b(\t,t,x), \frac{\xi}{1 + \pa_x \tilde\b (\t,t,x)} \Big).
\end{equation}
Now we study equation \eqref{eq p1}. By the definition of $(x(\t),\xi(\t))$,
\begin{equation}  \label{p1 int z}
p_1(\t,t,x(\t),\xi(\t)) = \int_0^\t z(s,t,x(s),\xi(s)) \, ds,
\end{equation}
where $z$ is given in \eqref{def z}. 
We examine $z$ in detail. 
By \eqref{sol p0}, for $k=1,2$,
\[
\pa_\xi^k p_0(\t,t,x,\xi) 
= (\pa_\xi^k q_0) \Big( t, x + \tilde\b(\t,t,x), \frac{\xi}{1 + \pa_x \tilde\b (\t,t,x)} \Big) \, \frac{1}{ [1 + \pa_x \tilde\b(\t,t,x)]^k }
\]
for all $\t,t,x,\xi$. 
Hence along the curves $(x(s), \xi(s))$, 
by \eqref{flow x(tau)},\eqref{flow xi(tau)}, one has 
\[
(\pa_\xi^k p_0)(s,t,x(s),\xi(s)) 
= \frac{(\pa_\xi^k q_0)( t, x_0, \xi_0)}{ [1 + \pa_x \tilde\b(s,t,x(s))]^k }\,, 
\]
where $(x_0, \xi_0) := (x(0), \xi(0))$,
and therefore, using \eqref{flow xi(tau)} again,
\[
(\pa_\xi p_0 + \xi \pa_{\xi \xi} p_0)(s,t,x(s), \xi(s)) 
= \frac{\pa_\xi q_0 (t, x_0, \xi_0) + \xi_0 \pa_{\xi \xi} q_0(t, x_0, \xi_0)}
{ 1 + \pa_x \tilde\b(s,t,x(s)) }\,.
\]
Now we note that 
\[
\frac{(\pa_{xx} \bb)(s,t,x(s),\xi(s))}{ 1 + (\pa_x \tilde\b)(s,t,x(s))} \, 
= \frac{d}{ds} \bigg\{ 
\frac{ (\pa_{xx} \tilde\b) (s,t,x(s)) }{ [1 + (\pa_x \tilde\b)(s,t,x(s)) ]^2 } \bigg\}\,,
\]
as it can be verified by a straightforward calculation, 
using also \eqref{Ham char} and the definition \eqref{F f pax} of $\bb$.
Hence, recalling the definition \eqref{def z} of $z$, 
\[
z(s,t,x(s),\xi(s)) = \frac{i}{2} 
\{ \pa_\xi q_0 (t, x_0, \xi_0) + \xi_0 \pa_{\xi \xi} q_0(t, x_0, \xi_0) \}
\frac{d}{ds} \bigg\{ \frac{ (\pa_{xx} \tilde\b) (s,t,x(s)) }{ [1 + (\pa_x \tilde\b)(s,t,x(s)) ]^2 } \bigg\},
\]
and, by \eqref{p1 int z},
\[
p_1(\t,t,x(\t),\xi(\t)) 
= \frac{i}{2} \{ \pa_\xi q_0 (t, x_0, \xi_0) + \xi_0 \pa_{\xi \xi} q_0(t, x_0, \xi_0) \}
\frac{ (\pa_{xx}\tilde\b) (\t,t,x(\t)) }{ [ 1 + (\pa_x \tilde\b)(\t, t,x(\t)) ]^2 }
\]
because $\tilde\b|_{\t = 0}=0$. We use the backward flow as above: 
given $\t_1, x_1, \xi_1$, 
the solution $(x(\t), \xi(\t))$ of \eqref{Ham char} with initial datum $(x(0), \xi(0)) = (x_0, \xi_0)$ satisfies $(x(\t_1), \xi(\t_1)) = (x_1, \xi_1)$ if the initial datum is \eqref{back flow}. 
Therefore, replacing $(x_0, \xi_0)$ by \eqref{back flow} in the last equality, 
we get a formula for $p_1$, which, writing $\t,x,\xi$ instead of $\t_1, x_1, \xi_1$, is
\begin{align} \label{sol p1}
& p_1(\t,t,x,\xi) 
= \frac{i}{2} \Big\{ (\pa_\xi q_0) \Big( t, x + \tilde\b(\t,t,x) , 
\frac{\xi}{1 + \pa_x \tilde\b(\t, t, x)} \Big) 
\\ 
& \  
+ \frac{\xi}{1 + \pa_x \tilde\b(\t, t, x)} 
(\pa_{\xi \xi} q_0) \Big( t, x + \tilde\b(\t,t,x), 
\frac{\xi}{1 + \pa_x \tilde\b(\t, t, x)} \Big) \Big\}
\, \frac{ \pa_{xx}\tilde\b (\t,t,x) }{ ( 1 + \pa_x \tilde\b(\t , t,x) )^2 }\,.
\notag 
\end{align}

\subsection{Conjugation of $L$}

We fix $q_0(\xi)$ to be the symbol of $L$ (see \e{defi:ell-lambda}) with a cut-off around $\xi = 0$, 
namely 
\[
q_0(\xi) 
:= (g + \kappamuet \xi^2)^{\frac12} \lambda(\xi)^{\frac12} \chi(\xi)
= (g + \kappamuet \xi^2)^{\frac12} \la \xi \ra^\mez \tanh^\mez(b\la \xi\ra) \chi(\xi), 
\]
where $\chi(\xi)$ is the cut-off function of Proposition \ref{T24}. 
Note that $\Op(q_0) = L$ on the periodic functions, as their symbols coincide 
at any $\xi \in \xZ$, and therefore no remainder is produced replacing $L$ by $\Op(q_0)$. 
In the previous section we have constructed $p_0, p_1$, 
and we have defined $p := p_0 + p_1$, $Q(\t) := \Op(p)$. 
Then $\mR_1(\t)$ 
defined in \eqref{def R storto} is an operator of order zero and it satisfies estimate \eqref{remainder estimate appendix}.
Now observe, in view of \e{def R storto}, that for any function $u_0\in L^2(\xT)$, 
$R(\tau)u_0$ solves an hyperbolic evolution equation. Using the energy estimate \e{n26}, we deduce that 
the difference $R(\t) := Q(\t) \Psi(\t) - \Psi(\t) L$ (see \eqref{def R dritto})
is also of order zero, and it satisfies the same estimate \eqref{remainder estimate appendix} 
as $\mR_1(\t)$. 
As a consequence, the conjugate of $L$ is 
\begin{equation}  \label{conj L}
\Psi(\t) L \Psi(\t)^{-1} = Q(\t) + \mR_2(\t), \quad \mR_2(\t) := - R(\t) \Psi(\t)^{-1}
\end{equation}
and $\mR_2(\t)$ satisfies the same estimate \eqref{remainder estimate appendix} as $R(\t)$. 
By formula \eqref{sol p0}, $p_0 = q_0 ( \xi (1 + \pa_x \tilde \b)^{-1})$. 
We expand 
\begin{equation}  \label{p0 Taylor}
p_0 = (1 + \pa_x \tilde \b)^{-\frac32} q_0 + r, 
\end{equation}
where the remainder $r$ 
satisfies $\| \Op(r) \|_{\mL(H^\mu, H^{\mu + 1/2})} \les \| \pa_x \tilde \b \|_{H^{\mu + \rho}}$ for all $\mu \geq 0$, for some absolute constant $\rho$ large enough, because 
\[
g + \kappamuet \xi^2 h^2
= h^2 (g + \kappamuet \xi^2) \Big( 1 + \frac{g (1 - h^2)}{h^2 (g + \kappamuet \xi^2)} \Big), 
\quad 
h := (1 + \pa_x \tilde\b)^{-1},
\]
and then use Taylor expansion for the square root of the last factor.
The second component $p_1$ is given by formula \eqref{sol p1}.
By Taylor expansion,
\[
\big| q_0'(\xi) - \tfrac32 \sqrkappa |\xi|^{-\frac12} \xi \big| 
\les (1 + |\xi|)^{-\frac32}, \quad 
\big| q_0''(\xi) - \tfrac34 \sqrkappa |\xi|^{-\frac12} \big| 
\les (1+|\xi|)^{-\frac52}, 
\]
so that we calculate
\begin{equation}  \label{p1 Taylor}
p_1 = i \tfrac98 (\pa_{xx} \tilde\b) \sqrkappa (1 + \pa_x \tilde\b)^{-\frac52} 
|\xi|^{-\frac12} \xi  \, \chi(\xi) + r, 
\end{equation}
where the remainder $r$ satisfies $\| \Op(r) \|_{\mL(H^\mu, H^{\mu + 3/2})} 
\les \| \pa_x \tilde \b \|_{H^{\mu + \rho}}$ for all $\mu \geq 0$, for some $\rho$ large enough. 
Assume that $\| \pa_\tau \tilde\b \|_{H^\mu} \les \| \tilde\b \|_{H^\mu}$ 
(this bound holds for the choice of $\beta$ we make below). 
By \eqref{conj L}, \eqref{p0 Taylor}, \eqref{p1 Taylor}, we have 
\begin{equation}  \label{conj L explicit}
\Psi(\t) L \Psi(\t)^{-1} 
= (1 + \pa_x \tilde \b)^{-\frac32} L 
+ \tfrac98 \sqrkappa (\pa_{xx} \tilde\b) (1 + \pa_x \tilde\b)^{-\frac52} 
|D_x|^{-\frac12} \pa_x + \mR_{0,1},
\end{equation}
where $\mR_{0,1}$ is defined by difference and it satisfies $\| \mR_{0,1} \|_{\mL(H^\mu, H^\mu)} 
\les \| \pa_x \tilde\b \|_{H^{\mu + \rho}}$ for all $\mu \geq 0$, for some $\rho$ large enough.  
With similar calculations, one proves that for any $r \in \R$
\begin{equation}  \label{conj Dx}
\Psi(\t) |D_x|^r \Psi(\t)^{-1} = (1 + \pa_x \tilde \b)^{-r} |D_x|^r + \mR_{0,2}
\end{equation}
where $\mR_{0,2}$ is defined by difference and it satisfies 
$\| \mR_{0,2} \|_{\mL(H^\mu, H^{\mu-r+1})} \les \| \pa_x \tilde\b \|_{H^{\mu + \rho}}$.

\subsection{Conjugation of $\wide P$}
We conjugate the operator in \eqref{P from Thomas} by $\Psi_1 := \Psi(1) = \Psi(\t)|_{\t = 1}$. 
From symbolic calculus it follows that
\be \label{formula LcL}
L^{\frac12} c L^{\frac12} 
= cL - \frac34 \sqrkappa (\pa_x c) \pa_x |D_x|^{-\frac12} + \mR_{0,3}, 
\ee
where $\mR_{0,3}$ is defined by difference and 
it satisfies $\| \mR_{0,3} \|_{\mL(H^\mu, H^{\mu+\frac12})} \les \| \pa_x c \|_{H^{\mu+\rho}}$ 
for all $\mu \geq 0$, for some $\rho$ large enough. We recall that $c-1$ is small, and therefore $\pa_x c$ is small. By definition (see \eqref{Psi},\eqref{Psi -1}), and recalling that $\tilde \b|_{\t = 1} = \tilde\b_1$, $\b|_{\t = 1} = \b_1$, we directly calculate 
\[ 
\Psi_1 \pa_t \Psi_1^{-1} = \pa_t + a_1 \pa_x + r_1, 
\quad 
\Psi_1 \pa_x \Psi_1^{-1} = a_2 \pa_x + r_2, 
\] 
where 
\begin{equation} \label{conj a1 a2}
a_1(t,x) := (\pa_t \b_1)(t,x + \tilde \b_1(t,x)),
\quad 
a_2(t,x) := (1 + \pa_x \tilde \b_1(t,x))^{-1}\,,
\end{equation}
and 
\begin{align*} 
r_1(t,x) & := \frac12 (\pa_{tx} \b_1)(t,x + \tilde\b_1 (t,x)) \, (1 + \pa_x \tilde\b_1(t,x)),
\\ 
r_2(t,x) & := \frac12 (1 + \pa_x \tilde \b_1(t,x)) \, (\pa_{xx} \b_1) (t, x + \tilde\b_1(t,x)).
\end{align*}
The conjugate of any multiplication operator $h \mapsto ah$ is the multiplication operator 
$h \mapsto (\tilde B a) h$,
\[ 
\Psi_1 a \Psi_1^{-1} = (\tilde B a), \quad (\tilde B a)(t,x) := a(t, x + \tilde\b_1(t,x)).
\] 
Thus 
\begin{align*} 
\Psi_1 \wide P \Psi_1^{-1} 
& = \pa_t + a_3 \pa_x 
+ i a_4 L 
+ i a_5 \pa_x |D_x|^{-\frac12} 
+ \tilde R_3 
\end{align*}
where 
\begin{align} 
a_3 & := a_1 + (\tilde B V) a_2, \qquad \ 
a_4 := (\tilde B c) (1 + \pa_x \tilde\b)^{-\frac32}, \quad  
\notag \\
a_5 & := 
- \frac34 \sqrkappa
\Big\{ -\frac{3}{2}\, (\tilde B c) (1 + \pa_x \tilde\b)^{-\frac52} (\pa_{xx} \tilde\b)
+ (\tilde B(\pa_x c)) (1 + \pa_x \tilde\b)^{-\frac12} \Big\},
\notag \\
\label{tilde R3 App}
\tilde R_3 & := r_1 + (\tilde B V) r_2 + i (\tilde B c) \mR_{0,1} 
+ - i \frac34 (\tilde B \pa_x c) r_2 (1 + \pa_x \tilde \b)^{\frac12} |D_x|^{-\frac12} 
+ \mR_{0,2} 
\\ & \quad 
+ i \Psi_1 \mR_{0,3} \Psi_1^{-1} 
+ \Psi_1 R_2 \Psi_1^{-1},  
\notag
\end{align}
$\mR_{0,1}$ is defined in \e{conj L explicit} with $\t = 1$, 
$\mR_{0,2}$ is defined in \e{conj Dx} with $\t = 1$, $r = -1/2$, 
and $\mR_{0,3}$ is defined in \e{formula LcL}. 
The remainder $\tilde R_3$ 
is of order zero and it is estimated in Lemma \ref{lemma:Psi}.
Moreover, as it is immediate to verify, $a_5 = - \frac34 \sqrkappa \pa_x a_4$. 
We choose $\b_1, \tilde\b_1$ such that the highest order coefficient $a_4$ 
is independent of $x$. This means 
\begin{equation} \label{cc1}
a_4(t,x) 
= c(t,x + \tilde\b_1(t,x)) \, (1 + \pa_x \tilde\b_1(t,x))^{-\frac32} = m(t) \quad \forall x \in \T,
\end{equation}
for some function $m(t)$ independent of $x$. 
Applying the inverse diffeomorphism, 
this is equivalent to 
\[ 
c(t,x) \, (1 + \pa_x \b_1(t,x))^{\frac32} = m(t) \quad \forall x \in \T.
\] 
This implies $1 + \pa_x \b_1(t,x) = m(t)^{\frac23} c(t,x)^{-\frac23}$,
which, after an integration in $dx$, gives 
\[ 
m(t) = \Big( \frac{1}{2\p} \int_\T c(t,x)^{-\frac23} \, dx \Big)^{-\frac32}.
\] 
Hence $m$ in \eqref{cc1} is determined. 
We fix $\b_1$ as
\[ 
\b_1(t,x) = \pa_x^{-1} [m(t)^{\frac23} c(t,x)^{-\frac23} - 1],
\] 
and then we fix $\b(\t,t,x) := \t \b_1(t,x)$. 
As a consequence, $\tilde\b(\t,t,y)$, $\tilde \b_1$ are also determined. 
Since $a_4(t,x) = m(t)$ is independent of $x$, 
it follows that $a_5 = - \frac34 \sqrkappa \pa_x a_4 = 0$ 
(as it was natural to expect, because the vector field in $\wide P$ is anti-selfadjoint and the transformation $\Psi$ preserves this structure).
We have conjugated $\wide P$ to 
\begin{equation} \label{conj tilde P 2} 
\wide P_1 := \Psi_1 \wide P \Psi_1^{-1} 
= \pa_t + i m(t) L + a_3 \pa_x + \tilde R_3. 
\end{equation}
We underline that the coefficient $m(t)$ is a function of time, independent of space.

\begin{lemm} \label{lemma:Psi}
There exists a universal constant $\delta_0 \in (0,1)$ such that if 
\[
\| c(t) - 1 \|_{L^\infty} < \delta_0 
\]
then $\| \pa_x \b_1(t) \|_{L^\infty} + \| \pa_x \tilde \b_1(t) \|_{L^\infty} < 1/2$ and 
\[
\| \pa_x \b_1(t) \|_{W^{\mu,\infty}} + \| \pa_x \tilde \b_1(t) \|_{W^{\mu,\infty}}  
\leq C_\mu \| c(t) -1 \|_{W^{\mu,\infty}} \quad \forall \mu \geq 0
\]
for some positive constant $C_\mu$ depending only on $\mu$.
As a consequence, $\Psi_1(t), \Psi_1(t)^{-1}$ are bounded transformations of $H^\mu(\T)$, with 
\[
\| \Psi_1(t) \|_{\mL(H^\mu)} + \| \Psi_1(t)^{-1} \|_{\mL(H^\mu)} 
\leq C_\mu (1 + \| c(t) - 1 \|_{H^\mu}) \quad \forall \mu \geq 0.
\]
Moreover $|m(t) - 1| \leq C \| c(t) - 1 \|_{H^1}$, 
\[
\| a_3(t) \|_{H^\mu} \leq C_\mu ( \| c(t) - 1 \|_{H^\mu} + \| \pa_t c(t) \|_{H^{\mu-1}} ) 
\quad \forall \mu \geq 1.
\]
The remainder $\tilde R_3(t)$ maps $L^2(\T)$ into itself, with
\[
\| \tilde R_3(t) \|_{\mL(L^2)} 
\leq C \big( \| c(t)-1 \|_{H^r} + \| V(t) \|_{L^\infty} + \| \pa_t c(t) \|_{L^\infty}
+ \| R_2(t) \|_{\mL(L^2)} \big),
\]
and, for all $\mu > 1/2$, $\tilde R_3(t)$ also maps $H^\mu(\T)$ into itself, with 
\[
\| \tilde R_3(t) \|_{\mL(H^\mu)} 
\leq C_\mu \big( \| c(t)-1 \|_{H^{\mu+r}} + \| V(t) \|_{H^\mu} + \| \pa_t c(t) \|_{H^\mu}
+ \| R_2(t) \|_{\mL(H^\mu)} \big),
\]
where $r > 0$ is a universal constant. 
\end{lemm}

\begin{proof} The estimates follow from the explicit formulas above, 
the usual estimates for the composition of functions (see, e.g., Appendix B in \cite{Baldi})
and Sobolev estimates for pseudo-differential operators (see \e{remainder estimate appendix}). 
The estimate of the pseudo-differential remainder term is the reason for which $r$ further space-derivatives are required on $c$. 
The term $\pa_t c$ appears only in $a_1$ and $r_1$. The term $V$ appears only in $a_3$ and $\tilde R_3$ where it is explicitly written, and nowhere else. 
The operator $R_2$ only appears in $\tilde R_3$ in the term $\Psi_1 R_2 \Psi_1^{-1}$. 
All the other terms depend only on $c$ and its space-derivatives.
\end{proof}

\subsection{Reparametrization of time}\label{S:AC2}
Now we want to replace the coefficient $m(t)$ in \eqref{conj tilde P 2} with a constant coefficient. 
We consider a diffeomorphism of the time interval
\[
\psi : [0,T] \to [0,T_1], \quad 
\psi(0) = 0, \quad \psi(T) = T_1, 
\quad \psi'(t) > 0,
\]
where $T_1 > 0$ has to be determined. 
We consider the pull-back $\psi_*$ defined as $(\psi_* h)(t,x) := h(\psi(t),x)$, and similar for its inverse $\psi^{-1}$.
Then we calculate the conjugate
\[
\psi_*^{-1} (\pa_t + i m(t) L) \psi_* 
= \psi'(\psi^{-1}(t)) \pa_t + i m(\psi^{-1}(t)) L.
\]
The two time-dependent coefficients are equal if 
$m(t) = \psi'(t)$ for all $t \in [0,T]$. 

We define 
\begin{equation} \label{cc 2.1}
\psi(t) := \int_0^t m(s) \, ds, \quad 
T_1 := \int_0^T m(t) \, dt, \quad 
\rho(t) := 
m(\psi^{-1}(t)).
\end{equation}
Since $|m-1|$ is small, then the ratio $T_1/T$ is close to 1, 
and also $\psi'(t)$ is close to 1 for all $t$. 
We have conjugate 
\begin{equation} \label{cc 2.2}
\psi_*^{-1} \wide P_1 \psi_* = \rho(t) \wide P_2, \quad 
\wide P_2 := \pa_t + i L + a_6 \pa_x + \tilde R_4,
\end{equation}
where 
\begin{equation} \label{cc 2.3}
a_6(t,x) := \frac{a_3(\psi^{-1}(t),x)}{\rho(t)}\,, \quad 
\tilde R_4 := \frac{1}{\rho(t)}\, \psi_*^{-1} \tilde R_3 \psi_*
\end{equation}
(and, more explicitly, $(\psi_*^{-1} \tilde R_3 \psi_*)(t) = \tilde R_3(\psi^{-1}(t))$). 
Now the coefficient of the highest order term $L$ is constant.

\subsection{Translation of the space variable}\label{S:AC3}
The goal of this section is to eliminate the space-average of the coefficient $a_6(t,x)$ in front of $\pa_x$. 
Consider a time-dependent change of the space variable which is simply a translation, 
\[
y = \ph(t,x) = x + p(t) \quad \Leftrightarrow \quad 
x = \ph^{-1}(t,y) = y - p(t),
\]
and its pull-back $(\ph_* h)(t,x) = h(t, \ph(t,x)) = h(t, x + p(t))$, and similarly for $\ph^{-1}$. 
Thus $\ph_*^{-1} \pa_t \ph_* = \pa_t + p'(t) \pa_x$, 
and $\ph_*$ commutes with every Fourier multiplier like $\pa_x, |D_x|^r, L$. 
We calculate the conjugate
\[
\wide P_3 := \ph_*^{-1} \wide P_2 \ph_* 
= \pa_t + i L + a_7 \pa_x + \tilde R_5, 
\]
where
\be \label{def a7}
a_7 := p'(t) + (\ph_*^{-1} a_6), \quad \ 
\tilde R_5 := \ph_*^{-1} \tilde R_4 \ph_* \,.
\ee
Since $\ph_*, \ph_*^{-1}$ preserve the space average, we fix 
\begin{equation}  \label{cc 3.1}
p(t) := - \frac{1}{2\pi} \int_0^t \int_\T a_6(s,x) \, dx \, ds.
\end{equation}
It follows that $\int_\T a_7(t,x) \, dx = 0$ for all $t \in [0,T_1]$.
Note that $\ph_*$ commutes with the multiplication operator $h \mapsto \rho(t) h$,
because $\rho(t)$ is independent of $x$. 
Moreover, by the change of time variable $s = \psi(t)$, $ds = m(t) dt$ in the integral, we get
\be \label{pT}
p(T_1) 
= - \frac{1}{2\pi} \int_0^{T_1} \int_\T a_6(s,x) \, dx ds
= - \frac{1}{2\pi} \int_0^T \int_\T a_3(t,x) \, dx dt.
\ee

\begin{proof}[Proof of Proposition \ref{P:38} concluded]
The composition $\Phi := \ph_*^{-1} \psi_*^{-1} \Psi_1$ 
of the previous three transformations conjugates $\wide P = \Phi^{-1} \rho \wide P_3 \Phi$. 
Also note that $\Phi^{-1} (\rho u) = m \Phi^{-1} u$ for all $u$. 
The transformation $\Psi_1$ is estimated in Lemma \ref{lemma:Psi}. 
The estimates for $\psi_*, \ph_*$ are straightforward.
Finally, rename $W := a_7$ and $R_3 := \tilde R_5$. 
\end{proof}

\noindent\emph{Notation.} In the following Proposition we use the shorter notation $\| u \|_{T, X}$ to denote 
the $C^0([0,T];X)$ norm of any $u$, with $X = L^2(\T), L^\infty(\T), H^\mu(\T), \mL(L^2(\T))$, etc.

\begin{prop}\label{P:38 App C}
Assume the hypotheses of Proposition \ref{P:38}. 

$(i)$ \emph{(Regularity).} 
In addition, let $\mu > 1/2$, let 
$\| c-1 \|_{T,H^\mu} \leq K < \infty$, and let 
\[
\mN_\mu := \| c-1 \|_{T,H^{\mu+r}} + \| V \|_{T,H^\mu} 
+ \| \pa_t c \|_{T,H^\mu} + \| R_2 \|_{T,\mL(H^\mu)} < \infty.
\]
Then $R_3$ maps $C^0([0,T_1] ; H^\mu(\T))$ into itself, with 
\be\label{p215 high}
\| R_3 \|_{T_1,\mL(H^\mu)} \leq C_{\mu,K} \mN_\mu 
\ee
for some constant $C_{\mu,K}$ depending on $\mu,K$. 
For $\mu \geq 1$, 
\be\label{p214 high}
\| W \|_{T_1,H^\mu} \leq C_\mu \big( \| c-1 \|_{T,H^\mu} 
+ \| \pa_t c \|_{T,H^{\mu-1}} + \| V \|_{T,H^\mu} \big)
\ee
and 
\[
\| \Phi u \|_{T_1,H^\mu} 
\leq C_\mu \| c \|_{T,H^\mu} \| u \|_{T,H^\mu},
\qquad 
\| \Phi^{-1} u \|_{T,H^\mu} 
\leq C_\mu \| c \|_{T,H^\mu} \| u \|_{T_1,H^\mu}
\]
for all $u = u(t,x)$, for some constant $C_\mu$ depending only on $\mu$. 

\smallskip

$(ii)$ \emph{(Stability).}
Consider another triple $(c',V',R_2')$ such that $c'$ also satisfies \e{basic smallness}, and $\mN_0 < \infty$ also for $(c',V',R_2')$. 
Let $\Phi', \Psi_1', \ph_*', \psi_*', T_1', W', R_3'$ be the corresponding objects obtained for the triple $(c',V',R_2')$. Then for all $u \in L^2(\T)$, all $t \in [0,T]$, 
\be \label{wwn 1}
\| \Psi_1(t) u - \Psi_1'(t) u \|_{L^2} 
+ \| \Psi_1(t)^{-1} u - \Psi'_1(t)^{-1} u \|_{L^2}
\leq C \| c(t) - c'(t) \|_{L^2} \| u \|_{H^1}\,.
\ee
Let $\lm := T_1 / T_1'$, and let $\mT$ be the time-rescaling operator $(\mT v)(t,x) := v(\lm t,x)$. 
Then for all $\mu \geq 0$, all $v = v(t,x)$, 
\begin{align}
\label{wwn 2}
\| \psi_* \ph_* v - \psi'_* \ph'_* (\mT v) \|_{T,H^\mu} 
& \leq C T \big( \| \pa_t v \|_{T_1,H^\mu} + \| v \|_{T_1,H^{\mu+1}} \big) \Delta_0 
\\
\label{wwn 2.1}
\| \ph_*'^{-1} \psi_*'^{-1} v - \mT (\ph_*^{-1} \psi_*^{-1} v) \|_{T_1',H^\mu} 
& \leq C T \big( \| \pa_t v \|_{T,H^\mu} + \| v \|_{T,H^{\mu+1}} \big) \Delta_0 
\end{align}
where 
\[
\Delta_0 := \| c - c' \|_{T,H^1} + \| \pa_t c - \pa_t c', V - V' \|_{T,L^2}.
\] 
Also, 
\be \label{wwn 3} 
|1 - \lm| + \| m - m' \|_{{C^0([0,T])}} \leq C \| c - c' \|_{T,L^\infty},
\ee
and, if 
\[
M(x) := \{ 1 + \pa_x \tilde \b_1(T, x - p(T_1)) \}^{\frac12}, \quad  
M'(x) := \{ 1 + \pa_x \tilde \b_1'(T, x - p'(T_1')) \}^{\frac12},
\] 
then 
\be \label{wwn 8} 
\| M - M' \|_{L^\infty(\T)} 
\leq C \big( \| c - c' \|_{T,H^2} + \| \pa_t c - \pa_t c', V - V' \|_{T,L^2} \big).
\ee 
For $\mu \geq 1$, if 
\be \label{wwn 4}
\| c-1 \|_{T,H^{\mu+1}} + \| \pa_t c \|_{T,H^\mu} + \| \pa_t^2 c \|_{T,H^{\mu-1}}  
+ \| V \|_{T,H^{\mu+1}} + \| \pa_t V \|_{T,H^\mu} \leq 1,
\ee
and \e{wwn 4} also holds for $c',V'$, then 
\be \label{wwn 5} 
\| W' - \mT W \|_{T_1',H^\mu} \leq C_\mu 
\big( \| c - c' \|_{T,H^\mu} + \| \pa_t c - \pa_t c' \|_{T,H^{\mu-1}} 
+ \| V - V' \|_{T,H^\mu} \big). 
\ee
Moreover, if 
\begin{align} \label{wwn 6}
& \| c-1 \|_{T,H^{r+1}} + \| \pa_t c \|_{T,H^{r+1}} + \| \pa_t^2 c \|_{T,L^2}  
+ \| V \|_{T,H^1} + \| \pa_t V \|_{T,L^2} 
\notag \\ & 
+ \| R_2 \|_{T,\mL(H^1) \cap \mL(L^2)} + \| \pa_t R_2 \|_{T,\mL(L^2)} 
\leq 1
\end{align}
and \e{wwn 6} also holds for $c',V',R_2'$, then 
\begin{multline} \label{wwn 7}
\| R_3' - (\mT R_3) \|_{T_1', \mL(L^2)} 
\leq C \big( \| c - c' \|_{T,H^{r+1}} + \| \pa_t c - \pa_t c' \|_{T,H^1} 
\\
+ \| V - V' \|_{T,H^1} + \| R_2 - R_2' \|_{T,\mL(L^2)} \big).
\end{multline} 
\end{prop}

\begin{proof}
To prove statement $(ii)$ we make repeatedly use of triangular inequality and explicit formulas. 
In particular, to estimate $p(\psi(\lm t)) - p'(\psi'(t))$, we use explicit formulas similar to \e{pT}. 
To estimate $\tilde R'_5 - (\mT \tilde R_5)$ we note that the rescaled operator 
$(\mT \tilde R_5)$ is the composition $\mT \tilde R_5 \mT^{-1}$, 
and then we also use \e{wwn 2}-\e{wwn 2.1}. 
Remember that we have renamed $W := a_7$ and $R_3 := \tilde R_5$. 
\end{proof}

\bigskip

\begin{flushleft}

\textbf{Thomas Alazard}\\
CNRS et D\'epartement de Math\'ematiques et Applications UMR 8553\\
\'Ecole Normale Sup\'erieure \\
45 rue d'Ulm, Paris F-75005, France

\medskip

\textbf{Pietro Baldi}\\
Dipartimento di Matematica e Applicazioni ``R. Caccioppoli''\\
Universit\`a di Napoli Federico II \\
Via Cintia, 80126 Napoli, Italy

\medskip 

\textbf{Daniel Han-Kwan}\\
CNRS et Centre de Math\'ematiques Laurent Schwartz UMR 7640\\
{\'E}cole Polytechnique\\
91128 Palaiseau Cedex - France 

\end{flushleft}

\end{document}